\documentclass{amsbook}
\usepackage{amssymb, enumitem}
\usepackage{aliascnt, graphicx,hyperref,xcolor}
\usepackage{mathtools}
\usepackage[all]{xy}
\usepackage{tikz}
\usepackage{tikz-cd}
\usetikzlibrary{arrows.meta, positioning}
\usepackage{imakeidx}
\usetikzlibrary{arrows.meta, decorations.pathreplacing, calc}
\usepackage{standalone}
\usetikzlibrary{shapes.geometric}
\usepackage{float}

\numberwithin{section}{chapter}
\numberwithin{equation}{chapter}

\newtheorem{lma}{Lemma}[section]

\newaliascnt{thmCt}{lma}
\newtheorem{thm}[thmCt]{Theorem}
\aliascntresetthe{thmCt}

\newaliascnt{corCt}{lma}
\newtheorem{cor}[corCt]{Corollary}
\aliascntresetthe{corCt}

\newaliascnt{prpCt}{lma}
\newtheorem{prp}[prpCt]{Proposition}
\aliascntresetthe{prpCt}

\newtheorem*{thm*}{Theorem}
\newtheorem*{cor*}{Corollary}
\newtheorem*{prp*}{Proposition}

\newaliascnt{thmChapterCt}{lmaChapter}

\aliascntresetthe{thmChapterCt}

\newaliascnt{prpChapterCt}{lmaChapter}

\aliascntresetthe{prpChapterCt}

\theoremstyle{definition}

\newaliascnt{pgrCt}{lma}
\newtheorem{pgr}[pgrCt]{}
\aliascntresetthe{pgrCt}

\newaliascnt{dfnCt}{lma}
\newtheorem{dfn}[dfnCt]{Definition}
\aliascntresetthe{dfnCt}

\newaliascnt{ntnCt}{lma}
\newtheorem{ntn}[ntnCt]{Notation}
\aliascntresetthe{ntnCt}

\newaliascnt{rmkCt}{lma}
\newtheorem{rmk}[rmkCt]{Remark}
\aliascntresetthe{rmkCt}

\newaliascnt{rmksCt}{lma}

\aliascntresetthe{rmksCt}

\newaliascnt{prblCt}{lma}
\newtheorem{prbl}[prblCt]{Problem}
\aliascntresetthe{prblCt}

\newaliascnt{exaCt}{lma}
\newtheorem{exa}[exaCt]{Example}
\aliascntresetthe{exaCt}

\newaliascnt{exasCt}{lma}
\newtheorem{exas}[exasCt]{Examples}
\aliascntresetthe{exasCt}

\newaliascnt{conjCt}{lma}

\aliascntresetthe{conjCt}

\newaliascnt{pgrChapterCt}{lmaChapter}

\aliascntresetthe{pgrChapterCt}

\newaliascnt{dfnChapterCt}{lmaChapter}

\aliascntresetthe{dfnChapterCt}

\newaliascnt{rmksChapterCt}{lmaChapter}

\aliascntresetthe{rmksChapterCt}

\newaliascnt{prblChapterCt}{lmaChapter}

\aliascntresetthe{prblChapterCt}

\newaliascnt{exaChapterCt}{lmaChapter}

\aliascntresetthe{exaChapterCt}

\newaliascnt{exasChapterCt}{lmaChapter}

\aliascntresetthe{exasChapterCt}

\newcounter{theoremintro}

\newaliascnt{thmIntroCt}{theoremintro}
\newtheorem{thmIntro}[thmIntroCt]{Theorem}
\aliascntresetthe{thmIntroCt}

\newaliascnt{dfnIntroCt}{theoremintro}

\aliascntresetthe{dfnIntroCt}

\newaliascnt{prpIntroCt}{theoremintro}

\aliascntresetthe{prpIntroCt}

\newaliascnt{corIntroCt}{theoremintro}
\newtheorem{corIntro}[corIntroCt]{Corollary}
\aliascntresetthe{corIntroCt}

\newaliascnt{qstIntroCt}{theoremintro}
\newtheorem{qstIntro}[qstIntroCt]{Question}
\aliascntresetthe{qstIntroCt}

\newcommand{\Ifree}{I_{\mathrm{free}}}
\newcommand{\Ireg}{I_{\mathrm{reg}}}

\newcommand{\Pfree}{\mathbb{P} _{\mathrm{free}}}
\newcommand{\Preg}{\mathbb{P}_{\mathrm{reg}}}
\newcommand{\Pmin}{\mathbb{P}_{\mathrm{min}}}
\newcommand{\Pnmfree}{\mathbb{P}_{\mathrm{nmfree}}}
\newcommand{\Pmfree}{\mathbb{P}_{\mathrm{mfree}}}

\DeclareMathOperator{\rL}{L}

\newcommand{\andSep}{\,\,\,\text{ and }\,\,\,}

\def\today{\number\day\space\ifcase\month\or   January\or February\or
   March\or April\or May\or June\or   July\or August\or September\or
   October\or November\or December\fi\   \number\year}

\newcommand{\Z}{{\mathbb{Z}}}
\newcommand{\N}{{\mathbb{N}}}
\newcommand{\ZZ}{{\mathbb{Z}}}
\newcommand{\NN}{{\mathbb{N}}}
\newcommand{\QQ}{{\mathbb{Q}}}
\newcommand{\id}{{\mathrm{id}}}

\newcommand{\CSP}{{\mathrm{CSP}}}

\newcommand{\ad}{{\mathrm{ad}}}

\newcommand{\rr}{{\bf r}}

\newcommand{\ol}{\overline}

\newcommand{\Sgen}{S^{\mathrm{gen}}}
\newcommand{\Xgen}{X^{\mathrm{gen}}}

\newcommand{\ca}{$\mathrm{C}^*$-algebra}

\newcommand{\mcan}{\mathfrak m_{\mathrm{can}}}

\newcommand{\Grot}[1]{G(#1)}

\makeindex[name=terms, title=Index of Terms]
\makeindex[name=symbols, title=Index of Symbols]

\author[P.~Ara, J.~Bosa, L.~Cantier, E.~Pardo, and E.~Vilalta]{Pere Ara\and
        Joan Bosa\and
        Laurent Cantier\and
        Enrique Pardo\and 
        Eduard Vilalta}

\address{Pere Ara,
Department de Matemàtiques, Universitat Aut\`onoma de Barcelona, 08193 Bellaterra (Barcelona), Spain}
\email[]{pere.ara@uab.cat}

\address{Joan Bosa,
Departamento de Matematicas, Universidad de Zaragoza, C/Pedro Cerbuna 12, 50009 Zaragoza, Zaragoza, Spain}
\email[]{jbosa@unizar.es}
\urladdr{https://personal.unizar.es/jbosa/}

\address{Laurent Cantier, 
Departamento de Matematicas, Universidad de Zaragoza, C/Pedro Cerbuna 12, 50009 Zaragoza, Spain}
\email[]{lncantier@gmail.com}
\urladdr{https://laurentcantier.fr/}

\address{Enrique Pardo,
Departamento de Matematicas, Facultad de Ciencias, Universidad de Cádiz, Campus de Puerto Real, 11510 Puerto Real (C\'adiz), Spain}
\email[]{enrique.pardo@uca.es}
\urladdr{https://sites.google.com/gm.uca.es/webofenriquepardo/inicio}

\address{Eduard Vilalta,
Department de Matemàtiques, Universitat Politècnica de Catalunya - BarcelonaTech (UPC), Diagonal 647, 08028 Barcelona}
\email[]{eduard.vilalta@upc.edu}
\urladdr{www.eduardvilalta.com}

\subjclass[2020]%
{
Primary
16D70;
Secondary 
16E50, 
19K14,   
46L05,   
46L80.  
}

\title{Super adaptable graphs and the realization problem}

\begin{document}

\begin{abstract}
We introduce and study the family of super adaptable monoids, which includes, among many others, all conical refinement monoids that are either finitely generated, or countable, primely generated, and regular. Our main result shows that any countable super adaptable monoid can be realized by a von Neumann regular ring, thus answering affirmatively the Realization problem for this class. The result is new even for the subclass of countable, primely generated, regular conical refinement monoids, which we show can be realized by both a von Neumann regular ring and a purely infinite, real rank zero $\mathrm{C}^*$-algebra.

To prove this result, we develop a systematic framework that connects three fundamental structures: abstract monoids, separated graphs (a specific class of colored graphs), and von Neumann regular rings. By establishing functorial connections among these classes, we translate several algebraic properties of monoids into combinatorial and ring-theoretic notions. This allows us not only to prove the main realization result, but also to obtain several characterizations of super adaptable monoids. As a particular application, these can be used to provide a complete characterization of which regular monoids arise as graph monoids. 

To our knowledge, this work offers the first self-contained treatment in the literature where the complete realization process ---from abstract monoids through I-systems and separated graphs to regular rings and $\mathrm{C}^*$-algebras--- is developed sequentially and in full detail within a single text.
\end{abstract}

\maketitle

\tableofcontents

\chapter{Introduction}

One of the longest-standing open problems in the theory of von Neumann regular rings is the \emph{realization problem}, which asks exactly which monoids $M$ can be realized as the monoid $V(R)$ of isomorphism classes of finitely generated projective modules over a regular ring $R$. Classical results establish that $V(R)$ is necessarily conical and satisfies the Riesz refinement property whenever $R$ is regular \cite[Theorem~2.8]{Goo91vNaReg}. Thus, the problem asks:

\begin{qstIntro}\label{qst:RealProb}Let $M$ be a conical refinement monoid. Is there a regular ring $R$ such that $V(R)\cong M$?
\end{qstIntro}

Originally motivated by its analogue on hereditary algebras \cite{Ber74CoProd,BerDic78UniDer}, \autoref{qst:RealProb} was posed over three decades ago by Goodearl \cite{Good95-direct-sums}. The question is deeply connected to other significant problems in non-stable $K$-theory, with notable examples including the \emph{separativity problem} \cite{AGOP98}, the study of regularity properties of \emph{real rank zero $\mathrm{C}^*$-algebras} (see \autoref{subsec:RR0} below for a discussion), the \emph{Atiyah problem} \cite{AraGoo17RealProbAtiy}, or the \emph{Congruence Lattice Problem} \cite{Weh04Semi,Weh07SolDil} (see also \cite[Section~3]{Ara09-realization}).

In this paper, we provide a positive answer to \autoref{qst:RealProb} for the class of countable \emph{super adaptable monoids}, a family that both subsumes most known positive solutions to date and significantly expands the class of monoids that can be realized. Throughout the text, we say that a monoid $M$ is \emph{realizable by a regular ring} if there exists such a ring $R$ for which $V(R)\cong M$.

Before formally defining super adaptability and detailing our main results, we briefly review the historical context and known literature surrounding this problem. The reader is referred to \cite{Ara09-realization} for a survey on the realization problem from 2009, and to \autoref{sect:Preliminaries} for a brief introduction to all the relevant notions. We take this opportunity to review recent results that did not appear in \cite{Ara09-realization} but which are fundamental for this work. 

\section{Background}

The class of conical refinement monoids naturally splits into two subclasses: those that can be expressed as a direct limit of finitely generated conical refinement monoids, and those that cannot. Following \cite{AraGoo15SgpFor}, the former are called \emph{tame}, while the latter are called \emph{wild}. As shown in \cite[Theorem~2.3]{AraGoo15SgpFor}, tame monoids are far better behaved than wild ones, and their additional structural properties make them amenable to study via combinatorial and graph-theoretic tools. For example, among the first classes of tame monoids studied in the literature were countable direct limits of monoids of the form $(\mathbb{Z}^+)^n$, which admit a combinatorial description through their Bratteli diagrams \cite{Bra72IndLim}. These were shown to be realizable by Elliott \cite{Ell76Real} and were later characterized abstractly in \cite{EHS1980,Grillet1976} as the countable dimension monoids; see Sections \ref{sec:RefMon} and \ref{sec:VR} for details. A restriction on the size of the monoids is needed, as Wehrung showed in \cite{Weh98NonMes} that there exist dimension monoids of cardinality at least $\aleph_2$ that cannot be realized by any von Neumann regular ring. Because of this, \autoref{qst:RealProb} is often posed for countable conical refinement monoids, for which no counterexample is known.

In contrast to tame monoids, wild monoids can be far less well-behaved. Although it is known that some wild monoids are realizable \cite{AraGoo17RealProbAtiy}, the general lack of additional properties such as separativity or unperforation places most classes of these monoids outside the scope of current techniques (the definitions of these properties can be found in \autoref{sec:RefMon}). Consequently, research has largely shifted toward countable tame monoids, leading to several positive partial solutions for specific subfamilies.


Two of the largest known families of realizable tame monoids are:
\begin{itemize}
\item[(a)] finitely generated conical refinement monoids \cite[Theorem~B]{AraBosPar20Sel}, a result which built upon the previous works \cite{AraBru02Quiver,AraMorPar07NonStab,AraGoo12Crelle};
\item[(b)] monoids of the form $G\sqcup\{ 0\}$ where $G$ is a countable abelian group \cite{AraGooPar02K0} (in analogy to R\o{}rdam's \cite{Ror95Class}, see \autoref{subsec:RR0} below).
\end{itemize}

The families (a) and (b) belong to the broader category of \emph{primely generated monoids} (see the discussion in \autoref{sect:realizing-larger-families}). Here, recall that an element $p$ in a monoid $M$ is \emph{prime} if it is not invertible and, whenever $p\leq a+b$ for some $a,b\in M$, one has $p\leq a$ or $p\leq b$. A monoid is then \emph{primely generated} if every non-invertible element can be expressed as a sum of prime elements. As shown in \cite[Theorem~4.6]{AraPar16Isr}, all primely generated conical refinement monoids are tame, making them a natural class to investigate within the context of the realization problem. Additionally, the class (b) consists of \emph{regular monoids}, that is, $2x\leq x$ for every element $x$.

Despite the aforementioned successes, tame monoids that fall outside the families (a)-(b) are abundant. These include regular monoids that are not of the type (b), or the non-regular monoid
\begin{equation}\label{eq:ExaMon}
    \langle 
a,b_1,b_2,\cdots \mid 
a = a + b_i \text{ and }
b_i = b_{i+1} + b_1
\rangle ,
\end{equation}
which, since it is not finitely generated, cannot be realized with previous techniques (\autoref{exa:SupAdapt}).

\section{Main results and strategy} In this work, we introduce and study the family of \emph{super adaptable monoids}. This subclass of primely generated conical refinement monoids simultaneously captures all such monoids that are regular, all those that are finitely generated, and examples such as \eqref{eq:ExaMon}. Our main result shows that any countable super adaptable monoid is realizable by a regular ring and, in fact, by a regular $K$-algebra for any field $K$.

\begin{thmIntro}[{\ref{thm:Main}}]\label{thmIntro:Main}
Let $M$ be a countable super adaptable monoid, and let $K$ be any field. Then, there exists a von Neumann regular $K$-algebra $R$ such that $V(R)\cong M$.
\end{thmIntro}

\begin{corIntro}[\ref{cor:RegRealizable}]
    Every countable primely generated regular conical refinement monoid can be realized by a regular ring.
\end{corIntro}

While super adaptability of a monoid is formally defined in terms of its asso-
ciated $I$-system (see \hyperlink{Step1}{Step (1)} below), we provide  an alternative, intrinsic characterization of this notion in \autoref{thmIntro:IntChar}. It is this characterization that the reader should keep in mind throughout the manuscript, albeit we do not take it as the definition of super adaptability as it would make our proofs significantly more demanding. Here, a prime element $p \in M$ is said to be \textit{free} if $mp \leq np$ implies $m \leq n$, and a monoid morphism $M \to N$ is termed \textit{free-injective} if, upon passing to the induced map on the antisymmetrizations $\overline{M} \to \overline{N}$, it maps primes to primes, non-minimal free primes to non-minimal free primes, and ensures that each non-minimal free prime of $\ol{N}$ has a fiber of cardinality at most $1$; see \autoref{dfn:free-injective}. 

More concretely, \autoref{thmIntro:IntChar} states that countable super adaptable monoids are those tame monoids whose connecting maps are free-injective:

\begin{thmIntro}[cf. \ref{thm:charac-super-adaptable}]\label{thmIntro:IntChar}
    Let $M$ be a countable conical refinement monoid. The following conditions are equivalent:
    \begin{enumerate}
    \item $M$ is super adaptable.
    \item $M\cong \varinjlim_{n\in \NN} (M_n,\gamma_n)$ for a directed system $(M_n,\gamma_n)$, where each $M_n$ is a finitely generated conical refinement monoid, and each $\gamma_n\colon M_n\to M_{n+1}$ is a free-injective monoid homomorphism.   
\end{enumerate}
\end{thmIntro}

The primary technical novelties of this paper, which allow us to generalize the methods of \cite{AraBosPar20Sel} to prove \autoref{thmIntro:Main}, are the introduction of three notions: 
\begin{itemize}
    \item[(a)] \emph{marked monoids} and their associated $I$-systems (\autoref{dfn:superadaptableIsystem}),
    \item[(b)] \emph{super adaptable graphs} (\autoref{dfn:SupAdSepGra}), and
    \item[(c)] the use of \emph{Szyma\'{n}sky models for simple regular monoids} (\autoref{pgr:DfnGJ}).
\end{itemize}

To properly introduce these concepts and clarify their role in the proof, we provide a general overview of our strategy below. Although our approach loosely follows the blueprint of \cite[Theorem~B]{AraBosPar20Sel}, these refined steps require significantly more work and the unveiling of new technical results. The following diagram provides a visual representation of our approach, which breaks down the proof into four phases.
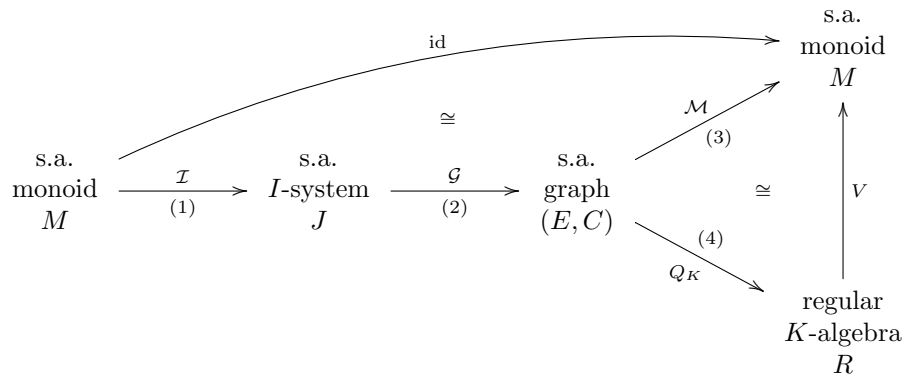
\[
    \xymatrixrowsep{1pc}
    \xymatrixcolsep{4pc}
        \xymatrix{
        & & & 
        {\begin{array}{c}
                 \text{s.a.}  \\
                  \text{monoid}\\
                  M
            \end{array}}\\
            {\begin{array}{c}
                 \text{s.a.}  \\
                  \text{monoid}\\
                  M
            \end{array}}
            \ar[r]^-{\mathcal{I}}_-{(1)}
            \ar@/^2.0pc/[rrru]^-{\id}& 
            {\begin{array}{c}
                 \text{s.a.}  \\
                  I\text{-system}\\
                  J
            \end{array}}
            \ar[r]^-{\mathcal{G}}_-{(2)}& 
            {\begin{array}{c}
                 \text{s.a.}  \\
                  \text{graph}\\
                  (E,C)
            \end{array}}
            \ar@{}[r] |(0.7) {\cong}
            \ar@{}[ul] |(0.5) {\cong}
            \ar[ru]^-{\mathcal{M}}_-{(3)}
            \ar[rd]_-{Q_K}^-{(4)} 
            &
             \\
            & & & 
            {\begin{array}{c}
                 \text{regular}  \\
                  K\text{-algebra}\\
                  R
            \end{array}}
            \ar[uu]_-{V}
        }
    \]

\noindent\hypertarget{Step1}{\textbf{Step (1).}} From any primely generated conical refinement monoid $M$ one can build an \emph{$I$-system}, a combinatorial construction that codifies the structure of $M$ using its prime elements. More precisely, consider the equivalence relation $\equiv$ on the set of prime elements of $M$ given by $x\equiv y$ if and only if $x\leq y$ and $y\leq x$. A \emph{mark} $\mathfrak m$ on $M$ is a choice of a prime element in each equivalence class, and a \emph{marked monoid} is a pair $(M,\mathfrak m)$. To each marked monoid one can associate an $I$-system $\mathcal{I}(M,\mathfrak{m})$, a construction from \cite{AraPar16Isr} that we recall in \autoref{pgr:IM}. In previous works, such as \cite{AraPar16Isr} or \cite{AraBosPar20Sel}, a mark was (implicitly) chosen at random in the beginning of the paper and the results that were proved did not depend on that choice. However, in our current generalized setting, we need to lift several monoid morphisms to maps between $I$-systems, which forces us to keep track of where elements in their marks are mapped.

In \autoref{sec:MondIsystm}, we revisit the results from \cite{AraPar16Isr} through the lens of marked monoids. In addition to emphasizing the role of the mark in the construction of $\mathcal{I}(M,\mathfrak{m})$, we also recall that any $I$-system $J$ gives rise to a primely generated conical refinement monoid $\mathcal{M}(J)$, which is now equipped with a canonical mark $\mathfrak{m}_J$. As shown in \cite[Theorem~0.1]{AraPar16Isr} (in the language of marks; see \autoref{thm:APIsr16}), these assignments are inverses of each other, in so far that
\begin{equation}\label{eq:AraParIsr}
    (\mathcal{M}(\mathcal{I}(M,\mathfrak{m})),\mathfrak{m}_{\mathcal{I}(M,\mathfrak{m})})\cong (M,\mathfrak{m})
    \quad\text{and}\quad 
    \mathcal{I}(\mathcal{M}(J),\mathfrak{m}_J)\cong J
\end{equation}
for any $(M,\mathfrak{m})$ and $J$.

Also in \autoref{sec:MondIsystm}, we define the class of \emph{super adaptable $I$-systems} and say that a monoid $M$ is \emph{super adaptable} if it has a mark $\mathfrak{m}$ such that $\mathcal{I}(M,\mathfrak{m})$ is super adaptable (\autoref{dfn:superadaptableIsystem}). This class of systems and monoids will be henceforth called s.a.~$I$-systems and s.a.~monoids, and will be our objects of study throughout. After providing several examples (\ref{exa:ExasSupAdaptMon}-\ref{exam:not-super-adaptable}), we end the chapter with the lifting results \autoref{prp:morphism-of-systems} and \autoref{cor:morphisms-inducing-plain-homos}, showing that several monoid morphisms can be realized as morphisms between $I$-systems. As stated, this is the primary reason to introduce marks in the first place.\vspace{0.2cm}

\noindent\hypertarget{Step2}{\textbf{Step (2).}} As defined in \cite[Definition~2.1]{AraGoo12Crelle}, a separated graph $(E,C)$ is a pair consisting of a graph $E$ and a particular coloring $C$ of its edges; see \autoref{dfn:SepGraph}. In \autoref{sec:FromIsystmToSepGraph}, we introduce \emph{super adaptable (separated) graphs}, a generalization of the notion of an adaptable graph from \cite{AraBosPar20Sel}. Crucially, we show that this class of graphs is closed under direct limits, and that any s.a.~graph is the limit of finite s.a.~graphs (\autoref{prp:SupAdapLim}).
    
Further, we associate to any s.a.~graph $(E,C)$ an $I$-system $\mathcal{I}(E,C)$ and, conversely, we construct a separated graph $\mathcal{G}(J)$ for every s.a.~$I$-system $J$. The main results of this chapter (Theorems \ref{prp:constructed-super adaptable} and \ref{thm:necessary-condition}) state that these constructions produce super adaptable systems and graphs respectively. In particular, given any s.a.~monoid $M$, the graph $\mathcal{G}(\mathcal{I}(M,\mathfrak{m}))$ is super adaptable.

The construction of $\mathcal{G}(\mathcal{I}(M,\mathfrak{m}))$ differs significantly from those in \cite[Proposition~5.13]{AraPar17JAlg} and \cite[Theorem~2.11]{AraBosPar20SgpFor}. Notably, several components are now built using what we call \emph{Szyma\'{n}sky models} (\autoref{pgr:DfnGJ}). These are strongly connected row-finite graphs whose $\mathrm{C}^*$-algebra has a  Murray-von Neumann semigroup of the form $G\sqcup\{ 0\}$ with $G$ an abelian group. By the results in \cite{Szy02Ind}, the converse is true: Any such monoid can be realized by a graph of this type. We use this extensively in \hyperlink{Step3}{Step (3)}.\vspace{0.2cm}

\noindent\hypertarget{Step3}{\textbf{Step (3).}} For a given separated graph (in particular, a super adaptable graph) $(E,C)$ one can construct its associated monoid $\mathcal{M}(E,C)$; see \autoref{dfn:monoidsepgraph}. In \autoref{thm:MondToGraph}, we show that
\begin{equation}\label{eq:MeqMGIM}
    M\cong \mathcal{M}(\mathcal{G}(\mathcal{I}(M,\mathfrak{m})))
\end{equation}
for any s.a.~monoid $M$, where $\mathfrak{m}$ is a mark for which $(M,\mathfrak{m})$ is super adaptable.
    
In short, the realization problem for super adaptable monoids reduces to realizing monoids of the form $\mathcal{M}(E,C)$ associated to super adaptable graphs.

We prove \eqref{eq:MeqMGIM} in two phases: In \autoref{prp:isoIsysMon}, we show that the $I$-systems associated to $(E,C)$ and $(\mathcal{M}(E,C),\mcan )$, respectively, agree for any super adaptable graph $(E,C)$. Then, in \autoref{lma:SuperAdpMonoid} we prove another isomorphism of $I$-systems, namely $J\cong \mathcal{I}(\mathcal{G}(J))$ for any super adaptable system $J$. In particular, $\mathcal{I}(M,\mathfrak{m})\cong \mathcal{I}(\mathcal{G}(\mathcal{I}(M,\mathfrak{m})))$. Combined with \eqref{eq:AraParIsr}, this leads to \eqref{eq:MeqMGIM}; see \autoref{thm:MondToGraph} for details.\vspace{0.2cm}

\noindent\hypertarget{Step4}{\textbf{Step (4).}} Let $K$ be any field. Building upon the construction of \cite[Section~2]{AraBosPar20Sel}, in \autoref{sec:RealMon} we introduce a ring $Q_K(E,C,\sigma)$ for any countable s.a. graph $(E,C)$. Here, $\sigma$ is some technical auxiliary data needed to perform our computations which did not appear in \cite{AraBosPar20Sel}; see \autoref{subsect:the-construction} for details. The goal of \autoref{sec:RealMon} is to prove \autoref{thm:realization-general-case}, showing that $Q_K (E,C,\sigma)$ is a regular ring and that 
\begin{equation}\label{eq:VQeqM}
    V(Q_K (E,C,\sigma))\cong \mathcal{M}(E,C).
\end{equation}
Paired with \eqref{eq:MeqMGIM}, this then leads to \autoref{thmIntro:Main}, since we have
\[
    M\overset{\eqref{eq:MeqMGIM}}{\cong } \mathcal{M}(\mathcal{G}(\mathcal{I}(M,\mathfrak{m})))
    \overset{\eqref{eq:VQeqM}}{\cong }
    V(Q_K (\mathcal{G}(\mathcal{I}(M,\mathfrak{m})),\sigma)).
\]
    
The proof of the isomorphism \eqref{eq:VQeqM} breaks down into three steps. First, whenever $(E,C)$ is a finite s.a.~graph, we show in \autoref{thm:realization-finite-case-with-condition-F} that $Q_K (E,C,\sigma)$ is regular and that \eqref{eq:VQeqM} holds. Using the results from Step (2), we then prove in \autoref{thm:QK-is-an-injective-direct-limit} that for any ---possibly infinite--- countable s.a.~graph $(E,C)$ the algebra $Q_K (E,C,\sigma)$ can be written as a direct limit with blocks of the form $Q_K (E_n,C^n,\sigma_n)$ with $(E_n,C^n)$ finite s.a.~graphs such that $(E,C)\cong\lim_n (E_n,C^n)$. Since both $\mathcal{M}(\cdot)$ and $V(\cdot)$ are continuous, we obtain 
\[
V(Q_K (E,C,\sigma))\cong
\varinjlim V(Q_K (E_n,C^n,\sigma_n))\cong 
\varinjlim \mathcal{M}(E_n,C^n)\cong 
\mathcal{M}(E,C)
,
\]
which finishes the proof of \eqref{eq:VQeqM} (i.e. \autoref{thm:realization-general-case}) and, consequently, of \autoref{thmIntro:Main}.

\section{\texorpdfstring{Real rank zero $\mathrm{C}^*$-algebras}{Real rank zero C*-algebras}}\label{subsec:RR0}

We conclude this introduction by detailing the connections of \autoref{qst:RealProb} to real rank zero $\mathrm{C}^*$-algebras mentioned earlier.

Recall that a unital $\mathrm{C}^*$-algebra is of \emph{real rank zero} \cite{BroPed91RR0} if the set of self-adjoint invertible elements is dense within the set of all self-adjoint elements. While these algebras are not generally von Neumann regular, they share significant structural similarities with regular rings (for example, they are both exchange rings). This parallel has produced a deeply fruitful exchange: many problems and results concerning real rank zero \ca{s} have direct analogues in regular rings, and vice versa.

For example, using techniques from the theory of regular rings, Perera and Rørdam showed in \cite[Theorem~5.8]{PerRor04AFEmb} that all separable, non-type I $\mathrm{C}^*$-algebras of real rank zero are weakly divisible. Conversely, the realization of monoids of the form $G \sqcup \{0\}$ \cite{AraGooPar02K0} draws inspiration from the operator-algebraic methods in \cite{Ror95Class}. The following question, known as the \emph{realization problem for real rank zero $\mathrm{C}^*$-algebras}, is the direct analogue of \autoref{qst:RealProb}:

\begin{qstIntro}\label{qstIntro:RealRR0}
    Let $M$ be a conical refinement monoid. Does there exist a real rank zero $\mathrm{C}^*$-algebra $A$ such that $V(A) \cong M$?
\end{qstIntro}

Other than its intrinsic interest, \autoref{qstIntro:RealRR0} is linked to the famous open problem of whether all simple real rank zero $\mathrm{C}^*$-algebras possess strict comparison; a question that dates back to R\o{}rdam's counterexamples of Elliott's classification program \cite{Ror03Simple,Ror05RRCertainSimp}. In the nuclear setting, the problem is often posed as the stronger question: are all simple separable nuclear non-elementary $\mathrm{C}^*$-algebras of real rank zero automatically $\mathcal{Z}$-stable? \cite[Problem~XXX]{SchTikWhi99arX:Prob}.

The connection between these problems and \autoref{qstIntro:RealRR0} is as follows: If \autoref{qstIntro:RealRR0} can be answered affirmatively for all simple conical refinement monoids, then there are simple real rank zero $\mathrm{C}^*$-algebras without strict comparison. Indeed, a real rank zero $\mathrm{C}^*$-algebra $A$ has strict comparison if and only if $V(A)$ is almost unperforated, and it is a consequence of Wehrung's \cite[Corollary~2.7]{Weh98Emb} (concretely, see \cite[Corollary~4.27]{Ara26:Intro} and also the argument in \autoref{exa:NotSep}) that there exist simple conical refinement monoids which fail to have this property.

Another important question in this direction, both for regular rings and real rank zero \ca{s}, is the \emph{separativity problem}. Here, recall that a monoid $M$ is \emph{separative} if $x=y$ whenever $x+x=x+y=y+y$. The problem asks:
\begin{qstIntro}\label{qstIntro:Sep}
    Let $R$ be a von Neumann regular ring (resp. a real rank zero \ca{}). Is $V(R)$ separative?
\end{qstIntro}

Once again, a positive answer to \autoref{qstIntro:RealRR0} implies a negative one for \autoref{qstIntro:Sep}, as there exist conical refinement monoids that are not separative; see, for example, \cite[Page~3]{Ara09-realization}. On the other hand, a positive answer to \autoref{qstIntro:Sep} for simple \ca{s} of real rank zero would solve in the affirmative R\o{}rdam's dichotomy problem \cite[Problems~XXIX,~LX]{SchTikWhi99arX:Prob}.

As previously stated, all monoids realized in this paper are tame and hence unperforated and separative \cite[Theorem~2.3]{AraGoo15SgpFor}. In particular, an adaptation of our proof to a \ca{ic} setting would not provide counter examples to \cite[Problems~XXIX,~XXX]{SchTikWhi99arX:Prob}. However, it is still an interesting question whether the proof of \autoref{thmIntro:Main} can be adapted to an operator-algebraic framework. Steps (1), (2) and (3) are exactly the same, with all difficulties lying in Step (4). An adaptation of this step would require developing $\mathrm{C}^*$-algebraic tools analogous to those in \autoref{sec:RealMon} to transform a \ca{} with prescribed Murray-von Neumann semigroup (built out of a s.a.~graph, akin to those in \cite[Section~3]{Ara26:Intro}) into a real rank zero \ca{}.

In this context, we can prove a realization result for those countable primely generated conical refinement monoids that are regular:
\begin{thmIntro}[{\ref{thm:charac-regular-primely-generated}}]
Let $M$ be a conical regular refinement monoid. Then the following conditions are equivalent:
\begin{itemize}
\item[(i)] $M$ is a countable primely generated monoid. 
\item[(ii)] $M\cong \mathcal M (E)$ where $E$ is a countable super adaptable graph with $\Pfree (E)= \emptyset$.
\item[(iii)] $M\cong \mathcal M (E)$ for a row-finite countable graph $E$.
\end{itemize}
Moreover, in this case $M$ is representable by a von Neumann regular ring and by a purely infinite real rank zero graph $\mathrm{C}^*$-algebra. 
\end{thmIntro}

We conclude the paper by presenting several open problems stemming from our results in \autoref{chap:FROQ}.

\section*{Acknowledgments} All authors were partially supported by MINECO (grant No.\ PID2023-147110NB-I00). PA and EV were also partially supported by the Comissionat per Universitats i Recerca de la Generalitat de Catalunya (grant No. 2021-SGR-01015). EP was also partially supported by PAIDI grants FQM-298, and ProyExcel 00780 ``Operator Theory: an interdisciplinary approach'' (2021), of the Junta de Andaluc\'{\i}a. JB was also partially supported by MINECO (grant PID2024-155800NB-C32). JB and LC were partially supported by the Spanish State Research Agency through Consolidacion Investigadora program No. CNS2022-135340. EV was also partially supported by PID2025-168310NB-I00, financed by MICIU/AEI/10.13039/501100011033 and the FSE+.

\section*{AI statement}

Generative AI was not used in any way to prove or discover any of the results in this paper. Gemini was used for English language editing and grammatical corrections of some parts of the paper during the final stages of writing.

\chapter{Preliminaries}\label{sect:Preliminaries}

In this chapter we collect some basic definitions and facts that are needed throughout the paper. More tailored preliminaries are deferred to the beginning of their respective chapters.

Some of the results listed below will not be used in our proofs. They are included primarily to introduce the unfamiliar reader to the main notions and techniques surrounding the realization problem.

\section{von Neumann regular rings}\label{subsec:RingsandAlgebras}

 The concept of a (von Neumann) regular ring was originally introduced by von Neumann \cite{vNa36RingOp} in relation to his study of operator algebras and continuous geometries. Although the concept allows for many equivalent characterizations (see \autoref{prp:CharacvNa}), the most commonly used definition is as follows:
 
 \begin{dfn}\label{dfn:vNr}
     A ring $R$ is called \textit{von Neumann regular}\index[terms]{ring!von Neumann regular} if for every $x \in R$ there exists $y \in R$ such that $x = xyx$.
 \end{dfn}
 
 We provide below a list of examples and characterizations of regular rings. All proofs, as well as a comprehensive treatment of such rings can be found in \cite{Goo91vNaReg}.

\begin{exas}\label{exas:vNReg} $ $
    \begin{itemize}
        \item[(a)] Any field $K$ and, more generally, any matrix algebra $M_n(K)$ over $K$ is von Neumann regular. In fact, any algebra of the form $M_n(R)$ with $R$ von Neumann regular is again von Neumann regular \cite[Theorem~24]{Kap69FieRing}.
        \item[(b)] Recall that a ring $R$ is called \emph{unit regular}\index[terms]{ring!unit regular} if for every $x \in R$ there exists a \textbf{unit} $y \in R$ such that $x = xyx$. By definition, all unit regular rings are von Neumann regular.
        \item[(c)] The endomorphism ring of any vector space is von Neumann regular \cite[Page~111,~Example~4]{Kap69FieRing}.
        \item[(d)] For a given finite von Neumann algebra $\mathcal{M}$, its affiliated operators \cite{MurvNa36RingOp} form a regular ring. This was the original motivation behind \autoref{dfn:vNr}.
    \end{itemize}
\end{exas}

\begin{thm}\label{prp:CharacvNa}
Let $R$ be a ring. The following are equivalent:
\begin{itemize}
    \item[(i)] $R$ is von Neumann regular;
    \item[(ii)] every finitely generated left (resp. right) ideal is generated by an idempotent;
    \item[(iii)] every finitely generated left (resp. right) ideal is a direct summand of $R$ as a module;
    \item[(iv)] every left (resp. right) $R$-module is flat;
    \item[(v)] every short exact sequence of left (resp. right) $R$-modules is pure exact.
\end{itemize}
\end{thm}

\begin{rmk}
    By \cite[Example 1]{AhnMarki87}, any von Neumann regular ring has \emph{local units}\index[terms]{local units}. Here, recall that a ring $R$ is said to have local units if $R=\bigcup_{e\in E} eRe$ with $E$ a directed family of idempotents of $R$ with respect to the order $e\le f$ given by $e=ef=fe$.
\end{rmk}

\section{Refinement monoids}\label{sec:RefMon}

Throughout the paper, $\N$ and $\Z^+$ will denote the semigroups of positive integers and non-negative integers respectively. All monoids will be commutative. 

Recall that a monoid $M$ is \emph{conical}\index[terms]{monoid!conical} if $x=y=0$ whenever $x,y\in M$ satisfy $x+y=0$. $M$ is said to be a \emph{refinement monoid}\index[terms]{monoid!refinement} (or that it satisfies \emph{Riesz refinement}) if, for all quadruples $a, b, c ,d\in M$ such that
\[
    a+b=c+d,
\]
there exist $x,y,z,t\in M$ such that 
\[
    a=x+y,\quad 
    b=z+t,\quad 
    c=x+z,\quad 
    d=y+t.
\]

\begin{exas} $ $
    \begin{itemize}
        \item[(a)] $\Z^+$ is trivially a conical refinement monoid.
        \item[(b)] Any abelian group $G$ is a refinement monoid, but it fails to be conical since $g+(-g)=e$ for any $g\in G$. By forcefully adding a new zero, the set $G\sqcup\{0\}$ becomes conical and still has refinement.
        \item[(c)] \emph{graph monoids} (whose definition we postpone until \autoref{pgr:DfnGJ}) are always conical and have refinement.
        \item[(d)] all \emph{continuous dimension scales} \cite[Definition~3-1.1]{GooWeh05CompDimThe} are conical refinement monoids. By \cite[Corollary~5-3.15]{GooWeh05CompDimThe}, they can all be realized by a von Neumann regular ring.
    \end{itemize}
\end{exas}

In any monoid $M$, we define the \emph{algebraic pre-order}\index[terms]{algebraic pre-order} by writing $x\leq y$ if there exists $z\in M$ such that $x+z = y$. Note that $\leq$ need not be a partial order: In example (a) above, $\leq$ is a total order; however, in example (b), one has $g\leq h$ for any pair $g,h\in G$.

As stated in the introduction, we say that a non-invertible element $p\in M$ is \emph{prime} if, whenever $p\leq x+y$, then either $p\leq x$ or $p\leq y$. Monoids where every non-invertible element can be written as a sum of prime elements are called \emph{primely generated}\index[terms]{monoid!primely generated}. For example, $1$ is the only prime element of $\Z^+$, whilst any group element in $G\sqcup\{0\}$ is prime. In both cases, the monoid is primely generated.

By \cite[Theorem 4.5]{Bro01Canc}, every primely generated refinement monoid is separative where, as also stated in the introduction, $M$ is \emph{separative}\index[terms]{monoid!separative} if $x=y$ whenever $x+x=x+y=y+y$.

If a countable conical refinement monoid is unperforated\index[terms]{monoid!unperforated} ---that is, $x\leq y$ whenever $nx\leq ny$ for some $n\in\mathbb{N}$--- and cancellative\index[terms]{monoid!cancellative} ---i.e. $x\leq y$ whenever $x+z\leq y+z$---, its structure becomes clearer. Recall that a monoid is said to be \emph{simplicial} \index[terms]{monoid!simplicial} if it is of the form $(\Z^+)^n$ for some $n\in\Z^+$.

\begin{thm}[{\cite{EHS1980, Grillet1976}}]
\label{thm:Effros-Handel-Shen-Grillet}    Let $M$ be a conical refinement monoid. Then, $M$ is unperforated and cancellative if and only if $M$ is isomorphic to a direct limit of simplicial monoids.
\end{thm}

Thus, in the unperforated cancellative case, conical refinement monoids (known as \emph{dimension monoids}) are completely understood; see also \cite{vas2022} for a generalization to dimension $\Gamma$-groups and a discussion of graded von Neumann regular rings. The main difficulty when unperforation or cancellativity are not assumed is that any conical monoid can be order-embedded into a conical refinement monoid. In particular, if one can produce a conical monoid with certain bad behaviour (e.g. \autoref{exa:NotSep}), then there will exist a refinement monoid with the same property:

\begin{thm}[{\cite{Weh98Emb}}]\label{prp:WehEmb}
    For any countable conical monoid $M$ there exists a countable conical refinement monoid to which $M$ can order-embed.
\end{thm}

\begin{exa}\label{exa:NotSep}
    Consider the monoid 
    \[
        M=\langle x\mid 2x=3x\rangle ,
    \]
    where note that, although $x\neq 2x$, one has
    \[
        x+x=x+2x=2x+2x.
    \]
    
    Thus, $M$ is not separative. By \autoref{prp:WehEmb} above, there is a non-separative countable conical refinement monoid.
\end{exa}
 
\section{\texorpdfstring{The monoid $V(R)$}{The monoid V(R)}}
\index[symbols]{$V(R)$}\label{sec:VR}

In \cite{MurvNa36RingOp}, Murray and von Neumann introduced a notion of equivalence for projections in operator algebras to adequately measure their associated subspaces. This equivalence, nowadays known as \emph{Murray-von Neumann equivalence}, reads as $p\sim q$ if $p=vv^*$ and $v^*v=q$ for some $v$. In the purely algebraic setting, where an involution is usually not present, one can study idempotents instead of projections, and change the definition of equivalence as follows: Two idempotent elements $e,f$ are \emph{equivalent}\index[terms]{equivalent idempotents} in a ring $R$ if $xy=e$ and $yx=f$ for some $x,y\in R$. This was formally defined by Kaplansky \cite{Kap68RingOp}, and has played an instrumental role in non-stable $K$-theory since. We note that, without loss of generality, one can take $x\in eRf$ and $y\in fRe$. Further, one can show that two idempotents $e,f$ are equivalent if and only if their corresponding principal right ideals $eR$ and $fR$ are isomorphic as right $R$-modules.

\begin{dfn}
    Consider the ring $M_{\infty}(R)$ built as the directed union of $M_n(R)$ ($n\in\mathbb N$) with transition maps $M_n(R)\to M_{n+1}(R)$ given by $x\mapsto \left( \smallmatrix x&0\\ 0&0\endsmallmatrix \right)$. Then, $ V (R)$ is defined to be the monoid of equivalence classes $[e]$ of idempotents $e$ in $M_\infty(R)$ with the operation
    \[
        [e]+ [f] := \Big[  \begin{pmatrix} e&0\\ 0&f \end{pmatrix}  \Big].
    \]
\end{dfn}

\begin{rmk}
    For a unital ring $R$, this construction agrees with the monoid of isomorphism classes of finitely generated projective right $R$-modules, where the addition is induced by direct sum. Thus, in this case, its Grothendieck group is $K_0(R)$.
\end{rmk}

\begin{exas} $ $
    \begin{itemize}
        \item[(a)] Let $R$ be any (unital) principal ideal domain. Since any finitely generated projective module is isomorphic to $R^k$ for some $k$, the map $[R^k]\mapsto k$ gives an isomorphism of monoids $V(R)\cong \mathbb{Z}^+$. As stated previously, this monoid is cancellative and unperforated.
        \item[(b)] Let $R=\mathbb{R}[t_0,t_1,t_2]/\langle t_0^2+t_1^2+t_2^2-1\rangle$ be the coordinate ring of the real $2$-sphere. As shown in \cite[Proposition~I.4.15]{Lam06Serre}, there exists an indecomposable (finitely generated projective) $R$-module $P$ such that $P\oplus R\cong R^3$. In particular, one does not have $P\cong R^2$, from which it follows that $V(R)$ is not cancellative.

        Further, as a consequence of Bass's cancellation theorem, one deduces that $P\oplus P\cong R^4$ is a free $R$-module. Therefore, $2[P]=2(2[R])$ and yet $[P]\neq 2[R]$, as seen above. We conclude that $V(R)$ is not unperforated.

 \item[(c)] As shown in  \cite[Example 4]{MenalMoncasi1982}, there exists a von Neumann regular ring $R$ with stable rank $2$ such that $\mathrm{K}_0(R)\cong \Z\oplus\Z/2\Z$. By a standard proof, one can show that unperforation of $V(R)$ implies that $\mathrm{K}_0(R)$ is unperforated. Hence, we conclude that $V(R)$ is not unperforated.
    \end{itemize}
\end{exas}

Let $R$ be a regular ring. Since $M_n (R)$ is von Neumann regular (\autoref{exas:vNReg}~(a)), it follows that $M_\infty (R)$ is also regular. In the case of such rings, $V(R)$ has Riez refinement:

\begin{thm}[{\cite[Corollary 1.3]{AGOP98}}]
    Let $R$ be a von Neumann regular ring. Then, $V(R)$ is a conical refinement monoid.
\end{thm}

Going back to our discussion on dimension monoids, one can use \autoref{thm:Effros-Handel-Shen-Grillet} in combination with Elliott's celebrated classification of AF-algebras (concretely, \cite[Theorem~5.5]{Ell76Real}) to obtain a realization result for all such monoids:

\begin{thm}
    Let $M$ be a countable dimension monoid. Then, there exist an AF-algebra $A$ and an ultramatricial algebra $R$ over an arbitrary field $F$ realizing $M$.
\end{thm}

We will also make use of the following definition:

\begin{dfn}
    A ring $R$ is said to be \emph{separative}\index[terms]{ring!separative} if its monoid $ V (R)$ is a separative monoid.
\end{dfn}

\section{Directed graphs}

Directed graphs are our basic combinatorial tool. Here we introduce the notation on directed graphs that we will use throughout the work. We follow the standard text \cite{AAS}.

A (directed) graph $E = (E^0, E^1, s, r)$ consists of two disjoint sets $E^0$ and $E^1$, called \emph{vertices} and \emph{edges} respectively, together with two maps $s, r: E^1 \longrightarrow E^0$. For $e\in E^1$, the vertices $s(e)$ and $r(e)$ are referred to as the \emph{source} and the \emph{range} of $e$, respectively.  A graph $E$ is called {\it row-finite} if $s^{-1}(v)$ is finite for all $v\in E^0$. It is called {\it finite} if both $E^0$ and $E^1$ are finite sets. 

A {\em sink} in a graph $E$ is a vertex $v \in E^0$ with $s^{-1}(v) = \emptyset$; a {\em source} is a vertex $v \in E^0$ with $r^{-1} (v) =\emptyset$.    

A (finite) \emph{path} in a graph $E$ is a string
$p=e_1\cdots e_n$ of edges $e_i\in E^1$ such that $r(e_i) = s(e_{i+1})$ for all $i$. The \emph{length} of the path $p = e_1 \cdots e_n$ is $n$, and is denoted by $|p|$.  
The source and range maps on edges are extended to paths as
\[s(p)=s(e_1)\qquad\text{and}\qquad r(p)=r(e_n).\]
Vertices are  regarded as paths of length $0$. We denote by $\text{Path}(E)$ the set of all paths in $E$.

\section{Posets}

Many of the objects that we will consider naturally induce poset structures. For example, each directed graph $E$ gives rise to a poset $\mathbb P (E)$, which is the antisymmetrization of the pre-order $\le $ on the set of vertices $E^0$ defined by $v\le w$ if and only if there is a path from $w$ to $v$ (see \autoref{ntn:notations-for-separated-graphs}). To deal with these partially ordered sets, we will need the following notation:

\begin{ntn}
    Given a poset $(I,\le )$, a subset $J$ of $I$ is called a \emph{lower subset}\index[terms]{lower subset} if $x\in J$ whenever $x\le y$ and $y\in J$.
    
    For $p\in I$, the set 
    \[
        I\downarrow p := \{ x\in I \mid x\le p \}
    \]    
    is the \emph{lower subset} of $I$ generated by $p$.
    
    We will write $\rL (p)$ to denote the \emph{lower cover}\index[terms]{lower cover} of $p$, that is, 
\[
    \rL(p)=\rL(I , p)=\{q\in I \mid q<p \text{ and } [q,p]=\{q,p\}\}
\]
where $[q,p]= \{x\in I : q\le x\le p \}$ is the \emph{poset interval}\index[terms]{poset!interval} from $q$ to $p$.

The set of all lower subsets will be written as $\mathcal L (I)$, which becomes a complete distributive lattice when equipped with intersection and union.
\end{ntn}

In addition to these notions, we will also need the concepts of tree and forest: 

\begin{dfn}
 \label{dfn:poset-be-a-tree}
 Let $(I,\leq )$ be a poset. We say that
 \begin{itemize}
     \item[(a)] $(I,\leq )$ is a \emph{tree}\index[terms]{poset!tree} if there is a greatest element $i_0\in I$ such that the interval $[i,i_0]$ is a chain for every $i\in I$. We will say that $i_0$ is the \emph{root}\index[terms]{poset!root} of $I$.
     \item[(b)] $(I,\leq )$ is a \emph{forest}\index[terms]{poset!forest} if one can write $I=\bigcup_{\alpha \in \Lambda} I_{\alpha}$ where $I_{\alpha}$ is a tree (with the induced order), and for each $\alpha\ne \beta$ the elements
 of $I_{\alpha}$ and  $I_{\beta}$ are pairwise incomparable.
 \end{itemize} 
 \end{dfn}




\chapter{From monoids to I-systems}\label{sec:MondIsystm}

In this chapter we carry out \hyperlink{Step1}{Step (1)} from the strategy in the introduction. Specifically, we define marked monoids, revisit the constructions of $\mathcal{I}(M,\mathfrak{m})$ and $\mathcal{M}(J)$, and define the notion of super adaptability for monoids and $I$-systems whilst keeping track of their marks. This establishes the constructions in the following diagram:
\[
\xymatrixrowsep{1pc}
\xymatrixcolsep{5pc}
\xymatrix{
    {\begin{array}{c}
         \text{s.a.}  \\
          \text{monoid}\\
          M
    \end{array}}
    \ar@/^/[r]^-{\mathcal{I}}
    &
    {\begin{array}{c}
         \text{s.a.}  \\
          I\text{-system}\\
          J
    \end{array}}
    \ar@/^/[l]^-{\mathcal{M}}
}
\]

In the second part of the chapter, we use the language of marks to prove \autoref{prp:morphism-of-systems}, which allows us to lift certain monoid morphisms to maps between $I$-systems.


\section{$I$-systems}
\label{sect:Isystems}

First, we need to recall certain aspects of primitive monoids and prime elements:

\begin{pgr}[Primitive monoids]\label{pgr:PrimMonod}
    A monoid is said to be \emph{primitive}\index[terms]{monoid!primitive} if it is primely generated, antisymmetric and has Riesz refinement.

    Given a primitive monoid $M$, consider its set of prime elements $\mathbb P (M)$\index[symbols]{$\mathbb P (M)\quad$(prime elements)} and the relation $\lhd $ on $\mathbb P (M)$ given by $q\lhd p$ if and only if $p= p+q$\index[symbols]{$\lhd\quad$(transitive relation)}. One can readily check that $\lhd $ is antisymmetric and transitive.
    
    Any primitive monoid $M$ is determined by these two constructions ($\mathbb P (M)$ and $\lhd $). Indeed, given any pair $(\mathbb P, \lhd)$, where $\mathbb P$ is a set and $\lhd$ is an antisymmetric transitive relation on $\mathbb P$, we know from \cite[Proposition~3.5.2]{Pierce} that the commutative monoid 
    \[
        M(\mathbb P, \lhd):=\langle p\in \mathbb P\mid p=p+q\text{ whenever }q\lhd p\rangle 
    \]
    is a primitive monoid, and that the assignments $M\mapsto (\mathbb P (M),\lhd)$ and $(\mathbb P, \lhd)\mapsto M(\mathbb P, \lhd)$ define mutually inverse correspondences (up to isomorphism) between primitive monoids and pairs $(\mathbb P,\lhd)$ as above.
\end{pgr}

\begin{pgr}[Free and regular primes]\label{pgr:FreeRegPrimes}
    Let $M$ be a primely generated conical refinement monoid. An element $x\in M$ is called \emph{regular}\index[terms]{prime!regular} if $2x\leq x$. It is called \emph{free}\index[terms]{prime!free} if, whenever $nx\leq mx$ for some $n,m\in\mathbb{N}$, one has $n\leq m$. If $M$ is primely generated, every element is either free or regular; see \cite[Theorem~4.5]{Bro01Canc}. In particular, one can divide the set of prime elements into those that are free, which we denote by $\Pfree (M)$, and those that are regular, denoted by $\Preg (M)$\index[symbols]{$\Pfree (M)\quad$ (set of free primes)}\index[symbols]{$\Preg (M)\quad$ (set of regular primes)}. Further, it follows from \cite[Theorem~4.6]{AraPar16Isr} that every primely generated conical refinement monoid is tame (\cite[Definition 2.1]{AraGoo15SgpFor}), that is, it is isomorphic to a direct limit of finitely generated refinement monoids.
    
    Define the relation $\equiv$ on $M$ by writing $x\equiv y$ if and only if $x\leq y$ and $y\leq x$\index[symbols]{$\equiv\quad$(antisymmetric relation)}. The quotient $\overline{M}:=M/\equiv$\index[symbols]{$\overline{M}$} becomes a primitive monoid, and we have a natural surjective homomorphism $\pi_M\colon M\to \ol{M}$. One can easily check that $\overline{p}\in \overline{M}$ is prime (resp. free, regular) if and only if $p\in M$ is prime (resp. free, regular). By choosing a representative in $M$ for each prime element in $\ol{M}$, we obtain a subset $\mathfrak m$ of prime elements in $M$. By \cite[Lemma 2.7]{Bro01Canc}, the set  $\pi_M^{-1}(p)$ is always a group whenever $p$ is a regular prime of $\ol{M}$, and in this case the representative that we select in $\pi_M^{-1}(p)$ is always the identity element of that group.  
   We call the pair $(M,\mathfrak{m})$ of a primely generated conical refinement monoid $M$ and a subset $\mathfrak m$ of representatives of primes in $\ol{M}$ a \emph{marked monoid}\index[terms]{marked monoid} with mark $\mathfrak m$. 
   Although the particular choice of a mark $\mathfrak m$ in $M$ does not affect the main results in this paper, it is convenient to consider marked monoids when we discuss morphisms in  \autoref{subsect:lifting-morphisms}.    
    
      We equip $\mathbb{P} (M)$ (and in particular any mark $\mathfrak m$ of $M$) with the order $\leq^*$\index[symbols]{$\leq^*$} defined by $x\leq^* y$ if and only if $x=y$ in $M$ or if $\bar{x}<\bar{y}$ in $\mathbb{P}(\overline{M})$. We write $x<^* y$ if $\bar{x}<\bar{y}$ in $\mathbb{P}(\overline{M})$.
      
      If $a\in M$, we denote by $M_a$ the \emph{archimedean component}\index[terms]{archimedean component} of $a$, that is, the set of elements $b\in M$ such that $a\leq mb$  and $b\leq na$ for some $m,n\in\mathbb{N}$. Clearly, one has $M_a=M_b$ whenever $a\equiv b$ in $M$, and one may denote the archimedean component $M_a$ by $M_{\ol{a}}$. In what follows (\autoref{dfn:Isystem}), we will make a slight abuse of notation and denote certain monoids (a priori unrelated to archimedean components) by $M_p$ with $p$ a poset element. This is justified in \autoref{pgr:IM}.
\end{pgr}

As stated in the outline of the strategy, we make use of the concept of \emph{$I$-systems} introduced in \cite[Definition~1.1]{AraPar16Isr} to build a bridge between super adaptable separated graphs and primely generated conical refinement monoids. We begin by recalling their definition and main constructions.

\begin{dfn}\label{dfn:Isystem}
Let $I$ be a poset with order $\leq$. An \emph{$I$-system}\index[terms]{I-system} $J$ is a tuple
\[
    (I,\leq ,(G_p)_{p\in I},\varphi_{p,q}\,(q< p))
\]
together with a partition $I=I_{\rm free}\sqcup I_{\rm reg}$ such that
\begin{itemize}
    \item[(a)] $(G_p)_{p\in I}$ is a family of abelian groups, where we adopt the notation
    \begin{enumerate}
        \item[(a1)] For $p\in I_{\rm reg}$, $M_p:=G_p$ and $\widehat{G_p}:=G_p$.
        \item[(a2)] For $p\in I_{\rm free}$, $M_p:=\mathbb{N}\times G_p$ and $\widehat{G_p}:=\mathbb{Z}\times G_p$, which is the Grothendieck group of $M_p$. We see $G_p\subset \widehat{G_p}$ through the embedding $g\mapsto (0,g)$. 
    \end{enumerate}
        
    \item[(b)] $\varphi_{p,q}$ are monoid morphisms $\varphi_{p,q}\colon M_q\to G_p$ such that its associated induced group morphisms $\widehat{\varphi}_{p,q}\colon \widehat{G_q}\to \widehat{G_p}$ satisfy:
    \begin{enumerate}
        \item[(b1)] The assignment
        \[
            \begin{split}
                p\mapsto \widehat{G_p}\\
                (q<p)\mapsto \widehat{\varphi}_{p,q}
            \end{split}
        \]
        is a functor from $(I,\leq )$ to the category of abelian groups.
        \item[(b2)] For each $p\in I_{\rm free}$, the map
        \[
            \bigoplus_{q<p}\varphi_{p,q}\colon \bigoplus_{q<p}M_q\to G_p
        \]
        is surjective.
    \end{enumerate}
\end{itemize}

We will say that the $I$-system is \emph{countable} if $I$ is a countable poset and each $G_p$ is countable.
\end{dfn}

\begin{ntn}\label{ntn:pMp}
    Let $J$ be an $I$-system and let $p\in I_{\rm free}$. Throughout the text we will denote by $p$ the element $(1,0_{G_p})$ in $M_p$, thus writing a general element in $M_p$ as $np+g$ with $n\in\NN$ and $g\in G_p$. The reason for this notation stems from the discussion in \autoref{pgr:IM} and the identification from \autoref{thm:APIsr16}.
\end{ntn}


We will now define \emph{elementary morphisms} of systems, a notion that generalizes \cite[Definition~3.3]{AraPar16Isr}. As we will see later, these elementary morphisms induce monoid homomorphisms between the associated monoids, extending \cite[Lemma 3.4]{AraPar16Isr}.  

Given a group homorphism $f\colon G\to H$ and an element $h\in H$, we may define the semigroup homomorphism
\[
    f^h\colon \NN \times G \longrightarrow \NN\times H,\quad f^h(n, a) = (n, f(a)+nh).
\]
This homomorphism extends to a group homomorphism $\ZZ\times G\to \ZZ\times H$ with the same formula, which will be also denoted by $f^h$. With the usual embedding $G\subset \ZZ\times G$, we have that $f^h$ extends $f$.  

With this notation at hand, we can now present our definition of elementary morphism of systems:

\begin{dfn}
	\label{dfn:hom-of-systems}
An \emph{elementary homomorphism}\index[terms]{I-system!elementary morphism} of systems $f\colon J_1\to J_2$, with 
\begin{align*}
    J_1&=(I^{(1)},\leq ,(G_p^{(1)})_{p\in I^{(1)}},\varphi_{p,q}^{(1)}\,(q< p)),\quad\text{and}\\
    J_2&=(I^{(2)},\leq ,(G_r^{(2)})_{r\in I^{(2)}},\varphi_{r,s}^{(2)}\,(s< r)),
\end{align*}
consists of 
\begin{itemize}
    \item[(a)] a poset map $\psi\colon I^{(1)}\to I^{(2)}$ that sends regular elements to regular elements, free elements to free elements, and pairs $q<p$ to pairs $\psi (q)<\psi (p)$; and
    \item[(b)] a family of group morphisms $f_p\colon G_p^{(1)}\to G_{\psi (p)}^{(2)}$, together with a family of elements $\{h(p) \in G^{(2)}_{\psi (p)}\mid p\in \Ifree^{(1)}\}$ such that the diagram
    \[
        \xymatrixcolsep{5pc}
        \xymatrixrowsep{4pc}
        \xymatrix{
            M_q^{(1)} \ar[r]^{\varphi_{p,q}^{(1)}} \ar[d]_{\bar{f_q}} & 
            G_p^{(1)} \ar[d]^{f_p} \\
            M_{\psi(q)}^{(2)} \ar[r]^{\varphi_{\psi (p),\psi (q)}^{(2)}} & G_{\psi (p)}^{(2)}
        }
    \]
    is commutative, where $\bar{f_q}=f_q$ if $q$ is regular and $\bar{f_q}= f_q^{h(q)}$ if $q$ is free.
\end{itemize}
\end{dfn}

\begin{rmk}
    Elementary morphisms of systems can be composed: if 
    \[
        J_i = (I^{(i)},\leq ,(G_p^{(i)})_{p\in I^{(i)}},\varphi_{p,q}^{(i)}\,(q< p)), \quad i=1,2,3
    \]
are systems, and $f_i \colon J_i\to J_{i+1}$, $i=1,2$, are elementary morphisms, given by the data $(\psi_i, f_p^{(i)}, h_i(p))$ respectively, then $f_2\circ f_1$ is an elementary morphism from $J_1$ to $J_3$, given by the data
\[
    (\psi_2\circ \psi_1, f^{(2)}_{\psi_1(p)}\circ f^{(1)}_p, f_{\psi_1(p)}^{(2)}(h_1(p)) + h_2(\psi_1(p))).
\]
\end{rmk}        

\begin{dfn}
We say that two systems $J_1$ and $J_2$ are \emph{elementary isomorphic} if there is an invertible elementary homomorphism between them, that is there are elementary morphisms of systems $f\colon J_1\to J_2$ and $g\colon J_2\to J_1$ such that $g\circ f =\text{id}_{J_1}$ and $f\circ g=\text{id}_{J_2}$.
\end{dfn}

\begin{ntn}
    The definition of morphism of systems given in \cite[Definition 3.3]{AraPar16Isr} coincides with our definition when taking the family $\{ e_{\psi(p)}\in G^{(2)}_{\psi (p)} \mid p\in \Ifree^{(1)} \}$ of neutral elements in $G_{\psi (p)}$ as $\{h(p) \in G^{(2)}_{\psi (p)}\mid p\in \Ifree^{(1)}\}$. In this case, we will refer to $f$ as a \emph{plain homomorphim}\index[terms]{I-system!plain morphism} of systems. A \emph{plain isomorphism} will be an elementary isomorphism of systems which is also plain. Observe that the inverse of a plain isomorphism is also a plain isomorphism. 
\end{ntn}

\begin{rmk}
    Let us remark on why we have decided to reserve the name \emph{morphism of systems} for future use and instead called our maps plain / elementary morphisms: while every primely generated conical refinement monoid $M$ is isomorphic to $\mathcal{M}(J)$ for some $I$-system $J$ (\autoref{thm:APIsr16}), an arbitrary monoid morphism $f: M_1 \to M_2$ does not generally lift to an elementary morphism of systems. For example, the automorphism of $\mathbb{N}$ defined by $1 \mapsto 2$ cannot be so lifted. Thus, we reserve the term \emph{morphism} for a broader class of maps (to be defined in future work) that will correspond 1-to-1 with monoid morphisms.
\end{rmk}

\section{The constructions}
Following the strategy from the introduction, we start by recalling the construction of the system $\mathcal{I}(M)$ and the monoid $\mathcal{M}(J)$ whilst keeping track of the marks. As shown in \cite[Theorem~0.1]{AraPar16Isr}, these constructions are each others' inverse. We define super adaptable monoids in \autoref{dfn:superadaptableIsystem} and give some examples. 

\begin{pgr}[Defining $\mathcal{I}(M)$]\label{pgr:IM}\index[terms]{I-system!of a monoid}\index[symbols]{$\mathcal{I}(M,\mathfrak{m})\quad$(I-system of a monoid)}
    Given a marked primely generated conical refinement monoid $(M,\mathfrak m)$, one defines its  associated $\mathbb{P}(\ol{M})$-system $\mathcal I (M,\mathfrak m)$ as follows (see \cite[Section~2]{AraPar16Isr} for details, where this was simply denoted by $\mathcal I (M)$):

    The poset is given by $\mathbb{P}(\ol{M})$ with its algebraic order $\leq$ (which is a partial order). Note that, for $p,q\in \mathbb{P}(\ol{M})$, we have $q\le p$ if and only if ($p=q$ or $p=p+q$).

    For $p\in \mathbb{P}(\ol{M})$, let $M_p$ denote the archimedean component of $M$ corresponding to $p$, that is, the set of elements $a\in M$ such that $\ol{a}\leq mp$  and $p\leq n\ol{a}$ (and thus $p\leq \ol{a}$) for some $m,n\in\mathbb{N}$.

    We define the groups $G_p$ as follows:
    \begin{itemize}
        \item If $p$ is regular, $M_p$ is an abelian group \cite[Lemma~2.7]{Bro01Canc}, and we set $G_p=M_p$.
        \item If $p$ is free, let $\Grot{M_p}$ denote the Grothendieck group of $M_p$. We define $G_p$ as the subgroup of $\Grot{M_p}$ given by
        \[
            \{
            (\tilde{p}+a )-\tilde{p}\mid a\in M\text{ and }\tilde{p}+a\leq \tilde{p}
            \},
        \]
        where $\tilde{p}$ is a representative of $p$ in $M$. It is easy to check that the element $(\tilde{p}+a) - \tilde{p}$ of $\Grot{M_p}$ appearing in the above definition does not
        depend on the particular representative $\tilde{p}$ of $p$. Hence, with a slight abuse of notation, we will denote it by $(p+a)-p$. 
        
        One can show that, for any free prime $p\in \mathbb P (\ol{M})$, and any representative $\tilde{p}$ of $p$, there is an isomorphism $\Phi_{\tilde{p}} \colon \NN\times G_p \to M_p$ such that 
        \[
            \Phi_{\tilde{p}} (n, (p+a)-p) =  n\tilde{p} +a,
        \]
        where $p+a\le p$; see \cite[Lemma~2.4]{AraPar16Isr}. 
    \end{itemize}

    By \cite[Lemma~2.6]{AraPar16Isr}, the maps $\varphi_{p,q}\colon M_q\to G_p$ (for $q<p$) given by
    \[
        \varphi_{p,q} (x)=(p+x)-p
    \]
    are well-defined semigroup homomorphisms; note that the semigroups $M_p$, the groups $G_p$, and the morphisms $\varphi_{p,q}$ do not depend on the mark $\mathfrak m$. However, since the semigroup $M_q^{\mathcal I}$ associated to the $\mathcal I$-system is defined to be $M_q^{\mathcal I}= \NN \times G_q$, instead of $M_q$,  whenever $q$ is a free prime, we need to precompose $\varphi_{p,q}$ with the isomorphism $\Phi_{\tilde{q}}\colon \NN\times G_q \to M_q$  in order to get the correct map, which does depend on the mark $\mathfrak m$. Therefore we define the map
    \[
        \varphi^{\mathfrak m}_{p,q} = \varphi_{p,q} \circ \Phi_{\tilde{q}}\colon M^{\mathcal I}_q \to G_p,
    \]
where $\tilde{q} \in \mathfrak m$, $\pi_M(\tilde{q}) =q$, $\Phi_{\tilde{q}} = \text{id}_{G_q}$ when $q$ is regular, and $\Phi_{\tilde{q}}$ is the above isomorphism when $q$ is free. Concretely we obtain the following formula for the connecting maps $\varphi_{p,q}^{\mathfrak m}$, where $q\in \Pfree (\ol{M})$:
\begin{equation}
	\label{eq:formula-for-varphi-m}
	\varphi_{p,q}^{\mathfrak m} (n,(q+a)-q) = (p+n\tilde{q}+a)-p,\qquad (q+a\le q, \tilde{q}\in \mathfrak m, \pi_M(\tilde{q})=q).
		\end{equation}
	
It follows that the system slightly depends on the mark $\mathfrak m$, and that is why we write $\mathcal I (M,\mathfrak m)$ to stress this dependence. 
\end{pgr}

Although the system associated to a given primely generated conical refinement monoid $M$ depends on the mark
we choose in $M$, all of them are elementary (but not plainly) isomorphic, as shown in the next lemma.

\begin{lma}
	\label{lma:isomorphism-of-marked-systems}
	Let $M$ be a primely generated conical refinement monoid, and let $\mathfrak m_1$ and $\mathfrak m_2$ be two marks on $M$. Then, there is an elementary isomorphim of systems $f\colon \mathcal I (M,\mathfrak m_1)\to \mathcal I (M,\mathfrak m_2)$ such that $f_p=\mathrm{id}_{G_p}$ for all $p\in I$, and $h(p)= (0,p_1-p_2)$, where $p_i\in \mathfrak m_i\cap \pi_M^{-1}(p)$ for all $p\in \Pfree (M)$.  
	\end{lma}

\begin{proof}
	First observe that, since $p_1\equiv p_2$ we can write $p_1= p_2+\alpha$, where $\alpha \in M$, and clearly
	$p_2+\alpha = p_1\leq p_2$, so that $p_1-p_2 = (p_2+\alpha)-p_2 \in G_p$, where $\pi_M(p_1)=\pi_M (p_2) =p$ and $p_i\in \mathfrak m_i$ for $i=1,2$. Hence $h(p)\in G_{\psi (p)} = G_p$ as required. 
	We need to check that $\varphi_{p,q} ^{\mathfrak m_1} = \varphi_{p,q}^{\mathfrak m_2}\circ \ol{\text{id}}$ whenever $q<p$. This is obvious when $q$ is a regular prime, so let us assume that $q$ is free. Then we have, for $n\in \NN$ and $q+a\leq  q$,
	\begin{align*}
		 (\varphi_{p,q}^{\mathfrak m_2} \circ \ol{\text{id}}) (n,(q+a)-q)  &= \varphi_{p,q}^{\mathfrak m_2}({\text{id}}^{h(q)} (n,(q+a)-q))\\
         &= \varphi_{p,q}^{\mathfrak m_2} (n, n(q_1-q_2)+a)\\
		  & = (p+nq_2+nq_1-nq_2+a) -p\\
          &= (p+nq_1+a) -p \\
          &= \varphi_{p,q}^{\mathfrak m_1}(n,(q+a)-q).
	\end{align*} 
	This shows the result. 
\end{proof}

\begin{rmk}\label{rmk:GpCountGen}
All the groups $G_p$ defined in \autoref{pgr:IM} above are countable whenever $M$ is a countable, primely generated, conical, refinement monoid. Further, in this case $\mathbb{P} (\ol{M})$ is also countable, which implies that $\mathcal{I}(M)$ is a countable system.
\end{rmk}

\begin{pgr}[Defining $\mathcal{M}(J)$]\label{pgr:MJ}\index[terms]{monoid!of system}\index[symbols]{$\mathcal{M}(J)\quad$(monoid of sysetm $J$)}
    Let $J$ be an $I$-system $(I,\leq ,(G_p)_{p\in I},\varphi_{p,q}\,(q< p))$. We define $\mathcal{M}(J)$ as the monoid generated by the symbols $\chi_p (x)$, with $p\in I$ and $x\in M_p$, subject to
    \[
        \chi_p(x)+\chi_p(z)=\chi_p(x+z)
        ,\andSep 
        \chi_p (x)+\chi_q(y)=\chi_p(x+\varphi_{p,q}(y))
    \]
    for all $x,z\in M_p$, $q<p$, and $y\in M_q$; see \cite[Section~1]{AraPar16Isr} for details.

    It is shown in \cite[Lemma 3.4]{AraPar16Isr} that every plain morphism of systems $f\colon J_1 \to J_2$ induces a monoid morphism $\mathcal{M}(f)\colon \mathcal{M}(J_1) \to \mathcal{M}(J_2)$ given by $$\mathcal{M}(f)(\chi_p (x)):=\chi_{\psi (p)}(\overline{f}_p(x)).$$ 
    The same holds for elementary homomorphisms, as we subsequently show, with the same formula.
    
We may endow $\mathcal M (J)$ with its \emph{canonical mark}\index[terms]{marked monoid!mark of  $\mathcal{M}(J)$}\index[symbols]{$\mathfrak m_J\quad$(mark of  $\mathcal{M}(J)$)} $\mathfrak m_J$, namely we select the neutral element $e_p\in G_p$ for each $p\in \Ireg$ and we select the 
element $(1,e_p)\in M_p=\NN\times G_p$ for each $p\in \Ifree$.
\end{pgr} 

\begin{pgr}[Order-ideals of $\mathcal{M}(J)$]\label{pgr:OIdMJ}
Let $M$ be a monoid. An \emph{order-ideal}\index[terms]{order-ideal} of~$M$ is a submonoid $I$ of~$M$ such that $x,y\in I$ whenever $x+y\in I$.

Now let $J$ be an $I$-system. As shown in \cite[Proposition 1.9]{AraPar16Isr}, there is a lattice isomorphism from the lattice of lower subsets of $I$ onto the lattice of order-ideals of $\mathcal{M}(J)$. More concretely, for any lower subset $L$ of $I$, one can form the monoid $\mathcal{M} (J|_L)$, where $J|_L$ denotes the restriction of the $I$-system to $L$. The monoid $\mathcal{M} (J|_L)$ naturally corresponds to an order-ideal of $\mathcal{M}(J)$.

Now let $p$ be a free prime in $I$, and let $J_p = \{q\in I \mid q<p \}$, which is a lower subset of $I$.  Therefore, we can consider the associated monoid $\mathcal{M}(J_p) := \mathcal{M}(J |_{J_p})$ as above, which is an order-ideal of $\mathcal{M}(J)$. 
  
Following \cite[Lemma~5.1]{AraPar17JAlg}, the maps $\varphi_{p,q}$, $q<p$, induce a surjective semigroup homomorphism
\[
    \varphi_p\colon \mathcal{M}(J_p)\to G_p.
\]

In particular, if $r\in\mathcal{M}(J_p)$ satisfies $\varphi_p(r)=0$, note that $\chi_p (p)+r=\chi_p(p)$.

The map $\varphi_p$ gives rise to a group homomorphism $\Grot{\varphi_p}\colon \Grot{\mathcal{M}(J_p)}\to G_p$, where $G$ denotes the Grothendieck functor. Following \cite[Section~5]{AraPar17JAlg}, we write $\Grot{\mathcal{M}(J_p)}^{++}:=\iota_{J_p}(\mathcal{M}(J_p))$ where $\iota_{J_p}\colon \mathcal{M}(J_p)\to \Grot{\mathcal{M}(J_p)}$ is the canonical homomorphism.
\end{pgr}

\begin{lma}\label{lma:morphisms-of-systems-induce-homos}
Let $f= (\psi, \{f_p\}_{p\in I}, \{h(p)\}_{p\in \Ifree})$ be an elementary homomorphism of systems $f\colon J_1\to J_2$. Then $f$ induces a homomorphism of monoids $\mathcal M (f)\colon \mathcal M (J_1)\to \mathcal M (J_2)$\index[symbols]{$\mathcal{M}(f)$}, given by
\[
    \mathcal M (f) (\chi_p(x)) = \chi_{\psi (p)} (\ol{f_p} (x))
\]
for all $x\in M_p$. The homomorphism $\mathcal M (f)$ sends free primes to free primes and regular primes to regular primes, and $\mathcal M (f)(q) <^* \mathcal M (f)(p)$ whenever $q<^*p$ in $\mathbb P (\mathcal M(I))$.   
\end{lma}
\begin{proof}
The first part of the proof is the same as the one of \cite[Lemma 3.4]{AraPar16Isr}. The last part is clear from the definitions and the fact that, for any $I$-system $J$, the primes in $\mathcal M (J)$ are precisely the elements of the form $\chi_i(x)$ for $x\in G_p$, where $p\in \Ireg$, and the elements of the form $\chi _i(1,x)$ for $x\in G_p$, where $p\in \Ifree$; see \cite[Remark 2.8]{AraPar16Isr}. 
\end{proof}  

\begin{thm}[c.f. {\cite[Theorem~0.1]{AraPar16Isr}}]\label{thm:APIsr16}
    Let $(M,\mathfrak m)$ be a marked, primely generated, conical, refinement monoid, and let $J$ be an $I$-system. Then,
    \begin{itemize}
        \item[(a)] $\mathcal{I}(M, \mathfrak m)$ is a $\mathbb{P}(\ol{M})$-system, and there is an isomorphism of marked monoids 
        \[
            \Phi_{(M,\mathfrak m)} \colon (M,\mathfrak m)\longrightarrow (\mathcal{M}(\mathcal{I}(M,\mathfrak m)), \mathfrak m_{\mathcal I (M,\mathfrak m)})
        \]
        such that
        \begin{enumerate}
            \item $\Phi_{(M,\mathfrak m)} (a) = \chi_{\ol{p}}(a)$ for $\ol{p}\in \Preg (\ol{M})$ and $a\in G_{\ol{p}}$;
            \item $\Phi_{(M,\mathfrak m)} (np+a)= \chi_{\ol{p}}(n,(p+a)-p)$ for $p\in \mathfrak m \cap \Pfree (M)$, $n\in \NN$ and $a\in M$ with $p+a\le p$.
        \end{enumerate}
        \item[(b)] $\mathcal{M}(J)$ is a primely generated, conical, refinement monoid, and there is a canonical plain isomorphism of systems $J\cong \mathcal{I}(\mathcal{M}(J),\mathfrak m_J)$.
    \end{itemize}
\end{thm}

Given an $I$-system $J$, condition (b2) in \autoref{dfn:Isystem} states that $G_p=\sum_{q<p}\varphi_{p,q}(M_q)$ whenever $p$ is free. In particular, $G_p$ is generated as a semigroup by the set
\[
    \left(\bigcup_{q<p\text{ regular}}
    \varphi_{p,q}(G_q)
    \right)
    \cup
    \left(\bigcup_{q<p\text{ free}}
    \varphi_{p,q}(\{1\}\times G_q^*)
    \right)
    \cup
    \{ \varphi _{p,q}(q)\}_{q<p\text{ free}},
\]
with the identification $q=(1,0_{G_q})$ from \autoref{ntn:pMp}. Here and everywhere else $M^*=M\setminus \{0\}$ for every monoid $M$. 

Super adaptability for $I$-systems means that one can dispense with $\varphi_{p,q}(\{1\}\times G_q^*)$:

\begin{dfn}
\label{dfn:superadaptableIsystem}
Let $J$ be an $I$-system. We say that $J$ is \emph{super adaptable}\index[terms]{I-system!super adaptable} if, for each $p\in \Ifree$, the group $G_p$ is generated as a semigroup by $(\cup_{q<p\text{ reg.}}\varphi_{p,q} (G_q))\cup \{ \varphi _{p,q}(q)\}_{q<p\text{ free}}$.  

A marked primely generated conical refinement monoid $(M,\mathfrak m)$ will be said to be \emph{super adaptable}\index[terms]{marked monoid!super adaptable} if $\mathcal{I}(M,\mathfrak m)$ is a super adaptable $I$-system.

A primely generated conical refinement monoid $M$ will then be said to be \emph{super adaptable}\index[terms]{monoid!super adaptable} if $\mathcal{I}(M,\mathfrak m)$ is a super adaptable $I$-system for \emph{some} mark $\mathfrak m$ on $M$. 
\end{dfn}

\begin{rmk}
    Somewhat surprisingly, the notion of super adaptable monoid really depends on the particular mark chosen in $M$; see \autoref{exa:not-super adaptable-Isystem}. It is an easy exercise to show that if $J_1$ and $J_2$ are two plainly isomorphic systems (i.e. there is a plain isomorphism between them), then $J_1$ is super adaptable if and only if $J_2$ is. 
\end{rmk}

\begin{exas}\label{exa:ExasSupAdaptMon}$ $
    \begin{itemize}
        \item[(a)] Let $J$ be an $I$-system such that $\Ifree$ is an artinian poset. Then, $J$ is super adaptable. Indeed, let $I'$ be the set of elements in $\Ifree$ such that the condition in \autoref{dfn:superadaptableIsystem} is not satified. By way of contradiction, assume that $I'$ is non-empty, and let $p\in I'$ be a minimal element in $I'$, which exists because $\Ifree$ is an artinian poset. Let $I_p$ denote the set of free primes $q\in \Ifree$ such that $q<p$.
        By the observation preceding \autoref{dfn:superadaptableIsystem}, in order to arrive to a contradiction, it suffices to show that $\varphi_{p,q} (1,g)$ is generated by $(\cup_{q'<p\text{ reg.}}\varphi_{p,q'} (G_{q'}))\cup \{ \varphi _{p,q''}(q'')\}_{q''<p\text{ free}}$, for each $q\in I_p$ and each $g\in G_q^*$. Since $q<p$ and $p$ is minimal in $I'$, it follows that we can write
        \[
            g= \sum _{i=1}^n \varphi_{q,q'_i} (g_i)+ \sum_{j=1}^m  \varphi _{q,q''_j}(q''_j),
        \]
       where $q_i'\in \Ireg $, $q_j''\in \Ifree$, $g_i\in G_{q'_i}$, and $q_i',q_j'' <q$, for all $i,j$. Thus we have
       \[
            \varphi_{p,q}  (1,g) = \sum _{i=1}^n\varphi_{p,q_i'} (g_i)+ \sum_{j=1}^m \varphi _{p,q_j''}(q_j'')
            + \varphi _{p,q} (q).
       \]
       
       Hence $G_p$ is generated as a semigroup by 
       \[
        (\cup_{q<p\text{ reg.}}\varphi_{p,q} (G_q))\cup \{ \varphi _{p,q}(q)\}_{q<p\text{ free}},
       \]
       which is a contradiction. This shows the result. 
        
      In particular, a monoid $M$ is super adaptable whenever it is a finitely generated conical refinement monoid, since in that case $\mathbb{P}(\ol{M})$ is finite. For an example in which $\Pfree (\ol{M})$ is artinian but not finite, consider the monoid 
      \[
        M:= \oplus^{\text{o}}_{n\in\NN}(\NN\cup\{0\})=(\oplus_{n\in\NN}  \NN)\sqcup \{(0,0,0,\dots )\},
      \]
      which has infinitely many non-equivalent free primes but all of them are minimal.
        
        \item[(b)] Consider the monoid $M$ given by the presentation
        \[
            M=\left\langle a,b_1,b_2,\cdots \mid         
            a=a+b_i\andSep 
            b_i=b_{i+1}+b_1
            \,\, (i\geq 1)\right\rangle .
        \]
        
        It is readily checked that $M$ is a conical refinement monoid that does not have a finitely generated presentation. One can check that $M$ is super adaptable (we postpone the proof until much later, in \autoref{exa:SupAdapt}).
    \end{itemize}
\end{exas}


\begin{exa}\label{exa:not-super adaptable-Isystem}\label{subsection:exa}
Consider the $I$-system $J = (I, \le , (G_p)_{p\in I},  \varphi_{pq}\,  (q<p))$ where $I=\ZZ$, with the usual order, $\Ifree = I$, and $\Ireg= \emptyset$.  For each $i\in I$, take $G_i= \ZZ_2$, the group with two elements. For $x= (n, \ol{m})\in \NN \times \ZZ_2$, we denote by  $x_i= (n,\ol{m})_i$ the corresponding element in $M_i := \NN \times G_i$. For $i<j\in \ZZ$, define $\varphi_{j,i}\colon M_i= \NN\times \ZZ_2 \to G_j= \ZZ_2$ by
\[
    \varphi _{j,i} ((n,\ol{m})_i) = \ol{m} \in \ZZ_2.
\]

The tuple $J =  (\ZZ, \le , (\ZZ_2)_{i\in \ZZ},  \varphi_{j,i} (i<j))$ is an $I$-system, and by definition all the primes are free in $I$. Observing that $\varphi _{j,i} ((1,\ol{0})_i) = \varphi_{j,i}(p_i)= \ol{0}$ for all $i<j$, we conclude that $J$ is not a super adaptable $I$-system. In particular, $(\mathcal{M}(J), \mathfrak m_J)$ is not a super adaptable marked monoid, because there is a canonical plain isomorphism $\mathcal I (\mathcal{M}(J), \mathfrak m_J) \cong J$ by \autoref{thm:APIsr16}~(b).

However, write $M:=\mathcal{M}(J)$ and consider the mark $\mathfrak m'= \{ (1,\ol{1})_i\}_{i\in \ZZ}$. This has a corresponding system $J'= \mathcal I (M,\mathfrak m')$. Although, by \autoref{lma:isomorphism-of-marked-systems}, $J$ is elementary isomorphic to $ J'$, $J'$ is a super adaptable system, because 
\[
    \varphi ^{\mathfrak m'} _{j,i} ((1,\ol{0})_i) = \varphi _{j,i} ((1,\ol{1})_i) = \ol{1}
\]
for $i<j$.
Hence $\ZZ_2$ is generated by $\varphi _{j,i}^{\mathfrak m'}((1,\ol{0})_i)$ for $i<j$, showing that $J'$ is super adaptable and hence that so is $(M,\mathfrak m')\cong (\mathcal M (J'),\mathfrak m_{J'})$. By definition, $M$ is super adaptable.   
\end{exa}

Next, we exhibit a class of examples which includes both super adaptable and non super adaptable primely generated refinement monoids. 

\begin{exa}\label{exam:not-super-adaptable}
Consider a countable $\ZZ$-system 
\[
    J= (\ZZ, \le, (G_i)_{i\in \ZZ}, \varphi_{ji} (i<j))
\]
constructed as follows:
    
Let $H$ be a countable abelian group. For $i\in \ZZ$, set 
\[
    G_i=\big( \bigoplus_{k\le i} \ZZ_2 \big)\oplus H,
\]
and denote its elements by $((a_i, a_{i-1},a_{i-2},\dots  ), h)$, where $a_j\in \{\ol{0},\ol{1}\}$ (almost all $a_j$ are $\ol{0}$) and $h\in H$. 
	
Now let $\varphi_i \colon \NN \times G_i \to G_{i+1}$ be the surjective semigroup homomorphism defined by
\[
    \varphi_i (n,g) = (\ol{n}, g).
\]

This defines a $\ZZ$-system  $J=(\ZZ, (G_i)_{i\in \ZZ}, \varphi_{j,i}(i<j))$, where 
\[
    \varphi _{j,i} = \widehat{\varphi}_{j-1}\circ \cdots \circ \widehat{\varphi}_{i+1}\circ \varphi_i \colon \NN\times G_i\longrightarrow G_j, 
\]
and where $\widehat{\varphi_i}\colon \ZZ\times G_i \to G_{i+1}$ is the induced group homomorphism described by the same formula $\widehat{\varphi}_i (n,g) = (\ol{n}, g)$.
	
Set $M:= \mathcal M (J)$, which is a primely generated conical refinement monoid by \autoref{thm:APIsr16}~(b). In what follows, we will see that $M$ is super adaptable if and only if $H$ is $2$-divisible. In particular, $M$ is not super adaptable when $H=\ZZ_2$ or $H=\ZZ$, and $M$ is super adaptable when $H=\QQ$. 
	
Suppose first that $M$ is super adaptable. Then, there is a mark $\mathfrak m'$ on $M = \mathcal M (J)$ which makes the $\ZZ$-system $\mathcal I (M,\mathfrak m')$ a super adaptable system. Note that necessarily we have $\mathfrak m_i'= (1, g_i)$ with $g_i\in G_i$. In particular, we get
\[
    G_0 = \langle \varphi_{0,-i} (1,g_{-i})\mid i\in \NN \rangle .
\]
    
For $i\ge 0$, write $x_{-i} = (\ol{0},\ol{0},\dots , \ol{1}, \ol{0}, \ol{0} ,\dots )\in \oplus _{k\le 0} \ZZ_2$, where $\ol{1}$ is in the $-i$ position of the direct sum. Note that 
\begin{align*}
			\varphi_{0,-i} (1,g_{-i}) & = \widehat{\varphi}_{-1}\circ \cdots \circ \widehat{\varphi} _{-i+1} \circ \varphi _{-i} (1,g_{-i})\\
			& = \widehat{\varphi}_{-1}\circ \cdots \circ \widehat{\varphi} _{-i+1} (\ol{1},g_{-i}) \\
			& = \widehat{\varphi}_{-1}\circ \cdots \circ \widehat{\varphi} _{-i+1} (0,\ol{1},g_{-i})\\
			& = \widehat{\varphi}_{-1}\circ \cdots \circ \widehat{\varphi} _{-i+2} (\ol{0},\ol{1},g_{-i}) \\
			& = \widehat{\varphi}_{-1} (0,\ol{0},\dots, \ol{0},\ol{1}, g_{-i}) \\
			& = (\ol{0},\ol{0},\dots ,\ol{0},\ol{1},g_{-i}) \\
                &= (x_{-i+1}+z_{-i}, h_{-i})
\end{align*}
with
\[
    g_{-i}= (z_{-i}',h_{-i})\in (\oplus_{k\le -i} \ZZ_2)\times H,
    \andSep 
    z_{-i}= (\ol{0},\overset{(i}{\dots} ,\ol{0}, z_{-i}').
\]
	
We want to show that $H$ is $2$-divisible. Take any element $h$ in $H$. Then, there must exist non-negative integers $a_{1},\dots , a_N$ for some $N\ge 0$ such that
    \[
    ((\ol{0},\ol{0},\dots ), h) =  \sum_{i=1}^N a_i (x_{-i+1}+z_{-i},h_{-i}).
    \]
    
If some $a_i$ is odd, we obtain a contradiction in the above equality. Hence, each $a_i$ is of the form $a_i= 2b_i$ for some non-negative integer $b_i$, and thus $h= 2(\sum_{i=1}^N b_ih_{-i})\in 2H$, showing that $H$ is $2$-divisible.  
	
Conversely, assume that $H$ is $2$-divisible. We want to build a mark $\mathfrak m'$ so that the $\ZZ$-system $\mathcal I (M, \mathfrak m')$ is super adaptable. Let $(h_i)_{i\in \NN}$ be a family of semigroup generators of $H$ such that each $h_i$ is repeated infinitely many times along the sequence. 
	
Take $g_i = {\bf 0} =  ((\ol{0},\ol{0},\dots ), 0_H)\in G_i$ for $i\ge 0$, and $g_{-i} = ((\ol{0},\ol{0},\dots ), h_i)\in G_{-i}$ for $i>0$. Now the mark $\mathfrak m'$ is defined by $\mathfrak m'_i= (1,g_i)$ for all $i\in \ZZ$. Since 
\[
     G_i = \langle \varphi _{i,i-1} (1,{\bf 0}) , \widehat{\varphi}_{i,i-1} (G_{i-1})\rangle ,
\]
we see that it is enough to show that 
\[
    G_{-i} = \langle \{ \varphi _{-i,-j} (1, g_{-j}) \mid j>i \} \rangle
\]
for all $i\ge 0$. 
    
We will only prove that $G_0 = \langle \{ \varphi _{0,-j} (1, g_{-j}) \mid j>0 \} \rangle $, but a  similar argument ---using that $(h_j)_{j>i}$ is a family of generators of $H$ for $i>0$--- gives the equality for $G_{-i}$. Set $S= \langle \{ \varphi _{0,-j} (1, g_{-j}) : j>0 \} \rangle $. By the above computation, we have $\varphi _{0,-j} (1,g_{-j}) = (x_{-j+1},h_j)$ for all $j>0$. Since $H$ is $2$-divisible, for every $h\in H$ there exist $0<j_1<j_2<\cdots < j_r$ and positive integers $a_1,a_2, \dots, a_r$ such that
\[
    h= 2a_1 h_{j_1} +2a_2h_{j_2}+ \cdots +2a_rh_{j_r} .
\]
Thus, we get 
\[
    ((\ol{0},\ol{0},\dots ), h) = 2a_1(x_{-j_1+1},h_{j_1})+ 2a_2(x_{-j_2+1},h_{j_2})+\cdots + 2a_r(x_{-j_r+1},h_{j_r}).
\]
This shows that $0\times H \subset S$. But now, since $(0,-h_j), (x_{-j+1},h_j)\in S$ for all $j>0$, we get also that $(x_{-j+1},0_H)\in S$ for all $j>0$, and thus $G_0=S$, as desired. This concludes the proof. 
	\end{exa}

\section{Lifting morphisms}
\label{subsect:lifting-morphisms}
In this second section we prove \autoref{prp:morphism-of-systems}, which shows that we can lift certain monoid morphisms to elementary homomorphisms of systems. A consequence of this result is \autoref{cor:morphisms-inducing-plain-homos}, which applies to plain homomorphisms and will be used in the proof of our main theorem.

Recall that, for $p\in I$, we denote by $e_p$ the neutral element in $G_p$.

Let $(M^{(1)},\mathfrak m_1)$ and $(M^{(2)},\mathfrak m_2)$ be two marked, primely generated conical refinement monoids. 
Whenever $f\colon \mathcal I (M^{(1)},\mathfrak m_1)\to \mathcal I (M^{(2)},\mathfrak m_2)$ is an elementary homomorphism of systems in the sense of \autoref{dfn:hom-of-systems}, we have a monoid homomorphism 
\[
    \mathcal M (f)\colon \mathcal M (\mathcal I (M^{(1)},\mathfrak m_1))\to \mathcal M (\mathcal I (M^{(2)},\mathfrak m_2))
\]
such that $\mathcal M (f)(\chi_p(x)) = \chi_{\psi (p)} (\ol{f_p}(x))$ for $x\in M_p^{(1)}$; see \autoref{lma:morphisms-of-systems-induce-homos}.


Therefore, suitably composing the above homomorphism with the canonical isomorphisms $\Phi_{(M^{(i)},\mathfrak m_i)}$ from \autoref{thm:APIsr16}~(a), we obtain a homomorphism of monoids
\[
    g :=   \Phi^{-1}_{(M^{(2)},\mathfrak m_2)} \circ \mathcal M (f) \circ \Phi_{(M^{(1)},\mathfrak m_1)}\colon M^{(1)} \to M^{(2)}.
\]

 Observe that, if $f$ is a plain homomorphism of systems, then, since all maps $f_p\colon G^{(1)}_p \to G^{(2)}_{\psi (p)}$ are group homomorphisms and $f$ is plain, the map $\mathcal M (f)$ is indeed a homomorphism of \emph{marked} monoids
 \[
    \mathcal M (f)\colon (\mathcal M (\mathcal I (M^{(1)},\mathfrak m_1)), \mathfrak m_{\mathcal I (M^{(1)},\mathfrak m_1)})\to (\mathcal M (\mathcal I (M^{(2)},\mathfrak m_2)), \mathfrak m_{\mathcal I (M^{(2)},\mathfrak m_2)}).
 \]
Hence, recalling that $\Phi_{(M^{(i)},\mathfrak m_i)}$ are isomorphisms of marked monoids we obtain that 
\[
    g =   \Phi^{-1}_{(M^{(2)},\mathfrak m_2)} \circ \mathcal M (f) \circ \Phi_{(M^{(1)},\mathfrak m_1)}\colon (M^{(1)}, \mathfrak m_1) \to (M^{(2)},\mathfrak m_2)
\]
is a homomorphism of marked monoids, that is, $g(\mathfrak m_1)\subseteq \mathfrak m_2$.

Our aim here is to characterize the monoid homomorphisms $g\colon M^{(1)}\to M^{(2)}$ that appear in the above form, as follows:

\begin{prp}
	\label{prp:morphism-of-systems}
	Let $g\colon M^{(1)}\to M^{(2)}$ be a monoid homomorphism between two primely generated, conical, refinement monoids. Then, the following conditions are equivalent:
\begin{itemize} 
	\item[(i)] $g$ is such that
    \begin{itemize}
        \item[(a)] it maps free primes to free primes, and
        \item[(b)] it maps regular primes to regular primes, and 
        \item[(c)] $g(p) <^* g(q)$ whenever $p<^*q$ in $M^{(1)}$.
    \end{itemize}
\item[(ii)] For any marks $\mathfrak m_i$ on $M^{(i)}$, $i=1,2$, there exists an elementary homomorphism of systems $f\colon \mathcal I (M^{(1)},\mathfrak m_1)\to \mathcal I (M^{(2)},\mathfrak m_2)$ such that 
\[
    g= \Phi^{-1}_{(M^{(1)},\mathfrak m_1)}\circ \mathcal M (f)\circ \Phi_{(M^{(2)},\mathfrak m_2)}.
\] 
	\item[(iii)] There exist marks $\mathfrak m_i$ on $M^{(i)}$, for $i=1,2$ and an elementary  homomorphism of systems
	$f\colon \mathcal I (M^{(1)},\mathfrak m_1)\to \mathcal I (M^{(2)},\mathfrak m_2)$ such that
\[
    g=\Phi^{-1}_{(M^{(2)},\mathfrak m_2)} \circ \mathcal M (f) \circ \Phi_{(M^{(1)},\mathfrak m_1)}.
\]
\end{itemize}
\end{prp}
\begin{proof}
	(i)$\implies$(ii). Assume that (i) holds and let $\mathfrak m_i$ be marks on $M^{(i)}$ ($i=1,2$). Let us denote the associated systems by
	\[
	J_i:= \mathcal{I}(M^{(i)},\mathfrak m_i)= (\mathbb{P}(\ol{M^{(i)}}),\leq , (G_p^{(i)})_p, \varphi^{\mathfrak m_i}_{p,q}),\quad i=1,2.
    \]
    
	For $p\in \Pfree (\ol{M^{(i)}})$, we denote by $\tilde{p}$ the unique element in $\mathfrak m_i$ such that 
	$\pi_{M^{(i)}}(\tilde{p}) =p$. 

	The conditions on $g$ ensure that there exists a poset morphism $\psi$ from $\mathbb{P}(\ol{M^{(1)}})$ to $\mathbb{P} (\ol{M^{(2)}})$, defined by $\psi (\ol{r}) = \ol{g(r)}$ for $r\in \mathbb P (M^{(1)})$,  that preserves strict inequalities and sends free/regular elements to free/regular elements. In order to get an elementary homomorphism of systems, we need to define suitable group homomorphisms $f_p\colon G^{(1)}_p\to G^{(2)}_{\psi (p)}$ for all $p\in \mathbb{P}(\ol{M^{(1)}})$, and suitable elements $h(p)\in G^{(2)}_{\psi (p)}$ for all $p\in \Pfree (\ol{M^{(1)}})$.

If $p$ is a regular prime in $\ol{M^{(1)}}$, then $M_p^{(1)}=G_p^{(1)}$ and $M_{\psi(p)}^{(2)}=G^{(2)}_{\psi (p)}$ are abelian groups, and clearly $g(G_p^{(1)}) \subseteq  G_{\psi (p)}^{(2)}$, hence we set $f_p:= g|_{M^{(1)}_p}$. If $p$ is a free prime in $\ol{M^{(1)}}$, the elements of $G^{(1)}_p$ are of the form $(\tilde{p}+a)-\tilde{p}$ in the Grothendieck group of the archimedean component $M_p^{(1)}$, where $a\in M^{(1)}$ and $\tilde{p}+a\le \tilde{p}$. Now observe that 
\[
    g(\tilde{p}) + g(a) = g (\tilde{p} +a)\leq g(\tilde{p}) \in  M^{(2)}_{\psi (p)},
\]
and clearly $g(\tilde{p})$ is a representative of $\psi (p)\in \Pfree (\ol{M^{(2)}})$. It follows that the element
\[
    f_p ((\tilde{p} +a)-\tilde{p}) := (g(\tilde{p})+g(a))- g(\tilde{p})
\]
is a well-defined element in $G^{(2)}_{\psi (p)}$. It is easy to check that $f_p ((\tilde{p} +a)-\tilde{p})$ does not depend on the particular representative $\tilde{p}$ of $p$, and that $f_p\colon G^{(1)}_p\to G^{(2)}_{\psi (p)}$ is a group homomorphism.  We now define the element $h(p)\in G_{\psi (p)}^{(2)}$. Observe that 
\[
    \pi _{M^{(2)}} (\widetilde{\psi (p)} ) = \psi (p)  = \pi_{M^{(2)}}(g(\tilde{p})),
\]
hence $\widetilde{\psi (p)} \equiv g(\tilde{p})$, and thus 
\[
    h(p):=g(\tilde{p}) - \widetilde{\psi (p)}
\]
is a well-defined element in $G^{(2)}_{\psi (p)}$. 
			 
Let $f\colon J_1\to J_2$ be the elementary homomorphism of systems with data 
\[
    (\psi, \{f_p\}_{p\in \mathbb P (\ol{M^{(1)}})}, \{ h(p)\}_{p\in \Pfree (\ol{M^{(1)}})}),
\]
and let $\mathcal M (f)\colon \mathcal M (J_1)\to \mathcal M (J_2)$ be the associated monoid homomorphism. We need to show that $g'=g$, where $g':= \Phi^{-1}_{(M^{(2)},\mathfrak m_2)}\circ \mathcal M (f)\circ \Phi_{(M^{(1)},\mathfrak m_1)} $. Since $M^{(1)}$ is primely generated, it suffices to check that both morphisms agree on the prime archimedean components. This is obvious for regular primes, so let us assume that $p$ is a free prime in $\ol{M^{(1)}}$, and let $n\in \NN$ and $x= (\tilde{p}+a)-\tilde{p}\in G^{(1)}_p$, where $a\in M$ and $\tilde{p}+a\le \tilde{p}$. Then we have
			 \begin{align*}
			 	g' (n\tilde{p} +a) & = \Phi^{-1}_{(M^{(2)},\mathfrak m_2)}( \mathcal M (f)( n, x)) \\
			 	& = \Phi^{-1}_{(M^{(2)},\mathfrak m_2)} \ol{f_p}( n, x)\\
                &= \Phi^{-1}_{(M^{(2)},\mathfrak m_2)}( n, f_p(x)+ n h(p)) \\
			 	& =   \Phi^{-1}_{(M^{(2)},\mathfrak m_2)} ( n, (g(\tilde{p})+g(a)) - g(\tilde{p})+ n ( g(\tilde{p}) -\widetilde{\psi (p)})) \\
			 	& = n\widetilde{\psi(p)} + (g(\tilde{p})+g(a)) - g(\tilde{p})+ n (g(\tilde{p})-\widetilde{\psi (p)})\\
			 	& = ng(\tilde{p}) + g(a)\\
                &= g(n\tilde{p}+a) .
		\end{align*} 
	This ends the proof of (i)$\implies $(ii).
	
	The implication (ii)$\implies$(iii) is obvious and (iii)$\implies$(i) follows from an application of  \autoref{lma:morphisms-of-systems-induce-homos}. 
	\end{proof}

In some relevant cases, we can prove a version of \autoref{prp:morphism-of-systems} for \emph{plain} homomorphisms of systems:

\begin{cor}
	\label{cor:morphisms-inducing-plain-homos}
	Let $g\colon M^{(1)}\to M^{(2)}$ be a monoid homomorphism between two primely generated, conical, refinement monoids. Suppose that $g$ is such that 
    \begin{itemize}
        \item[(a)] it maps free primes to free primes, and
        \item[(b)] it maps regular primes to regular primes, and 
        \item[(c)] $g(p) <^* g(q)$ whenever $p<^*q$ in $M^{(1)}$, and
        \item[(d)] the induced map
    \[
        \psi \colon \mathbb{P}(\ol{M^{(1)}}) \to \mathbb{P} (\ol{M^{(2)}}),\quad \psi (\ol{p}) = \ol{g(p)}
    \]
    is injective.
    \end{itemize}
    
    Then, there are marks $\mathfrak m_i$ on $M^{(i)}$, $i=1,2$, and a plain homomorphism 
	$f\colon \mathcal I (M^{(1)},\mathfrak m_1)\to \mathcal I (M^{(2)},\mathfrak m_2)$ such that
    \[
        g= \Phi^{-1}_{(M^{(2)},\mathfrak m_2)} \circ \mathcal M (f) \circ \Phi_{(M^{(1)},\mathfrak m_1)}\andSep 
        g(\mathfrak m_1)\subseteq \mathfrak m_2.
    \]
    
	In particular, any isomorphism of monoids $g\colon M^{(1)}\to M^{(2)}$ is induced by a plain isomorphism of systems. 
\end{cor}

\begin{proof}
	Suppose that the conditions in the statement are fulfilled. Take any mark $\mathfrak m_1$ of $M^{(1)}$. Since 
	$\psi$ is injective, we may choose a mark $\mathfrak m_2$ of $M^{(2)}$ such that $g(\mathfrak m_1)\subseteq \mathfrak m_2$. If follows that $\widetilde{\psi(p)} = g(\tilde{p})$ for all free primes $p$, and thus $h(p)= e_{\psi (p)}$, where $h(p)$ are the elements built in the proof of \autoref{prp:morphism-of-systems}. Hence the homomorphism of systems $f$ built in \autoref{prp:morphism-of-systems} is plain. 
\end{proof}


\begin{rmk}
	\label{rmk:Dobbertin} In \cite[Theorem 2]{Dobb84}, Dobbertin established an equivalence of categories between the category of \emph{partial orders of groups}  --—a particular class of I-systems--—   with a restricted class of morphisms, and the category of primely generated regular refinement monoids with a restricted class of monoid homomorphisms. However, since we are interested in different types of morphisms, this result does not play a role in our approach. 
	\end{rmk}

\chapter{From I-systems to separated graphs}\label{sec:FromIsystmToSepGraph}
In this chapter we introduce the class of \emph{super adaptable} (separated) graphs. This provides the main ingredient of the paper, and after proving a number of important properties, such as that any such graph can be written as the direct limit of finite ones (\autoref{prp:SupAdapLim}), we establish \hyperlink{Step2}{Step (2)} from the introduction.

To this end, first, we introduce the separated graph $\mathcal{G}(J)$ of a s.a.~$I$-system $J$ and show in \autoref{prp:constructed-super adaptable} that it is a super adaptable graph. Then, we define the system $\mathcal{I}(E,C)$ of a super adaptable graph and prove that $\mathcal{I}(E,C)$ is a s.a.~system (\autoref{thm:necessary-condition}). This corresponds to the diagram:
\[
\xymatrixrowsep{1pc}
\xymatrixcolsep{5pc}
\xymatrix{
    {\begin{array}{c}
         \text{s.a.}  \\
          I\text{-system}\\
          J
    \end{array}}
    \ar@/^/[r]^-{\mathcal{G}}
    &
    {\begin{array}{c}
         \text{s.a.}  \\
          \text{graph}\\
          (E,C)
    \end{array}}
    \ar@/^/[l]^-{\mathcal{I}}
}
\]

\section{Super adaptable graphs}

We begin the section by recalling the definition of a separated graph from \cite{AraGoo12Crelle}.

\begin{dfn}[{\cite[Definition~2.1]{AraGoo12Crelle}}]\label{dfn:SepGraph}
    A pair $(E,C)$ is a \emph{separated graph}\index[terms]{separated graph} if $E$ is a directed graph and $C=\sqcup_{v\in E^0} C_v$ with $C_v$ a partition of $s^{-1}(v)$, where $C_v=\emptyset$ if $s^{-1}(v)=\emptyset$.

    Further, $(E,C)$ is said to be \emph{finitely separated}\index[terms]{separated graph!finitely separated} if $\vert X\vert$ is finite for each $v\in E^0$ and $X\in C_v$.
\end{dfn}

One thinks of each $C_v$ as a coloring of $s^{-1}(v)$, with each $X\in C_v$ codifying the edges with the same color. For example, a graph with vertices $v_1,v_2,v_3$ and edges $C_{v_1}=\{ v_1\rightarrow v_3\}$, $C_{v_2}=\{ v_2\rightarrow v_1\}$, $C_{v_3}=\{ \{v_3\rightarrow v_3\}, \{v_3\rightarrow v_1\}\}$ should be though of as the following colored graph:

\begin{center}
\begin{tikzpicture}[>=Stealth, scale=1.5]
    \tikzstyle{vertex} = [inner sep=1.5pt]

    \node[vertex] (v1) at (0, 1) {$v_1$};
    \node[vertex] (v3) at (-1, 0) {$v_3$};
    \node[vertex] (v2) at (1, 0) {$v_2$};

    \draw[->, blue, thick, bend left=15] (v1) to (v3);
    
    \draw[->, red, thick] (v2) -- (v1);
    
    \draw[->, orange, thick, in=150, out=210, looseness=6] (v3) to (v3);
    
    \draw[->, green!60!black, thick, bend left=15] (v3) to (v1);
\end{tikzpicture}
\end{center}

\begin{ntn}
\label{ntn:notations-for-separated-graphs}
    Throughout the section, we will make use of the following notation:
    \begin{itemize}
        \item For $C$ as above, we let $C_{\rm fin}$ denote the subset of $C$ that consists on all elements of finite cardinality.
        \item The set of vertices of a directed graph comes equipped with a natural relation $\leq$, given by $v\leq w$ whenever there is a finite directed path from $w$ to $v$. We will denote by $\mathbb{P}(E)$ the poset given by the  antisymmetrization of $(E^0,\leq )$.
        \item A subset $H$ of vertices of a directed graph $E$ is said to be \emph{hereditary} if $v\le w$ and $w\in H$ imply $v\in H$. Note that hereditary subsets of $E^0$ correspond to lower subsets of $\mathbb{P}(E)$.
        
        If $(E,C)$ is a separated graph and $H$ is a hereditary subset of $E^0$, we denote by $(E_H,C^H)$ the restriction of $(E,C)$ to $H$. We thus have $(E_H)^0=H$ and $C^H_v= C_v$ for $v\in H$.  
\item A subset $S$ of vertices of a directed graph $E$ is said to be \emph{saturated} if for each $v\in E^0$ such that $0<|s^{-1}(v)|<\infty$, we have that $v\in S$ whenever $r(e)\in S$ for all $e\in s^{-1}(v)$.     
    \end{itemize}

\end{ntn}

Our notion of super adaptability will make use of \emph{SPI graphs}\index[terms]{graph!SPI}. To define them, recall that a non-empty graph is \emph{strongly connected}\index[terms]{graph!strongly connected} (or \emph{transitive}) if every pair of vertices can be connected by a finite directed path.

A closed path $\mu:=e_1\cdots e_n$ is said to be \emph{simple}\index[terms]{graph!simple} if $s(e_j)\neq s(e_1)$ for all $j>1$. We denote the set of closed simple paths based at $v$ (i.e., $s(\mu)=v$) by $\CSP (v)$. A \emph{cycle}\index[terms]{graph!cycle} is a closed path 
$\mu:=e_1\cdots e_n$ such that $s(e_i)\neq s(e_j)$ for all $1\le i <j\le n$. 

We say that a row-finite graph $E$ is \emph{Simple Purely Infinite} (SPI) if the following conditions hold:
\begin{itemize}
    \item[(a)] the only hereditary and saturated subsets of $E^0$ are $\emptyset$ and $E^0$; and
    \item[(b)] each cycle has an exit; and
    \item[(c)] every vertex connects to a cycle.
\end{itemize}

As shown in \cite[Theorem 3.1.10]{AAS}, $E$ is SPI if and only if the Leavitt path algebra $L_K(E)$\index[symbols]{$L_K(E)\quad$(Leavitt path algebra)} is a simple purely infinite ring (for any field $K$).

\begin{lma}\label{lma:charcSPI}
    Let $E$ be a strongly connected row-finite graph. Then, $E$ is either a single vertex without edges, a single cycle, or a SPI graph.  
    
    Moreover, the following are equivalent:
    \begin{itemize}
        \item[(i)] $E$ is SPI;
        \item[(ii)] $|\CSP (v)| \ge 2$ for some $v\in E^0$;
        \item[(iii)] $| \CSP (v)|\ge 2$ for all $v\in E^0$.
    \end{itemize}
\end{lma}
\begin{proof}
	We begin by proving the second statement. By (the proof) of \cite[Lemma 3.3.8]{AAS}, we have that for any two vertices $v,w$ of $E^0$, $|\CSP (v)|\ge 2$ if and only if $| \CSP (w)|\ge 2$. This gives that (ii) and (iii) are equivalent.
	
	Suppose that $| \CSP (v)| \ge 2$ for all $v\in E^0$ and let us prove (i). We have to check conditions (a), (b), and (c) above. For (a), note that the only hereditary subsets of $E^0$ are $\emptyset$ and $E^0$, because $E$ is strongly connected. For (b), note that each cycle has an exit, because $| \CSP (v)| \ge 2$ for each vertex $v$ in the cycle.  For (c), observe that $E$ must have some cycle, because $E$ has some closed simple path. This proves that $E$ is SPI.
	
	Conversely, if $E$ is SPI, then $E$ contains a cycle $C$. Using condition (b), the cycle $C$ has an exit, say $e$. Then there must be a path $\ell $, which may be chosen of minimal length, from $r(e)$ to $s(e)$, and thus $e\ell$ and $C$ provide two distinct closed simple paths based at $s(e)$. Hence $| \CSP (s(e))| \ge 2$. This shows the desired equivalence.
	
	Now let $E$ be any strongly connected row-finite graph. If $E$ has no edges, then clearly $E$ is a single vertex, because $E$ is strongly connected. If $E$ has some edge, then $E$ must have a cycle, again by strong connectedness. 
	Let $C$ be a cycle in $E$. If $C$ has an exit, then $E$ is SPI by the above argument. Thus, suppose that $C$ has no
	exits. If $w$ is a vertex of $E$ which is not in the cycle $C$, then there must be a directed path from some vertex of the cycle to $w$, so the cycle must have an exit, which is not the case. Hence, every vertex and every edge of $E$ belong to the cycle $C$, and thus $E$ consists of a single cycle $C$. 
\end{proof}

\begin{dfn}\label{dfn:SupAdSepGra}
	Let $(E,C)$ be a finitely separated graph, and let $\mathbb P (E)$ be its induced poset. We say that $(E,C)$ is \emph{super adaptable}\index[terms]{super adaptable graph} if the family of full subgraphs $\{E_p\}_{p\in \mathbb{P}(E)}$ with vertex set
	\[
	E_p^0=\{ w\in E^0\mid p=[w] \}
	\]
	is such that either:
	\begin{itemize}
		\item[(a)] $E_p$ is a row-finite (strongly connected) SPI graph and $\vert C_w\vert =1$ for each $w\in E_p$; or
		\item[(b)] $E_p^0=\{ v^p\}$, with either 
        \begin{itemize}
            \item $s^{-1}(v^p)=\emptyset$; or
            \item $C_{v^p}=\{ X_\alpha^{(p)} \mid \alpha\in E_p^1 \}$ where $|E_p^1|\ge 1$, and  $2\leq \vert X_\alpha^{(p)}\vert<\infty $ and $X_\alpha^{(p)}\cap E_p^1 =\{ 
	\alpha\}$ for each $\alpha\in E_p^1$.
        \end{itemize}
	\end{itemize}
	
	The elements $p\in \mathbb{P}(E)$ which satisfy (a) are called \emph{regular}\index[terms]{vertex!regular}, while those that satisfy (b) are called \emph{free}\index[terms]{vertex!free}. Given $q < p\in \mathbb{P}(E)$ we say that a \emph{connector}\index[terms]{graph!connector} from $p$ to $q$ is any edge $e\in E^1$ such that $s(e)\in E^0_p$ and $r(e)\in E^0_q$; notice that, if $p$ is free, all edges in $X_\alpha^{(p)}\setminus\{\alpha\}$ are connectors for any $\alpha\in E^1_p$.
\end{dfn}

\begin{pgr}[Adaptable graphs]\label{pgr:AdapGraph}
    The notion of \emph{adaptability}\index[terms]{separated graph!adaptable} for separated graphs was introduced in \cite[Definition~1.4]{AraBosPar20SgpFor}. Standard arguments show that the original definition is equivalent to the following: $(E,C)$ is \emph{adaptable} if $\mathbb{P}(E)$ is finite and the family of full subgraphs $\{E_p\}_{p\in \mathbb{P}(E)}$ with vertex set
	\[
	E_p^0=\{ w\in E^0\mid p=[w] \}
	\]
	is such that either:
	\begin{itemize}
		\item[(a)] $E_p$ is a row-finite (strongly connected) graph with $\vert C_w\vert =1$ and $\vert s^{-1}_{E_p}(w)\vert \geq 2$ for each $w\in E_p$; or
		\item[(b)] $E_p^0=\{ v^p\}$, with either 
        \begin{itemize}
            \item $s^{-1}(v^p)=\emptyset$; or
            \item $C_{v^p}=\{ X_\alpha^{(p)} \mid \alpha\in E_p^1 \}$ with $1\le \vert E_p^1\vert<\infty$, $2\leq \vert X_\alpha^{(p)}\vert<\infty $ and $X_\alpha^{(p)}\cap E_p^1 =\{ 
	\alpha\}$ for each $\alpha\in E_p^1$.
        \end{itemize}
	\end{itemize}

    Using \autoref{lma:charcSPI} (and the fact that the subgraphs $E_p$ are strongly connected), one sees that any adaptable graph is super adaptable. However, we note that there exist finite super adaptable graphs that are not adaptable. For example, consider the SPI graph
    
    
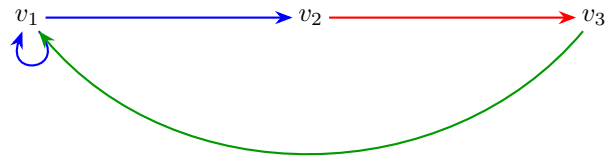
\begin{center}
    \begin{tikzpicture}[>=Stealth, scale=1.5]
    \tikzstyle{vertex} = [inner sep=2pt]

    \node[vertex] (v1) at (0,0) {$v_1$};
    \node[vertex] (v2) at ([xshift=2.5cm]v1) {$v_2$};
    \node[vertex] (v3) at ([xshift=2.5cm]v2) {$v_3$};

    \draw[->, blue, thick] (v1) -- (v2);
    \draw[->, blue, thick, in=-110, out=-50, looseness=8] (v1) to (v1);
    
    \draw[->, red, thick] (v2) -- (v3);
    
    \draw[->, green!60!black, thick, bend left=50] (v3) to (v1);
    \end{tikzpicture}
    \end{center}
    with colouring $C_{v_i}=\{ s^{-1}(v_i)\}$ for each $i$. 

    It follows from \autoref{dfn:SupAdSepGra} and \autoref{lma:charcSPI} that the resulting separated graph is super adaptable. However, it is not adaptable. Indeed, $\vert s^{-1}(v_2)\vert=\vert s^{-1}(v_3)\vert =1$ which shows that neither (a) nor (b) above are satisfied.
\end{pgr}

We now move to the study of morphisms between super adaptable graphs. First, let us recall the definition of morphism of separated graphs from \cite[Definition~3.2]{AraGoo12Crelle}.

\begin{pgr}[Morphisms in $\rm SGr$]
\label{pgr:morphisms-in-SGr}
    Recall that a \emph{morphism}\index[terms]{separated graph! morphism} $\phi\colon (E,C)\to (F,D)$ between separated graphs is a graph homomorphism $\phi = (\phi^0,\phi^1) \colon E\to F$ such that
    \begin{enumerate}
        \item[(M1)] $\phi^0$ is injective and $\phi^1\vert_X$ is injective for every $X\in C$.
        \item[(M2)] There exists $\tilde{\phi}\colon C\to D$ such that $\phi^1(X)\subseteq \tilde{\phi}(X)$ for every $X\in C$.
        \item[(M3)] $\tilde{\phi}(C_{\rm fin})\subseteq D_{\rm fin}$ and for every $X\in C_{\rm fin}$ the map $\phi^1\vert_X\colon X\to \tilde{\phi}(X)$ is bijective.
    \end{enumerate}

    We will denote by $\rm SGr$\index[terms]{separated graph!category}\index[symbols]{$\rm SGr\quad$(separated graph category)} the category of separated graphs and their morphisms.
    
    Further, for a finitely separated graph $(E,C)$, a \emph{complete} subobject \index[terms]{separated graph!complete subobject} of $(E,C)$ will be a separated graph $(F,D)$ such that $F$ is a subgraph of $E$ and the inclusion map $i\colon F \to E$ induces a morphism of separated graphs
    $i\colon (F,D)\to (E,C)$. This amounts to say that, for each $v\in F^0$, we have 
    \[
        D_v= \{ Y\in C_v \mid  Y\cap F^1\ne \emptyset  \}.
    \]
\end{pgr}

\begin{prp}\label{prp:SuperGr}
    Let $(E,C)$ and $(F,D)$ be super adaptable graphs, and let $\phi\colon (E,C)\to (F,D)$ be a morphism. Then, the poset morphism $\rho\colon \mathbb{P}(E)\to \mathbb{P}(F)$ induced by $\phi$ satisfies:
    \begin{itemize}
        \item[(a)] $\phi (E_p)\subseteq F_{\rho (p)}$ for each $p\in\mathbb{P}(E)$;
        \item[(b)] $\rho (\Preg (E))\subseteq \Preg (F)$.
    \end{itemize}
\end{prp}
\begin{proof}
    Let $E_p$ and $F_q$ be the full subgraphs with vertex sets $\{ w\in E^0\mid p=[w] \}$ and $\{ v\in F^0\mid q=[v] \}$ respectively.

    Take $p\in\mathbb{P}(E)$ and note that, since $E_p$ is strongly connected, so is $\phi (E_p)$. Thus, we must have $\phi (E_p)\subseteq F_{\rho (p)}$. Now assume that $p$ is regular. If $\vert E_p^0\vert >1$, we have $\vert F_{\rho (p)}^0\vert \geq \vert \phi^0  (E_p^0)\vert\geq 2$ because $\phi^0$ is injective. This implies that $\rho (p)$ is regular. If $\vert E_p^0\vert =1$, let $w$ be the vertex in $E_p$. Then, property (a) in \autoref{dfn:SupAdSepGra} implies that $C_w=\{ X_w\}$ where $X_w$ contains at least two distinct $w$-loops (\autoref{lma:charcSPI}). Since $\phi^1$ is injective on any set of $C$, the set $\tilde{\phi} (X_w)\in D_{\phi^0 (w)}$ contains two distinct loops at $\phi^0 (w)$. By definition, $\rho ([w])=\rho (p)$ cannot be free. Thus, it is regular.
\end{proof}

Recall that the category $\rm SGr$  of separated graphs is closed under \emph{arbitrary} directed limits \cite[Proposition~3.3]{AraGoo12Crelle}.

\begin{prp}\label{prp:LimitSuperAdapt}
    The full subcategory ${\rm SGr}_{sa}$\index[symbols]{${\rm SGr}_{sa}$} of super adaptable graphs is closed under direct limits in $\rm SGr$.
\end{prp}
\begin{proof}
    Let $(E,C)$ be the $\rm SGr$-limit of a system $((E_j , C_j)_j,\phi_{j,k})$ with $(E_j,C_j)$ super adaptable for each $j$. Denote by $\phi_j$ the limit maps $(E_j,C_j)\to (E,C)$.

    Let $\mathbb{P}(E)$ and $\mathbb{P}(E_j)$ denote the induced posets of $E$ and $E_j$ respectively. For $p\in \mathbb P (E)$,  let $E_p$ be the canonical full subgraphs from \autoref{dfn:SupAdSepGra}. We will show that each $E_p$ satisfies either (a) or (b) in that definition.
    
    Set 
    \[
        E^0_{\rm reg} := \{
            w\in E^0 \mid 
            \exists w'\in E^0_j \text{ such that }
            w=\phi_j^0 (w') \text{ and }
            [w']\in \Preg (E_j)
        \}.
    \]
    Let $p\in\mathbb{P}(E)$. Assume first that $p=[w]$ for some $w\in E^0_{\rm reg}$. We show that $p$ satisfies (a) in \autoref{dfn:SupAdSepGra}. Let $w'\in E_j^0$ be as above. \autoref{prp:SuperGr} implies that $[\phi_{k,j}(w')]$ is regular for every $k\geq j$. Take any $v\in E_p^0$, and let $\gamma$ be a closed path in $E$ with $r(\gamma)=s(\gamma) = w$ and such that $\gamma$ passes through $v$. Then, there exist some $i\ge j$ and a path $\gamma'$ in $E_i$ such that $\phi_i(\gamma')= \gamma$. Using injectivity of the map $\phi_i^0$, we get $\phi_{i,j}^0(w') = s(\gamma') = r(\gamma')$. We can write $\gamma' = \gamma'_1\gamma'_2$, with $\phi_i^0 (v')=v$, where $v'=r(\gamma_1')= s(\gamma'_2)$. Now since $[\phi_{i,j}^0(w')]$ is regular, and $[v']= [\phi_{i,j}^0(w')]$, it follows that $[v']$ is regular and thus that $v= \phi_i^0(v')\in E^0_{\rm reg}$. It follows that $v\in E^0_{\rm reg}$ for any $v\in E_p$.
    
    Fix $v\in E_p^0$ and find $i\in\NN$ and $v'\in E^0_i$ such that  $v=\phi_i^0 (v')$ and $[v']\in \Preg (E_i) $. By \autoref{prp:SuperGr}, $\phi_{k,i}(v')$ has a regular class for each $k\geq i$ and, in particular, $\vert (C_k)_{\phi_{k,i}(v')}\vert =1$ for all $k\ge i$. This gives that $\vert C_v\vert =1$. Moreover, the unique set $s^{-1}_E(v)$ in $C_v$ is in bijective correspondence (through $\phi_i$) with the finite non-empty set $s^{-1}_{E_i}(v')$. In particular, this implies that $E_p$ is a row-finite and strongly connected graph with at least one edge.

    Further, $E_p$ cannot consist of a single cycle. Indeed, given $v\in E_p^0$ and $v'\in E_i^0$ as above, we know that $v'$ lies in the strongly connected SPI full subgraph $(E_i)_{[v']}$ of $E_i$. Thus, \autoref{lma:charcSPI} implies the existence of two edges in this component that share their source (and necessarily have the same color). Since $\phi_j^1$ is injective on edges with the same color and $\phi_i^0$ maps all of $(E_i)_{[v']}^0$ into $E_p^0$, it follows that $E_p$ has at least one vertex that is the source of two edges. This shows that $E_p$ cannot be a single cycle. By \autoref{lma:charcSPI}, $E_p$ must be a SPI graph. In short, if $w\in E^0_{\rm reg}$, then $E_{[w]}$ satisfies (a) in \autoref{dfn:SupAdSepGra}.

    Now, take $p=[w]$ with $w\not\in E^0_{\rm reg}$. It follows that all the elements in $E_p^0$ come from free elements that remain free throughout the system (\autoref{prp:SuperGr}). We will show that in this situation the full subgraph $E_p$ satisfies (b) in \autoref{dfn:SupAdSepGra}.

    First, note that $E_p^0$ only contains one element. Indeed, assume that $E_p^0$ contains two vertices. By definition, these two vertices would be connected by a closed simple path, which could be realized in some super adaptable graph $(E_j, C_j)$. This whole path ---in particular, the two distinct vertices--- is contained in a free component of $(E_j, C_j)$, which by definition only has one vertex. This leads to a contradiction. Hence $E_p^0$ is a singleton. 
    
    Now let $v^p$ be the vertex in $E_p$. If $s^{-1}(v^p)=\emptyset$, we are done. Else, we need to prove that $C_{v^p}=\{ X_{\alpha}^{(p)} \mid \alpha\in E_p^1 \}$ with $2\leq \vert X_\alpha^{(p)}\vert <\infty$ and $X_\alpha^{(p)}\cap E_p^1 =\{\alpha\}$.

    Take $X\in C_{v^p}$. First, let us prove that it contains at most one loop. Thus, let $e,f\in X\cap E_p^1$. By definition of the limit partition \cite[Proposition~3.3]{AraGoo12Crelle}, this happens if $e,f$ come from edges $e',f'\in X'$, where $X'\in (C_j)_{v_j}$ for some $j\in\NN$ and $v_j\in E_j^0$ such that $\phi_j^0(v_j)=v^p$; notice that $e', f'$ must be loops, because $\phi_j^0$ is injective. By our assumptions $[v_j]$ is free and so, since $e',f'$ belong to the same element of the partition, we must have $e'=f'$. This implies $e=f$, as required.

    Since each $X\in C_{v^p}$ comes from a partition at a finite stage associated to a free element, we get $2\leq \vert X\vert <\infty$. This shows that $(E,C)$ is super adaptable, as required.
\end{proof}

We can now show that each super adaptable graph is a direct limit of finite super adaptable graphs.

\begin{prp}\label{prp:SupAdapLim}
 Any super adaptable graph $(E,C)$ is the direct limit in {\rm SGr} of finite complete super adaptable subgraphs $(F,D)$ with the property that the associated map $\rho \colon \mathbb P (F)\to \mathbb P (E)$ sends each non-minimal free element $p\in \Pfree (F)$ to a free element $\rho (p) \in \Pfree (E)$.
 
 Hence any countable super adaptable graph is the direct limit  in {\rm SGr} of a sequence of finite super adaptable graphs with the above property. 
\end{prp}
\begin{proof}
Let $(E,C)$ be a super adaptable graph. By \cite[Proposition 3.5]{AraGoo12Crelle}, $(E,C)$ is the direct limit of finite complete subobjects $(F,D)$. Thus, it suffices to construct, for any given finite complete subobject $(F,D)$ of $(E,C)$, another finite complete suboject $(G,D')$ of $(E,C)$ such
 that $F$ is a subgraph of $G$ and $(G,D')$ is a super adaptable graph. 

Fix a finite complete subobject $(F,D)$ of $(E,C)$. Let $\mathbb{P}(E)$ and $\mathbb{P}(F)$ be the posets associated to $E$ and $F$ respectively. Since $F^0$ is finite, $\mathbb{P}(F)$ is also finite. Let $\rho \colon \mathbb{P}(F)\to \mathbb{P}(E)$ be the map from \autoref{prp:SuperGr}, defined by the property that $[w]_E = \rho (p)$ whenever $w\in F_p^0$. 
Let $p\in \mathbb{P}(F)$, and recall that each graph $F_p$ is strongly connected. 
In particular, if $F_p^0$ contains a vertex $w$ that is a sink in $F$, then necessarily $F_p^0= \{w\}$.

Else, if $F_p^0$ does not contain a sink in $F$, we consider two cases. If $\rho (p)\in \Pfree (E) $, then $F_p$ must satisfy (b) in \autoref{dfn:SupAdSepGra} (with respect to the coloring $D$) since $(F,D)$ is a complete subobject of $(E,C)$. Suppose now that $\rho (p)\in \Preg (E)$. 
We then have that $C_v = \{ s_E^{-1}(v)\}$ and $1\le |s_E^{-1}(v)| <\infty$ for all $v\in E_{\rho(p)}$.  
If $F_p$ is SPI, then, since $(F,D)$ is a complete subobject of $(E,C)$, it follows that $F_p$ satisfies (a) in \autoref{dfn:SupAdSepGra} with respect to $D$, and moreover $s_F^{-1}(v)= s_E^{-1}(v)$ for all $v\in F_p^0$. Else, it follows from \autoref{lma:charcSPI} that $F_p$ is either a single vertex without edges or a single cycle:
\begin{itemize}
    \item[(a)] If $F_p$ is a single vertex, write $F_p^0= \{w\}$. Since $\rho (p)\in \Preg (E)$ (in particular, since $E_{\rho (p)}$ is SPI), \autoref{lma:charcSPI} proves the existence of two distinct closed simple paths based at $w$ in $E_{\rho(p)}$, which we denote by $c_1 (w)$ and $c_2 (w)$.
    \item[(b)] If $F_p$ consists of a single cycle $c$, we know from \autoref{lma:charcSPI} that for every $w\in F_p^0$ there exists another closed simple path $c'(w)$ in $E_{\rho (p)}$.
\end{itemize}

We set $\Lambda _1 = \bigcup_w (c_1(w)^0\cup c_2(w)^0)$, where $w$ ranges on the non-sinks of $F$ such that their associated class $p=[w]_F$ satisfies (a) above (i.e. $F_p$ is a single vertex). Similarly, we set  $\Lambda _2 = \bigcup_w (c'(w))^0$, where $w$ now ranges on the non-sinks of $F$ such that their associated class $p=[w]_F$ satisfies (b) above. Let $G$ be the subgraph of $E$ with
\[
    G^0 = F^0 \cup \Lambda _1 \cup \Lambda _2 \cup r_E(s_E^{-1}(\Lambda_1\cup \Lambda_2)),\andSep 
    G^1= F^1 \cup   s_E^{-1}(\Lambda_1\cup \Lambda_2).
\]

By construction, $F\subseteq G$. Note that $G$ is a finite graph because $|s_E^{-1}(v)|<\infty$ for all $v\in E^0$ such that $[v]\in \Preg (E)$. 

Let $\mathbb{P} (G)$ be the set of components associated to $G$. By discussing a number of cases separately, we will show that for each $p\in\mathbb{P}(G)$ the full subgraph $G_p$ of $G$ is either a strongly connected SPI graph contained in some regular component of $E$, with $s_G^{-1}(w)=s_E^{-1}(w)$ for all $w\in G_p^0$; or has a single vertex $w$ with $s_G^{-1}(w)=\emptyset$; or $\rho(p)$ is free in $(E,C)$ and $s_G^{-1}(w)=\cup_{X\in L} X$ for some $\emptyset \neq L\subseteq C_{w}$. This will prove that the coloring $D'$ of $G$ induced by $C$ makes $(G,D')$ both a complete subobject of $(E,C)$ and a super adaptable graph. Moreover, the map $\rho \colon \mathbb P (G)\to \mathbb P (E)$ will satisfy that $\rho (p)\in \Pfree (E)$ for each non-minimal $p\in \Pfree (G)$. 

Let $p\in\mathbb{P}(G)$ and $w\in G^0$ such that $p=[w]_G$. If $w$ is a sink in $G$, then $G_p^0=\{w\}$ and $s_G^{-1}(w)=\emptyset$. Thus, assume that $w$ is not a sink in $G$. In this case, the definition of $G$ forces $w$ to be either in $F^0$ or in $\Lambda_1\cup\Lambda_2$.

Assume first that $w\in \Lambda_1\cup\Lambda_2$. Then, it follows by construction that $G$ contains at least two closed simple paths based at $w$. Also, it follows from \autoref{lma:charcSPI}~((ii)$\implies $(i)) that $G_p$ is a SPI graph. By construction, $G_p$ is contained in a regular component of $E$, and $s_G^{-1}(w) =s_E^{-1}(w)$.

Now assume that $w\in F^0\setminus (\Lambda_1\cup \Lambda _2)$. If $w$ is a sink in $F$, then it is also a sink in $G$, and we are done. Otherwise, we have two final cases: If $\rho ([w]_F) \in \Pfree (E) $, then ---as observed above---  $F_{[w]_F}$ satisfies (b) in \autoref{dfn:SupAdSepGra} with respect to the coloring $D$. Since $w$ emits the same edges in $F$ as in $G$, it follows that $s_G^{-1}(w)=\cup_{X\in D_w} X$ where $\emptyset \ne D_w\subseteq C_{w}$. If $\rho ([w]_F) \in \Preg (E) $ and $[w]_G=[w']_G$ for some $w'\in \Lambda_1\cup \Lambda_2$, then we may reduce to the above case. Finally suppose that $\rho ([w]_F) \in \Preg (E) $ and $[w]_G\ne [w']_G$ for every $w'\in \Lambda_1\cup \Lambda_2$. Then we have $p=[w]_G=[w]_F$, $F_{p}$ must be SPI,
 and ---as also observed above--- $G_p=F_p$ is a SPI graph contained in a regular component of $E$, with $s_G^{-1}(v)= s_F^{-1}(v) = s_E^{-1}(v)$ for all $v\in G_p^0=F_p^0$.  

This concludes the case by case study, and the proof.   
\end{proof}

\section{The separated graph of a s.a.~I-system}
Given a countable super adaptable system $J$ (as defined in \autoref{dfn:superadaptableIsystem}), we now associate to it a countable separated graph, denoted by $\mathcal{G}(J)$, and show that the graph is super adaptable. We do not know whether this construction can be replicated at the level of system morphisms, thus obtaining a functor between the category of s.a.~systems and s.a.~graphs. The main difficulty relies on the non-canonical choices made to construct $\mathcal{G}(J)$.

\begin{pgr}[Defining $\mathcal{G}(J)$]\label{pgr:DfnGJ}\index[terms]{super adaptable graph!of a system}\index[symbols]{$\mathcal{G}(J)\quad$(graph of a system)}
\label{prg:construction}
Let us first fix any countable super adaptable $I$-system 
\[
    J = (I,\le, (G_p)_{p\in I}, \varphi_{p,q}).
\]
To construct $\mathcal G(J)$ as a separated graph $(E,C)$, we start by describing all the notation needed for its definition, then define the set of vertices $E^0$, and finally introduce the collection of edges and colors. 

Given a row-finite graph $E$, recall from \cite{AAS,AraMorPar07NonStab} that its \emph{graph monoid}\index[terms]{monoid!of graph}\index[symbols]{$\mathcal{M}(E)\quad$(graph monoid)} $\mathcal{M}(E)$ is 
\begin{equation}
\label{eq:definition-graph-monoid}
\mathcal{M}(E):=
        \left\langle
        a_v\,\, (v\in E^0)
        \mid
        a_v = \sum_{e\in s^{-1}(v)} a_{r(e)}\text{ for each }v\in E^0\setminus \text{Sink}(E)
        \right\rangle
,   
\end{equation}
and, as we will see in \autoref{dfn:monoidsepgraph}, this construction can be generalized to separated graphs.

For each $p\in \Ireg$, we fix a \emph{Szymanski model}\index[terms]{graph!Szymanski model} for $G_p$. This is a countably infinite, strongly connected, SPI row-finite graph $S_p$ such that $\mathcal{M}(S_p) \cong \{0\}\sqcup G_p$ (equivalently, $K_0(C^*(S_p))\cong G_p$). The existence of this graph is guaranteed by \cite{Szy02Ind}, where we note that the construction performed in \cite{Szy02Ind} always produces a countably infinite number of vertices. By definition, the set $\{ a_v\mid v\in S_p^0\}$ is a family of semigroup generators for $G(S_p):= \mathcal{M}(S_p)^*$, although one can often find a smaller subset of vertices that generate $G(S_p)$. For each regular $p$, fix a countable (finite or infinite) non-empty subset $\Sgen_p$ of vertices of $S_p$ such that $\{a_v\mid v\in \Sgen_p\}$ is a family of semigroup generators of $G(S_p)$.

Further, fix a concrete isomorphism 
\[
    \gamma_p \colon G_p \to G(S_p)
\]
for each $p\in \Ireg$, and denote by $\Xgen_p $ the set $\{\gamma_p^{-1} (a_v)\mid v\in \Sgen _p\}$, which is a set of semigroup generators of $G_p$. 

For a given non-minimal prime $p\in I$, take the lower subset $J_p= \{ q\in I : q<p \}$ of $I$.  We consider the following subset of $\sqcup_{q\in J_p}M_q$:
\[
    {\bf X}_p := \left(\bigsqcup _{q\in \Ireg\cap J_p} \Xgen_q \right)\sqcup \{ q \mid q\in \Ifree\cap J_p \},
\]
where, as stated in \autoref{ntn:pMp}, we identify $q\in \Ifree $ with the element $(1,0_{G_q})$ in $M_q$. Note that ${\bf X}_p$ is countable, because $I$ and each $\Sgen_p$ are countable.

Since $J$ is a super adaptable $I$-system, the set of elements $\chi_q ( x)$, with $ x \in {\bf X}_p \cap M_q$, is a countable family of semigroup generators of $\mathcal{M}(J_p)$ (where recall from \autoref{pgr:OIdMJ} that $\mathcal{M}(J_p):=\mathcal{M}(J\mid_{J_p})$).
Indeed, as defined in \autoref{pgr:MJ}, each element of $\mathcal{M}(J_p)$ is a finite sum of elements of the form $\chi_q(z)$ with $z\in M_q$ and $q\in J_p$. 
If $q\in \Ireg$ then $z$ is a finite sum of elements in $\Xgen_q$. If $q\in \Ifree$, then $z= nq+h$, where $n\in \NN$, $h\in G_q$ and, since $J $ is super adaptable, $h$ is a finite sum of two types of elements: $\varphi _{q,q'}(g)$ with $q'<q$ regular and $g\in \Xgen_{q'}$; and $\varphi _{q,q''} (q'')$ with $q''<q$ free. Using this notation, it follows that $\chi_q(z)$ is a sum of elements of the form $\chi_{q'} (g)$ for $g\in \Xgen_{q'}$ and $\chi_{q''}(q'')$ for $q''\in \Ifree$, as desired.  

Now, we  define the directed graph $E$ and certain subgraphs $E_p$, for $p\in I$, as follows (these subgraphs will correspond to the canonical subgraphs in \autoref{dfn:SupAdSepGra} once we identify $I$ with $\mathbb{P}(E)$ in \autoref{prp:constructed-super adaptable}). For the set of vertices $E^0 =\bigsqcup_{p\in I} E_p^0$, we need to distinguish three cases:

\textbf{Case 1.} If $p\in I$ is minimal and regular, let $E_p= S_p$, that is, $E_p^0= S_p^0$ and $s_E^{-1} (S_p^0)= E_p^1= S_p^1$.

\textbf{Case 2.} If $p\in I$ is non-minimal and regular, let $(v_i^p)_{i\in\NN}$ be an enumeration of the vertices of $S_p$ and set $E_p^0 = S_p^0\sqcup \{ \widehat{v_i^p} , (v_i^p)^j \mid i,j\in\NN\}$.

\textbf{Case 3.} If $p\in I$ is free, let $E_p^0= \{v^p\}$ be a single vertex. Further, if $p$ is minimal, set $E_p^1= s_E^{-1}(v^p) = \emptyset$, that is, $v^p$ is a sink in $E$ and in $E_p$.
\vspace{0.2cm}

What is now left to do is to describe the edges emitted in $E$ by vertices in $E_p^0$ whenever $p$ is non-minimal. Thus, fix such a $p$, and consider the set of vertices in $E$ corresponding to ${\bf X}_p$: 
\[
    {\bf V}_p := \left(\bigsqcup_{q\in \Ireg, q<p} \Sgen_q \right) \sqcup \{ v^q\mid q\in \Ifree , q<p\}.
\]
There is a bijection 
\[
    \gamma^p \colon {\bf X}_p \to {\bf V}_p
\]
such that $z=\gamma_q^{-1}(a_{\gamma^p(z)})$ whenever $z\in X_q^{\rm gen}$ for $q\in \Ireg$, $q<p$, and $\gamma^p (q)= v^q$ for $q\in \Ifree$, $q<p$.

The description of the emitted edges is divided between the regular and free case: 

\noindent \textbf{(I)  $p$ regular.} Suppose that $p$ is regular but non-minimal. The vertices $v_i^p$ emit the same edges in $E$ as they do in $S_p$, plus a new edge from $v_i^p$ to $\widehat{v_i^p}$ for each $i$.

The other edges emitted from $E_p$ are precisely those that induce the following relations in $\mathcal{M}(E)$ for each $i\in\NN$ (with $B_i$, $v(i)$ and $v_i^{p-}$ specified below):
\begin{enumerate}
\item [(R1)] $\widehat{v_i^p} =	v_i^p + \left( v_i^p\right)^1 + B_i + v(i) + v_i^{p-}$;
\item[(R2)] $\left( v_i^p \right)^{1} = \widehat{v_i^p} + \left( v_i^p \right)^{1} +\left( v_i^p \right)^{3}$;
\item[(R3)] $\left( v_i^p \right)^{2j} = \left( v_i^p \right)^{2j} + \left( v_i^p \right)^{2j-1}$ for all $j\ge 1$;
\item[(R4)] $\left( v_i^p \right)^{2j+1} = \left( v_i^p \right)^{2j+1} + \left( v_i^p \right)^{2j} + \left( v_i^p \right)^{2j+3}$ for all $j\ge 1$.   
\end{enumerate}

Here, we are identifying the vertices in $E$ with the canonical generators of $\mathcal{M}(E)$ to simplify notation. For example, one should read (R2) as saying that $(v_i^p)^1$ emits an edge to $\widehat{v_i^p}$, $(v_i^p)^1$ and $(v_i^p)^3$.

For each fixed $i$, $v_i^{p-}$ is a non-negative integral combination of vertices in $S_p$ that represents the element $-v_i^p$ in $G(S_p)=\mathcal{M}(S_p)^*$.

To define $B_i$ and $v(i)$, consider the lower subset $J_p= \{ q\in I \mid q<p \}$ of $I$, and recall the definition of the sets ${\bf X}_p$ and ${\bf V}_p$ given above. Fix a bijection $g\colon [1,m_p)\cap\NN \to {\bf X}_p$, where $ [1,m_p)\cap\NN$ is an interval of natural numbers and $m_p\le \omega$.

If $ [1,m_p)\cap\NN$ is finite and $i\ge m_p$, set $B_i=v(i)=0$. Else, for $i\in  [1,m_p)\cap\NN$, we have that $g(i)\in M_q$ for a unique $q<p$, and we set 
\[
    B_i := \sum_{j=1}^\infty b_{j,i}v_j^p,
    \text{ where }
    - \varphi_{p,q} (g(i)) = \sum_{j=1}^\infty b_{j,i} \gamma_p^{-1}(a_{v_j^p}),
\]
with $b_{i,j}\geq 0$ for each $j$ and $b_{i,j}=0$ for all but finitely many $j$'s. In addition, we set 
\[
    v(i) := \gamma^p (g(i)) \in {\bf V}_p \text{ for all }   i\in  [1,m_p)\cap\NN.
\]

We define $E_p$ to be the full subgraph with vertex set $E_p^0$ (note that we have defined all the edges emitted by $E_p^0$, so this makes sense), and we set $C_v=\{s_E^{-1}(v)\}$ for all $v\in E_p^0$ in this case.

\noindent \textbf{(II) $p$ free.} If $p$ is free but not minimal, recall that its vertex set is $E_p^0=\{ v^p\}$. Consider the semigroup $K_p$ of all the elements $y\in \mathcal{M}(J_p)$ such that $\varphi_p (y) = 0_{G_p}$, where recall that $\varphi_p\colon \mathcal{M}(J_p)\to G_p$ is the map induced by the maps $\varphi_{p,q}$ for $q<p$; see \autoref{pgr:OIdMJ}.  As we already observed, ${\bf X}_p$ is a family of semigroup generators of $\mathcal{M}(J_p)$. Now, consider the set of tuples $(n_z)_z\in \oplus^{\text{o}}_{z\in {\bf X}_p}(\NN\cup\{0\})$ such that $\sum_{z\in {\bf X}_p} n_z \chi (z)\in K_p$, where ---in order to simplify the notation--- we write $\chi(z)$ for $\chi_q(z)$ whenever $z\in M_q$ for some $q\in I$. Since ${\bf X}_p$ is countable, the set of all such tuples is also countable.  Let $({\bf n}_i)_{i\in \NN}$ be an enumeration of the above set, where we now write ${\bf n}_i:=(n_{z,i})_{z\in {\bf X}_p}$ and denote the set of all the ${\bf n}_i$'s by ${\bf N}_p$.

Using ${\bf N}_p$, we can now define the edges emitted by $v^p$ and the countable coloring $C_{v^p}$ of $s^{-1}(v^p)$. For each $i\in\NN$, the edges with color $X_i\in C_{v^p}$ are: A loop based at $v^p$; and $n_{z,i}$ edges from $v^p$ to $\gamma^p(z)$ with $z$ ranging in ${\bf X}_p$ and ${\bf n}_i = (n_{z,i})_z\in{\bf N}_p$. Set $C_{v^p}:= \{X_i\}_{i\in \NN}$. 

We remark two facts for future use:
\begin{itemize}
    \item[(a)] First, note that $E_p^1=\{ e\in E^1\mid r(e)=s(e)=v^p\}$ and, by construction, $\vert E_p^1\cap X_i\vert =1$ whenever $p\in I$ is free.
    \item[(b)] Second, for every $z\in{\bf X}_p$ there exists ${\bf n}_i\in{\bf N}_p$ such that $n_{z,i}\neq 0$. Indeed, let $q\in J_p$ be such that $z\in M_q$. Then, $\varphi_{p,q} (z)\in G_p$ and, since $\varphi_p$ is surjective, we can find some element $z'\in \mathcal{M}(J_p)$ such that $-\varphi_{p,q}(z) = \varphi_p (z')$. In other words, the element $\chi_q (z) + z'$ is in $K_p$. By writing $z'$ as a non-negative integral combination of the elements in ${\bf X}_p$, we see that there is some  ${\bf n}_i\in{\bf N}_p$ such that $n_{z,i}\neq 0$, as desired.
\end{itemize}

This completes the construction of the separated graph $\mathcal{G}(J):=(E,C)$ and of the subgraphs $E_p$ for $p\in I$.
\end{pgr}

\begin{thm}\label{prp:constructed-super adaptable}
Let $J$ be a countable super adaptable $I$-system. Then, the associated separated graph $\mathcal{G}(J)$ is a countable, super adaptable row-finite graph. Moreover, the poset $\mathbb{P}(E)$ associated to the graph $\mathcal{G}(J)=(E,C)$ is naturally order-isomorphic to $I$. 
\end{thm}
\begin{proof}
By construction, $\mathcal{G}(J):=(E,C)$ is a countable finitely separated graph. Let $\mathbb{P}(E)$ be the poset associated to $\mathcal{G}(J)$. 

Let $E_p$ be the strongly connected full subgraphs of $E$ defined above, which by definition partition the set of vertices in $E$. Further, note that, if there is a path from some vertex $v\in E_p^0$ to some $w\in E_q^0$, then $p\ge q$ by construction of the edges of $E$. Thus, there is a well-defined and order-preserving map $\delta \colon \mathbb{P}(E)\to I$ sending $[v]$ to $p$ whenever $v\in E_p^0$.

Each $E_p$ contains at least one vertex, so $\delta$ is surjective. Using that all the $E_p$'s are strongly connected in $E$, one sees that $\delta$ is also injective. Finally, let $[v],[w]\in\mathbb{P}(E)$ such that $\delta ([v]) \geq \delta ([w])$. Set $p=\delta ([v])$ and $q= \delta([w])$. If $p=q$ then we already have shown that $[v]=[w]$ in $\mathbb{P}(E)$. Else, if $q<p$ and $p\in \Ireg$, then there is some $i\in \NN$ such that, in the notation used in the construction, the vertex $v(i)$ belongs to $E_q^0$, and hence there is an edge from the vertex $\widehat{v_i^p}\in E_p^0$ to $v(i)$, showing that $[w]=[v(i)]< [v]$. If $q<p$ and $p\in \Ifree$, then there exists a vertex $w'\in {\bf V}_p$ such that $w'\in E_q^0$. As observed, $n_{(\gamma^p)^{-1}(w'),i}\neq 0$ for some $i\in\NN$ and ${\bf n}_i\in{\bf N}_p$, so it follows that there is an edge from $v^p$ to $w'$ in $E$, showing that $[w']= [w]<[v^p]$. Thus, $\delta$ is an order-isomorphism between $\mathbb{P}(E)$ and $I$.

Using this map, we see that the sets $E_p$ coincide (up to changing the index $p\in I$ for $\delta^{-1}(p)$) with the canonical subgraphs in \autoref{dfn:SupAdSepGra}. It follows from their construction that the $E_p$'s satisfy either (a) (if $p\in\Ireg$) or (b) (if $p\in \Ifree$) in that definition, which shows that $(E,C)$ is super adaptable.
\end{proof}

\section{The I-system of a s.a.~graph}
We now proceed with the converse of \autoref{prp:constructed-super adaptable}, that is, we define the system of a super adaptable graph and prove that it is always a super adaptable system.

\begin{pgr}[Defining $\mathcal{I}(E,C)$]
\label{pgr:Isystem-of-superadaptable-graph}\index[terms]{I-system!of separated graph}\index[symbols]{$\mathcal{I}(E,C)\quad$(I-system of separated graph)}
Let $(E,C)$ be a super adaptable graph. Let $\mathbb{P}(E)$ denote the poset induced by $E$, and let $\mathbb{P}(E)=\Pfree (E)\sqcup \Preg (E)$ be the canonical partition of the set (\autoref{dfn:SupAdSepGra})\index[symbols]{$\Pfree (E)\quad$ (free vertices of $(E,C)$)}\index[symbols]{$\Preg (E)\quad$ (regular vertices of $(E,C)$)}. For a fixed $p\in \mathbb{P}(E)$, we define $G_p$ as follows:
\begin{enumerate}
\item If $p\in \Pfree (E)$ is minimal, $G_p=\{ 0\}$.

\item If $p\in \Pfree (E)$ is not minimal, $G_p$ is the abelian group generated by elements $x_w^p$, with $w\in E^0$ such that $[w]<p$, with the relations
\begin{equation}\tag{$e_{\rm reg}$}
    x_w^p = \sum_{e\in s_E^{-1}(w)}x^p_{r(e)}
    \text{ if }
    [w]\in  \Preg (E)
\end{equation}
and
\begin{equation}\tag{$e_{\rm free}$}
    0 = \sum_{e\in X\setminus E_{[v]}^1} x^p_{r(e)}
    \text{ if }
    [v]\in  \Pfree (E) ,X\in C_v,\text{ and } [v]\leq p.
\end{equation}

\item If $p\in \Preg (E)$, $G_p$ is the abelian group generated by elements $x_w^p$, with $w\in E^0$ such that $[w]\leq p$, with the relations ($e_{\rm reg}$) and ($e_{\rm free}$).
\end{enumerate}

Finally, for a pair $q<p$, we define the morphism $\varphi_{p,q}\colon M_q\to G_p$ as 
\[
    \varphi_{p,q} (x^q_w) = 
    x^p_w
\]
if $q\in \Preg (E)$ and
\[
    \varphi_{p,q} \left(n,\sum_{[w]<[v^q]} c_w x^q_w\right)
    = nx^p_{v^q} + \sum_{[w]<[v^q]} c_w x^p_w
\]
if $q\in \Pfree (E)$.

One can check that $\mathcal{I}(E,C):= (\mathbb{P}(E),\leq , (G_p)_p, \varphi_{p,q})$ with the partition $\mathbb{P}(E)=\Pfree (E)\sqcup \Preg (E)$ is a $\mathbb{P}(E)$-system.
\end{pgr}




\begin{thm}
\label{thm:necessary-condition}
Let $(E,C)$ be a super adaptable graph. Then, the associated $\mathbb{P}(E)$-system $\mathcal{I}(E,C)$ is super adaptable. 
\end{thm}
\begin{proof}	
Let $p$ be a free prime in $\mathbb{P}(E)$. By the construction in \autoref{pgr:Isystem-of-superadaptable-graph} the group $G_p$ is generated as a group by all the elements $x^p_w$, where $w\in E^0$ and $[w]<p$.

Now let $w\in E^0$ be such that $[w]<p$. If $q:= [w]$ is regular then $x^p_w = \varphi_{p,q} (x^q_w)\in \varphi _{p,q} (G_q)$. If $w=v^q$, where $q\in \Pfree (E)$, then $\varphi_{p,q} (q) = x^p_{w}$, where we have once again used the identification $q=(1,0_{G_q})$. Hence $G_p$ is generated as a group by the family given in  \autoref{dfn:superadaptableIsystem}.

We now show that $G_p$ is also generated as a semigroup by the family $\{x^p_w : [w] < p\}$. For this, it is enough to show that the inverses $-x^p_w$ in $G_p$ of all the generators can be expressed as sums of generators $x^p_{w'}$.  Let $w\in E^0$ such that $[w] <p$. This means that there is a finite directed path $\ell $ such that $s(\ell) = v^p$ and $r(\ell) = w$. We proceed by induction on the length of $\ell $. If the length of $\ell $ is $1$, then $\ell = e\in X\setminus E^1_p$ for some $X\in C_{v^p}$. By ($e_{\rm free}$), with $v=v^p$, we have 
\[
    x^p_{w} + \sum _{f\ne e,f\in X\setminus E_p^1} x^p_{r(f)} = 0,
\]
so the inverse $-x^p_w$ of $x^p_w= x^p_{r(e)}$ is a sum of generators in our family. 

Now fix $k\geq 1$ and suppose that $-x^p_{w'}$ is a sum of generators for all vertices $w'$ such that there exists a path of length at most $k$ from $v^p$ to $w'$, and let $\ell = \ell' e$ be a path of length $k+1$ from $v^p$ to some vertex $w$ such that $[w]<p$. Here $\ell '$ is a path of length $k$ and $e\in E^1$. Set $w':= r(\ell ')$. If $q:= [w']\in \Preg (E) $, then by  ($e_{\rm reg}$), we have
\[
    x^p_{w'} = x^p_w + \sum _{f\ne e,\,  f\in s_E^{-1}(w')} x^p_{r(f)},
\]
hence 
\[
    0= x^p_{w'} + (-x^p_{w'}) =   x^p_w + \sum _{f\ne e, f\in s_E^{-1}(w')} x^p_{r(f)} + (-x^p_{w'}),
\]
and using the induction hypothesis, we see that $-x^p_w$ is a sum of generators. If $q\in \Pfree (E)$, then $w'=v^q$ and we conclude as before from  ($e_{\rm free}$), with $v=v^q$, that $-x^p_w= -x^p_{r(e)}$ is a sum of generators. This concludes the proof.
\end{proof}

\chapter{From separated graphs to monoids}\label{sec:FromSepGraphToMon}

The purpose of this chapter is to prove \autoref{thm:MondToGraph}, which shows that the realization problem for super adaptable monoids can be reduced to realizing the monoids $\mathcal{M}(E,C)$ of super adaptable graphs $(E,C)$. Alluding to the diagram from the outlined strategy in the introduction, in this chapter we prove the following isomorphism:
\[
\xymatrixrowsep{1pc}
\xymatrixcolsep{4pc}
\xymatrix{
    {\begin{array}{c}
         \text{s.a.}  \\
          \text{monoid}\\
          M
    \end{array}}
    \ar[r]^-{\mathcal{I}}
    \ar@/^3.0pc/[rrr]^-{\id}_{\cong}& 
    {\begin{array}{c}
         \text{s.a.}  \\
          I\text{-system}\\
          J
    \end{array}}
    \ar[r]^-{\mathcal{G}}& 
    {\begin{array}{c}
         \text{s.a.}  \\
          \text{graph}\\
          (E,C)
    \end{array}}
    \ar[r]^-{\mathcal{M}}&
    {\begin{array}{c}
         \text{s.a.}  \\
          \text{monoid}\\
          M
    \end{array}}
}
\]

We divide the chapter in two sections: In \autoref{sec:IsoIsystem}, we prove that $\mathcal{I}(E,C)$ is plainly isomorphic to $\mathcal{I}(\mathcal{M}(E,C),\mcan )$; see \autoref{prp:isoIsysMon}. In \autoref{sec:IsoMonod}, we use this result to prove \autoref{thm:MondToGraph}.


\section{\texorpdfstring{Isomorphism of $I$-systems}{Isomorphism of I-systems}}\label{sec:IsoIsystem}

\begin{pgr}[The monoid of a separated graph]\label{dfn:monoidsepgraph}
Let $(E,C)$ be a finitely separated graph. Following \cite[Definition~4.1]{AraGoo12Crelle}, we define its associated monoid $\mathcal{M}(E,C)$\index[terms]{monoid!of separated graph}\index[symbols]{$\mathcal{M}(E,C)\quad$(separated graph monoid)} as
\[
    \mathcal{M}(E,C):=
    \left\langle
    a_v^C\,\, (v\in E^0)
    \mid
    a_v^C = \rr (X)\text{ for each }X\in C_v,\, v\in E^0
    \right\rangle
\]
where $\rr (X):=\sum_{e\in X} a_{r(e)}^C$.

Here, we write the generators as $a_v^C$ to distinguish them from the generators of the monoid $\mathcal{M}(E)$ defined in \autoref{prg:construction}.

As a quick example, the monoid associated to the coloured graph
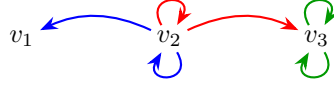
\begin{center}
\begin{tikzpicture}[>=Stealth, scale=1.3]
    \tikzstyle{vertex} = [inner sep=2pt]

    \node[vertex] (v1) at (0,0) {$v_1$};
    \node[vertex] (v2) at (1.5,0) {$v_2$};
    \node[vertex] (v3) at (3,0) {$v_3$};

    \draw[->, red, thick, in=60, out=120, looseness=6] (v2) to (v2);
    
    \draw[->, blue, thick, in=-120, out=-60, looseness=6] (v2) to (v2);
    \draw[->, blue, thick, bend right=25] (v2) to (v1);
    
    \draw[->, red, thick, bend left=25] (v2) to (v3);

    \draw[->, green!60!black, thick, in=60, out=120, looseness=6] (v3) to (v3);
    \draw[->, green!60!black, thick, in=-120, out=-60, looseness=6] (v3) to (v3);
\end{tikzpicture}
\end{center}

is
\[
    \langle\quad 
    a_{v_1}^C,a_{v_2}^C,a_{v_3}^C\quad 
    \mid\quad 
    a_{v_2}^C=a_{v_1}^C+a_{v_2}^C,\quad 
    a_{v_2}^C=a_{v_3}^C+a_{v_2}^C,\quad 
    a_{v_3}^C=2a_{v_3}^C\quad 
    \rangle .
\]
\end{pgr}


As stated, our first goal is to identify two systems, the one associated to $(E,C)$ (given in \autoref{pgr:Isystem-of-superadaptable-graph}) and the one associated to the monoid $\mathcal{M}(E,C)$ (given in \autoref{pgr:IM}). First, in order to get a system for $\mathcal{M}(E,C)$, we need to ensure that $\mathcal{M}(E,C)$ is a primely generated, conical, refinement monoid.

\begin{lma}\label{thm:SuperAdaptRefin}
    Let $(E,C)$ be a super adaptable graph. Then, $\mathcal{M}(E,C)$ is a primely generated, conical, refinement monoid. Moreover, all elements $a^C_v$, for $v\in E^0$, are prime elements of $\mathcal M (E,C)$. 
\end{lma}

\begin{proof} The proof of this lemma is the same as the one in \cite{AraBosPar20SgpFor} for adaptable graphs. 
	The proofs of Lemmas 2.3 and 2.4 in \cite{AraBosPar20SgpFor} hold \emph{verbatim} in our situation, hence by the proof of \cite[Proposition 4.4]{AraMorPar07NonStab}, $\mathcal M (E,C)$ is a refinement monoid. Now the proof of \cite[Proposition 2.6]{AraBosPar20SgpFor} gives that each $a_v^C$, $v\in E^0$, is a prime element of $\mathcal M (E,C)$, and thus this monoid is primely generated. 
	\end{proof}


We now obtain a \emph{plain} isomorphism of systems between the system $\mathcal I (E,C)$ and the system
$\mathcal I (\mathcal M (E,C),\mcan)$, where $\mcan = \mcan (E,C)$ is a certain \emph{canonical mark} on $\mathcal M(E,C)$.
This is completely analogous to \cite[Proposition 2.9]{AraBosPar20SgpFor}, but we warn the reader that the marks
are not explicitly considered in \cite{AraBosPar20SgpFor}, although they are implicit in the arguments.
Here we clarify the situation by explicitly defining the mark $\mcan$.  First, it is convenient to identify 
$\mathbb P (\ol{\mathcal M (E,C)})$ with $\mathbb P (E)$, as follows.

We define an antisymmetric transitive relation $\lhd$ on the poset $\mathbb{P}(E)$ induced by $E$ by setting $p\lhd q$ if $p<  q$ or if $p=q\in \Preg (E)$. Let $M(\mathbb{P}(E),\lhd )$ be the primitive monoid generated by $\mathbb{P}(E)$ with relations $q=q+p$ if $p\lhd q$; recall the discussion in \autoref{pgr:PrimMonod}.

\begin{lma}\label{lma:identifing-the-primes} Let $(E,C)$ be a super adaptable graph. 
The antisymmetrization $\overline{\mathcal{M}(E,C)}$ of $\mathcal{M}(E,C)$ is isomorphic to $M(\mathbb{P}(E),\lhd )$. Hence
\[
    (\mathbb P (\overline{\mathcal{M}(E,C)}),\lhd) \cong (\mathbb P (E),\lhd).
\]
In particular, the free/regular primes of $\overline{\mathcal{M}(E,C)}$ correspond to the free/regular elements of $\mathbb P (E)$.
\end{lma}
\begin{proof}
First, observe that there is a surjective monoid morphism from $\mathcal{M}(E,C)$ to $M(\mathbb{P}(E),\lhd )$ such that $a_v^C\mapsto [v]$. Since $M(\mathbb{P}(E),\lhd )$ is a primitive monoid, we get a well-defined induced map $\overline{\mathcal{M}(E,C)}\to M(\mathbb{P}(E),\lhd )$ such that $\overline{a_v^C}\mapsto [v]$. One can now easily check from the definition of super adaptability that such a map is also injective.

Now by the general theory of primitive monoids \cite{Pierce}, we can recover $(\mathbb{P}(E),\lhd)$ from $M(\mathbb{P}(E),\lhd )$, that is,
\[
    (\mathbb P (\overline{\mathcal{M}(E,C)}),\lhd) \cong (\mathbb P (M(\mathbb P  (E),\lhd)), \lhd) \cong (\mathbb P (E),\lhd). 
\]
This proves the lemma.
\end{proof}

Using \autoref{lma:identifing-the-primes} above, we identify the poset $\mathbb{P}(\ol{\mathcal{M}(E,C)})$ with 
$\mathbb P (E)$ through the isomorphism  $\overline{a_v^C}\mapsto [v]$. 

We can now define the \emph{canonical mark}\index[terms]{marked monoid!mark of $\mathcal{M}(E,C)$}\index[symbols]{$\mcan\quad$ (mark of $\mathcal{M}(E,C)$)} $\mcan = \mcan (E,C)$ on $\mathcal M (E,C)$ by 
\[
    \mcan = \{ e_p \mid p\in \Preg (E)\} \cup \{ a^C_{v^p} \mid p\in \Pfree (E)\},
\]
where, for $p\in \Preg (E)$, $e_p$ is the neutral element of the group $G_p= \pi^{-1} (p)$ for $\pi \colon \mathcal M (E,C)\to \ol{\mathcal M (E,C)}$ the canonical projection.

\begin{prp}\label{prp:isoIsysMon}
Let $(E,C)$ be a super adaptable graph. Then, there is a plain isomorphism of systems 
\[
    \mathcal{I}(E,C)\cong \mathcal{I}(\mathcal{M}(E,C),\mcan ).
\]
\end{prp}
\begin{proof}
We follow the steps from \cite[Proposition~2.9]{AraBosPar20SgpFor}. Using \autoref{lma:identifing-the-primes}, we identify $\mathbb P (\ol{\mathcal M (E,C)})$ with $\mathbb P (E)$. Denote the components of the systems by
\[
    \mathcal{I}(\mathcal{M}(E,C),\mcan) = (\mathbb{P} (E) , \leq , (G_p)_p, \varphi_{p,q} \,(q<p))
    \]
    and \[ \mathcal{I}(E,C)= (\mathbb{P}(E),\leq , (G_p'')_p, \varphi_{p,q}''\, (q<p)).
\]

Now take $p\in \Pfree (E)$ and let us show that the group $G_p''$ of $\mathcal{I}(E,C)$ is isomorphic to $G_p$. First, recall from \cite[Remark~2.5]{AraPar16Isr} that $G_p$ is isomorphic to the group
\[
    G_p'=
    \{
        a_{v^p}^C+\alpha \mid 
        a_{v^p}^C+\alpha\leq a_{v^p}^C
    \}
\]
with product 
\[
    (a_{v^p}^C+\alpha)\circ (a_{v^p}^C+\beta):=a_{v^p}^C+(\alpha +\beta).
\]

Define the map $\lambda^{p}_0 \colon G_p''\to G_p'$ by $x^p_w\mapsto a_{v^p}^C+a_w^C$. Note that this is a well-defined group homomorphism, since
\[
    a_{v^p}^C+a_w^C = a_{v^p}^C+\rr (s_E^{-1}(w))=\bigcirc_{e\in s_E^{-1}(w)} (a_{v^p}^C+a_{r(e)}^C)
\]
if $[w]$ is regular, and
\[
    a_{v^p}^C+a_w^C=a_{v^p}^C+\rr\left( X_i^{[w]}\right) = \bigcirc_{e\in X_i^{[w]}} (a_{v^p}^C+a_{r(e)}^C),
\]
which implies
\[
a_{v^p}^C+0 =  \bigcirc_{e\in X_i^{[w]}\setminus E^1_{[w]}} (a_{v^p}^C+a_{r(e)}^C)
\]
if $[w]$ is free with $[w]\leq p$ and $X_i^{[w]}\in C_w$. Hence relations 
($e_{\rm reg}$) and ($e_{\rm free}$) are preserved by $\lambda^p_0$.

Using that $p$ is free, one can check that all the elements in $G_p'$ can be written (when viewed as elements in $\mathcal{M}(E,C)$) as $a_{v^p}^C+\alpha$ with $\alpha$ a sum of $a_w^C$'s with $w< v^p$. Thus, $\lambda^p$ is surjective.
    
The same proof as in \cite[Proposition~2.9]{AraBosPar20SgpFor} shows that $\lambda^{p}_0$ is injective. Thus, the map $\lambda^{p} \colon G_p '' \to G_p$ given by
\[
    x_{w}^{p}\mapsto (a_{v^p}^C+a_w^C)-a_{v^p}^C
\]
is an isomorphism, where note that $\lambda^p$ is the composition of $\lambda^{p}_0$ with the isomorphism $G_p'\cong G_p$ given by $a_{v^p}^C+\alpha\mapsto (a_{v^p}^C+\alpha)-a_{v^p}^C$.

Now, if $p$ is regular, define $\lambda^p\colon G_p''\to G_p$ by $x^p_w\mapsto e_p+a_w^C$. It is readily checked  that this is an isomorphism.
Moreover, the maps $\lambda^p$ provide a plain isomorphism of systems, as follows:

Suppose that $q<p$. Then we have to show that $\lambda^p\circ \varphi''_{p,q} = \varphi_{p,q}\circ \ol{\lambda ^q}$. We will only check this when $p,q$ are free, and leave the rest of cases to the reader.  Thus, suppose that $p,q\in \Pfree (E)$. Using the formula in \autoref{pgr:Isystem-of-superadaptable-graph}, we get
\begin{align*}
	(\lambda^p \circ \varphi_{p,q}'')\left(n,\sum _{[w]< [v^q]} c_w x^q_w\right) & = \lambda ^p \left(nx^p_{v^q} + \sum _{[w]<[v^q]} c_w x_w^p\right) \\
	& = \left(a^C_{v^p} +na^C_{v^q} + \sum_{[w]<[v^q]} c_w a^C_{w}\right) -a_{v^p}^C.
	\end{align*}
Now we use \eqref{eq:formula-for-varphi-m}, taking into account that $\varphi_{p,q} = \varphi_{p,q}^{\mcan}$, to compute
\begin{align*}
(\varphi_{p,q}\circ \ol{\lambda^q})\left(n,\sum _{[w]< [v^q]} c_w x^q_w\right) & = \varphi_{p,q}^{\mcan} \left(n, \left(a_{v^q}^C +\sum_{[w]<[v^q]} c_w a_w^C\right)-a_{v^q}^C\right) \\
& = \left(a_{v^p}^C +n a_{v^q}^C +  \sum_{[w]<[v^q]} c_w a_w^C \right) -a_{v^p}^C.	
\end{align*}
This shows the formula, and completes the proof.
\end{proof}

\section{Isomorphism of monoids}\label{sec:IsoMonod}

We begin the section with the following lemma, which is fundamental to show our main result. 

\begin{lma}\label{lma:SuperAdpMonoid}
Let $J$ be a countable super adaptable $I$-system. Then there exists a plain isomorphism of systems 
\[
    J \cong \mathcal{I}(E,C),
\]
where $(E,C) = \mathcal G (J)$ is the super adaptable graph associated to $J$. 
\end{lma}

\begin{proof}
Using that $\mathbb{P}(E)$ is order-isomorphic to the poset $I$ associated to $J$ (\autoref{prp:constructed-super adaptable}), we write $J$ and $\mathcal{I}(E,C)$ as follows:
\[
    J= (\mathbb{P}(E),\leq , (G_p)_p, \varphi_{p,q})
    ,\andSep 
    \mathcal{I}(E,C)= (\mathbb{P}(E),\leq , (G_p'')_p, \varphi_{p,q}'').
\]

We first show that some relations hold in the groups $G_p''$, where $p\in \Preg (E)$. Throughout the proof we make use of the notation introduced in both \autoref{prg:construction} and \autoref{pgr:Isystem-of-superadaptable-graph}.

Suppose that $p\in \Preg (E)$. It follows from $(e_{\mathrm{reg}})$ (in \autoref{pgr:Isystem-of-superadaptable-graph}) and (R3) (in \autoref{prg:construction}(I)) that $x^p_{(v_i^p)^{2j-1}} = 0$ for all $j\ge 1$. Using again $(e_{\mathrm{reg}})$, it follows from (R4) that $x^p_{(v_i^p)^{2j}} = 0$ for all $j\ge 1$, and from (R2) that $x^p_{\widehat{v_i^p}} = 0$. Each vertex $v_i^p$ emits the same edges in $E$ as in $S_p$ plus an additional edge to $\widehat{v_i^p}$, and we know that $x^p_{\widehat{v_i^p}} = 0$. Thus, $(e_{\mathrm{reg}})$ implies that
\begin{equation}
\label{eq:xvp-relations}
    x_v^p = \sum_{e\in s_{S_p}^{-1} (v)} x_{r(e)}^p
\end{equation}
for all $v\in S_p^0$.
  
We observe, for later use, that for $p\in \mathbb P (E)$, the group $G_p''$ is generated by the family
\[
    \{x^p_w\mid w\in S_q^0, q\in \Preg (E), q\le p\}\cup \{x^p_{v^q}\mid q\in \Pfree (E), q<p\} .
\]
Indeed, by definition $G_p''$ is generated by
\[
    \{x^p_w\mid w\in E^0, [w]\in \Preg (E), [w] \leq p\}\cup 
    \{x_{v^q}^p\mid q\in \Pfree (E), q<p\}.
\]

However, since we have $x^q_{(v_i^q)^j}= x^q_{\widehat{v_i^q}}=0_q$ for all $q\in \Preg (E)$ and all $i,j\in \NN$, we see that
\[
    x^p_{(v_i^q)^j} = \varphi_{p,q} (x^q_{(v_i^q)^j}) = \varphi (0_q) = 0_p
\]
and
\[
    x^p_{\widehat{v_i^q}} = \varphi_{p,q} (x^q_{\widehat{v_i^q}})= \varphi (0_q) = 0_p, 
\]
for $q\in \Preg (E)$, $q\le p$, $i,j\in \NN$. Hence the mentioned subset of generators is indeed a generating set of $G_p''$. 
  
We now define group homomorphisms $\zeta_p\colon G_p\to G_p''$ for $p\in \Preg (E)$. Recall that $\gamma_p \colon G_p \to G(S_p)$ is a group isomorphism, and that $G(S_p)$ is the group generated by $\{a_v\mid v\in S_p^0\}$, subject to the defining relations
\[
    a_v= \sum_{e\in s_{S_p}^{-1}(v)} a_{r(e)}
\]
for all $v\in S_p^0$. It follows from \eqref{eq:xvp-relations} that we have a well-defined homomorphism $\zeta_p \colon G_p \cong G(S_p)\to G_p''$ given by 
\[
    \zeta_p( \gamma_p^{-1} (a_v)) = x_v^p\in G_p''
\]
for all $v\in S_p ^0$. 

Recall that $v_i^{p-} = \sum _j c_{j,i}v_j^p$, where $c_{j,i}\ge 0$ are such that 
\[
    a_{v_i^p} + \sum_j c_{j,i} a_{v_j^p} = 0
\]
in the group $G(S_p)$. Since $\zeta_p$ is a group homomorphism, we obtain that 
\[
    x^p_{v_i^p} + \sum _j c_{ji} x^p_{v_j^p} = 0
\]
in $G_p''$. 

On the other hand, recall the notation $B_i = \sum _j b_{j,i} v^p_j$ and the definition of the coefficients $b_{j,i}\ge 0$ and of $v(i)$ from  \autoref{prg:construction}(I). The set of edges in $E$ emitted by $\widehat{v_i^p}$ is described by (R1), which in expanded form reads as follows:
\begin{equation}
    \tag{R1}\widehat{v_i^p} = v_i^p + (v_i^p)^1 + \sum_j b_{j,i} v^p_j + v(i) + \sum_j c_{j,i} v_j^p.
\end{equation}

Therefore by  $(e_{\mathrm{reg}})$ we get
\[
    x^p_{\widehat{v_i^p}} = x^p_{v_i^p} + x^p_{(v_i^p)^1}  + \sum_j b_{j,i} x^p_{v^p_j} + x^p_{v(i)} + \sum_j c_{j,i} x^p_{v_j^p}.
\]

Now, since 
\[
    x^p_{\widehat{v_i^p}} = x^p_{(v_i^p)^1} = x^p_{v_i^p} +  \sum_j c_{j,i} x^p_{v^p_j} =0,
\]
it follows that
\begin{equation}
\label{eq:elemenstA_i-and-vi}
    \sum_j b_{j,i} x^p_{v_j^p} + x^p_{v(i)} = 0.
\end{equation}  

Using \eqref{eq:elemenstA_i-and-vi}, we can show commutativity of some diagrams. In the rest of the proof, we will denote the neutral element of $G_p$ by $e_p$, and the neutral element of $G_p''$ by $0_p$. \medskip

\noindent\hypertarget{Claim1.5.5}{\textbf{Claim 1.}} \emph{With the above notation, we have 
\begin{equation}
\label{eq:commuting1}
    \zeta_p \circ \varphi_{p,q} =\varphi''_{p,q} \circ \zeta_q, \quad \text{ for } q<p,\quad  p,q\in \Preg (E).
\end{equation}
Moreover, one also has
\begin{equation}
\label{eq:commuting2}
    \zeta_p (\varphi_{p,q} (n, e_q))  = \varphi''_{p,q} (n, 0_q), \quad \text{ for } q<p,\quad
    p\in \Preg (E), q\in \Pfree (E), n\in \NN.
\end{equation}}
\medskip 

\noindent\emph{Proof of Claim 1.} To prove \hyperlink{Claim1.5.5}{Claim 1}, note that all maps involved in \eqref{eq:commuting1} are group homomorphisms. Thus, it suffices to check the property for the generating family $X_q^{\rm gen}$. If $z\in X_q^{\rm gen}$, there exists a unique $i\in \NN$ such that $z= \gamma_q^{-1}(a_{v(i)})$. Further, since $z=g(i)$, we have $\varphi_{p,q} (z) = -\sum_j b_{j,i}\gamma_p^{-1}(a_{v_j^p})$. Hence we get, using \eqref{eq:elemenstA_i-and-vi} for the third equality,
\begin{align*}
    \zeta_p (\varphi_{p,q} (z)) &= \zeta_p \left(-\sum_j b_{j,i}\gamma_p^{-1}(a_{v_j^p}) \right)\\
    &= - \sum _j b_{j,i} x_{v_j^p}^p \\
    &= x_{v(i)}^p = \varphi''_{p,q} (x_{v(i)}^q) =  \varphi''_{p,q} (\zeta_q (z)).
\end{align*}
To show \eqref{eq:commuting2}, suppose $q<p$,  $p\in \Preg (E)$ and  $q\in \Pfree (E)$. Again it suffices to show the case where $n=1$. Now recall that the element $(1,e_q)$ has been denoted also by $q$ in the setting of \autoref{prg:construction}. There exists a unique $i\in \NN$ such that $g(i)= q$. Moreover in the bijection
$\gamma^p\colon \mathbf{X}_p \to \mathbf{V}_p$, we have that $\gamma^p(q) = v^q$. Thus we obtain
\[
v(i)= \gamma^p (g(i)) = \gamma^p (q) = v^q
\]  
so that $v(i)= v^q$. We obtain, again using \eqref{eq:elemenstA_i-and-vi} for the third equality,
\begin{align*}
    \zeta_p (\varphi_{p,q} (1, e_q)) &=  \zeta_p \left(-\sum_j b_{j,i}\gamma_p^{-1}(a_{v_j^p}) \right)\\
    &= - \sum _j b_{j,i} x_{v_j^p}^p\\
    &= x_{v(i)}^p = x_{v^q}^p =  \varphi''_{p,q} (1,0_q),
\end{align*}
as desired. This proves \hyperlink{Claim1.5.5}{Claim 1}.\medskip

\noindent\hypertarget{Claim2.5.5}{\textbf{Claim 2.}} \emph{
There exists a surjective monoid homomorphism $\eta\colon \mathcal{M}(E,C)\to \mathcal{M}(J)$ such that $\eta (a_v^C) = \chi_p (\gamma_p^{-1} (a_v))$ for all  $p\in \Preg (E)$ and $v\in S_p^0$, and $\eta (a_{v^q}^C) = \chi_q(q)$ for $q\in \Pfree (E)$.}
\medskip

\noindent\emph{Proof of Claim 2.} 
First, we define the assignment
\begin{align*}
    \eta (a_{v^p}^C) &:= \chi _p (p),\quad &&\text{if }p\in \Pfree (E)\\
    \eta (a_v^C) &:= \chi_p (\gamma_p^{-1} (a_v)),\quad &&\text{for all }v\in S_p^0,\, p\in \Preg (E)\\
    \eta (a^C_{\widehat{v_i^p}}) &:= \chi_p(e_p),\quad &&\text{for all }i\ge 1,\, p\in \Preg (E)\\
    \eta (a^C_{(v_i^p)^j}) &:= \chi_p(e_p),\quad &&\text{for all }i,j\ge 1,\, p\in \Preg (E)
\end{align*}
where $e_p$ is the neutral element of $G_p$.

We have to show that the relations in $\mathcal{M}(E,C)$ are preserved by $\eta$. If $p$ is regular and $v= v_i^p\in S_p^0$, we have
\begin{align*}
	     \sum_{e\in s_E^{-1}(v)} \eta (a_{r(e)}^C)  &= \sum_{e\in s_{S_p}^{-1}(v)} \eta (a_{r(e)}^C) + \eta (a_{\widehat{v_i^p}}^C)  \\
   &    = \chi_p\circ\gamma_p^{-1}\left(\sum_{e\in s_{S_p}^{-1}(v)}  a_{r(e)}\right) + \chi_p(e_p) \\
   & =  \chi_p\circ\gamma_p^{-1}(a_v) =  \eta(a_v^C).
\end{align*}
Note that we have a well-defined semigroup homomorphism $j_p\colon G(S_p)\to \mathcal M (E,C)$ such that $a_v\mapsto a^C_v$ for all $v\in S_p^0$, and that the composition $\eta \circ j_p$ is the map
given by 
\[
(\eta \circ j_p) (a_v) = \chi _p (\gamma_p^{-1} (a_v)). 
\]
Hence $\eta \circ j_p = \chi_p \circ \gamma_p^{-1}$, where $\chi_p\colon G_p \to \chi_p(G_p)\subseteq \mathcal M (J)$ is the canonical isomorphism from $G_p$ onto $\chi_p(G_p)$. In particular $\eta\circ j_p$ is a group isomorphism from $G(S_p)$ onto $\chi_p(G_p)$, and so, since $a_{v_i^{p}} + \sum _j c_{j,i} a_{v_j^p} = 0$ in 
$G(S_p)$, we get
\begin{align*}
	\sum _j c_{j,i} \eta (a^C_{v_j^p}) + \eta (a^C_{v_i^p}) &= \sum_j c_{j,i} (\eta\circ j_p)(a_{v_j^p}) + (\eta \circ j_p) (a_{v_i^p})\\
	& =   (\eta\circ j_p) \left(\sum_j c_{j,i} a_{v_j^p} +  a_{v_i^p}\right)\\
	& =  (\eta\circ j_p)(0) = \chi_p (e_p).
\end{align*}
Hence for the preservation of the (extended form) of (R1) we have for $q:= [v(i)] < p$,
\begin{align*}
& \eta(a_{v_i^p}^C) + \eta(a_{(v_i^p)^1}^C) + \sum_j b_{j,i} \eta (a_{v^p_j}^C) + \eta (a_{v(i)}^C) + \sum_j c_{j,i} \eta (a_{v_j^p}^C) \\
& = \chi_p (e_p) + \eta (a_{v(i)}^C)+ \sum_j b_{j,i} \eta (a_{v^p_j}^C) \\
& = \chi_p (e_p) +  \chi_q (g(i)) + \sum_j b_{j,i} \chi_p  (\gamma_p^{-1} (a_{v^p_j})) \\ 
& = \chi _p (\varphi_{p,q} (g(i))) + \sum_j b_{j,i}   \chi_p (\gamma_p^{-1} (a_{v^p_j}))\\
& = \chi_p \left(- \sum _j b_{j,i} \gamma_p^{-1} (a_{v_j^p}) + \sum_j b_{j,i}  \gamma_p^{-1} (a_{v^p_j})\right)\\
&  = \chi_p (e_p) = \eta ( a^C_{\widehat{v_i^p}})
\end{align*}
which shows that indeed relation (R1) is preserved by $\eta$. Also, relations (R2), (R3) and (R4) are trivially preserved by $\eta$.

Let $p$ be a free prime, and recall the notation and caveats introduced in \autoref{prg:construction}(II). 
In particular, recall that $K_p$ denotes the kernel of the map $\varphi_p \colon \mathcal M (J_p)\to G_p$.
For each ${\bf n}_i\in {\bf N}_p$, set
\[
    \widehat{x_i} := \sum_{z\in {\bf X}_p} n_{z,i} a_{\gamma^p (z)}^C,
\]

Then we have to show that the relations 
\[
    a_{v^p}^C = a_{v^p}^C + \widehat{x_i}
\]
are preserved by $\eta$ for all $i$. Observe that 
\[
\sum_{z\in {\bf X}_p} n_{z,i} \eta (a_{\gamma^p (z)}^C) =  \sum_{z\in {\bf X}_p} n_{z,i} \chi (z) \in K_p,
\]
since $\mathbf{n}_i = (n_{z,i})_{z\in \mathbf{X}_p} \in \mathbf{N}_p$; see \autoref{prg:construction}(II). 
Thus, 
\begin{align*}
    \eta (a_{v^p}^C) + \sum_{z\in {\bf X}_p} n_{z,i}  \eta (a_{\gamma^p (z)}^C)   &= \chi_p \left(p + \varphi_p \left(\sum_{z\in {\bf X}_p} n_{z,i} \chi (z)\right)\right)\\
    &= \chi_p (p) = \eta (a_{v^p}^C),
\end{align*}
which proves that $\eta$ preserves the desired relation.
		
It follows that $\eta$ is a well-defined monoid homomorphism. Note that the image of $\eta$ contains $\chi_p(x)$ for every $x\in \Xgen _p$, for all $p\in \Preg (E)$ and $\chi_p(p)$ for all $p\in\Pfree (E)$. Since $J$ is super adaptable, $\mathcal{M}(J)$ is generated as a semigroup by these elements, which shows that $\eta$ is surjective. This finishes the proof of \hyperlink{Claim2.5.5}{Claim 2}.\medskip

We will now define a surjective plain homomorphism of systems, which we will denote by 
$f\colon \mathcal I (E,C) \to J$. First observe that $\eta $ defines indeed a homomorphism of marked monoids
\[
\eta \colon (\mathcal M (E,C),\mcan) \longrightarrow (\mathcal M (J), \mathfrak m_J)
\]
because $\eta (a^C_{v^p})= \chi_p(p)= \chi_p (1,e_p)$ for all $p\in \Pfree (E)$.
Moreover, $\eta$ satisfies all the hypothesis in \autoref{cor:morphisms-inducing-plain-homos}, so it follows that there is an induced plain homomorphism of systems
\[
\mathcal I (\eta) \colon \mathcal I (\mathcal M (E,C),\mcan) \longrightarrow \mathcal I (\mathcal M (J), \mathfrak m_J).
\]
Since $\eta$ is surjective, we see that $\mathcal I (\eta)$ is surjective; i.e., all maps $\mathcal I (\eta)_p$ are surjective.

On the other hand, we have defined in \autoref{prp:isoIsysMon} a plain isomorphism of systems $\lambda \colon \mathcal I (E,C) \to \mathcal I (\mathcal M (E,C),\mcan )$. We therefore consider the composition of the plain homomorphisms
\[
\mathcal I (E,C) \overset{\lambda}{\longrightarrow}  \mathcal I (\mathcal M (E,C),\mcan)
\overset{\mathcal I (\eta)}{\longrightarrow}  \mathcal I (\mathcal M (J), \mathfrak m_J) \overset{\chi^{-1}}{\longrightarrow}  J,
\]
where $\chi^{-1}= \{\chi_p^{-1}\}_{p\in \mathbb P (E)}$ is the natural plain isomorphism from \autoref{thm:APIsr16}~(b).

Now define $f= \chi^{-1}\circ \mathcal I (\eta)\circ \lambda$, which is a surjective plain homomorphism of systems from $\mathcal I (E,C)$ to $J$.
To finish the proof, we need to show that each map $f_p$ is injective.

Let us compute the map $f_p$ for $p\in \mathbb P (E)$. Suppose first that 
$p \in \Preg (E)$ and consider $x_w^p\in G_p''$, with $q:=[w]\leq p$,
$q\in \Preg (E)$ and $w\in S_q^0$. Then we have
\begin{align*}
	f_p (x^p_w) & = \chi_p^{-1}(\mathcal I (\eta)_p (\lambda^p(x^p_w)))\\
	&= \chi_p^{-1}(\mathcal I (\eta)_p (e_p+a^C_w))\\
        &=	\chi_p^{-1}(\chi_p (e_p)+ \chi_q(\gamma_q^{-1}(a_w)))\\
	& = \chi_p^{-1}(\chi_p (e_p+ \varphi_{p,q}(\gamma_q^{-1}(a_w))))\\
    &= \varphi_{p,q}(\gamma_q^{-1}(a_w)).
\end{align*}

Here and below, $\varphi_{p,q}= \text{id}_{G_p}$ if $p=q$.

If $w=v^q$ for $q\in \Pfree (E)$, $q<p$, then
\begin{align*}
	f_p (x^p_{v^q}) & = \chi_p^{-1}(\mathcal I (\eta)_p (\lambda^p(x^p_{v^q})))\\
	& = \chi_p^{-1}(\mathcal I (\eta)_p (e_p+a^C_{v^q}))\\
	& =	\chi_p^{-1}(\chi_p (e_p) + \chi_q(q))\\
        &= \chi_p^{-1}(\chi_p (e_p+\varphi_{p,q}(q)))\\
        &= \varphi_{p,q}(q) .
\end{align*}

Suppose now that $p\in \Pfree (E)$. Then, similar computations lead to the same formulas
\[
f_p (x^p_w)= \varphi_{p,q} (\gamma _q^{-1}(a_w))
\]
when $q:=[w]<p$, $q\in \Preg (E)$, and $w\in S_q^0$, and
\[
f_p (x^p_{v^q})= \varphi_{p,q} (q)
\]
when $q<p$ and $q\in \Pfree (E)$.

We now show that $\zeta_p \circ f_p = \text{id}_{G_p''}$ for $p\in \Preg (E)$: 	

Recall from the beginning of the proof that  $G_p''$ is generated by the family
\[
    \{x^p_w\mid w\in S_q^0, q\in \Preg (E), q\le p\}\cup \{x^p_{v^q}\mid q\in \Pfree (E), q<p\} .
\]
Thus, to show the identity $\zeta_p \circ f_p = \text{id}_{G_p''}$, it suffices to check that $\zeta (f_p (x^p_w))
= x^p_w$ for every element $x^p_w$ in this family. 

	 For a regular $[w]=q\leq p$, $w\in S_q^0$, we can use \hyperlink{Claim1.5.5}{Claim 1} and the above formula for $f_p$ to see that 
\[        \zeta_p(f_p (x^p_w)) =
     \zeta_p(\varphi_{p,q}(\gamma_q^{-1} (a_w))) = \varphi_{p,q}'' (\zeta_q (\gamma_q^{-1} (a_w))) = \varphi_{p,q}''(x^q_w) = x^p_w \]
     and for $q\in \Pfree (E)$, $q<p$,
     \[        \zeta_p(f_p (x^p_{v^q})) = 
     \zeta_p(\varphi_{p,q} (q)) = \zeta_p(\varphi_{p,q} (1,e_q))= \varphi''_{p,q}(1,0_q) = x^p_{v^q} \]
  Hence $\zeta_p \circ f_p= \text{id}_{G_p''}$, which implies that $f_p$ is injective for all $p\in \Preg (E)$. 
  Since we already know that $f_p$ is surjective, we conclude that $f_p$ is an isomorphism, with inverse $\zeta_p$, for all $p\in \Preg (E)$. 
  
  In particular $\zeta_q$ is surjective for all $q\in \Preg (E)$, and thus $G_q''$ is generated as a semigroup by $\{ x^q_w : w\in \Sgen _q\}$. In addition, since $\mathcal I (E,C)$ is super adaptable, we know that, for each $p\in \Pfree (E)$,  $G_p''$ is generated as a semigroup by  
  \[
  (\cup_{q<p\text{ reg.}}\varphi_{p,q}'' (G_q''))\cup \{ \varphi _{p,q}''(q)\}_{q<p\text{ free}}.
  \]
We deduce that $G_p''$ is generated as a semigroup by   
\[
\{ \varphi_{p,q}'' (x_w^q)\}_{q<p\text{ reg.}, w\in \Sgen_q}\cup \{ \varphi _{p,q}''(q)\}_{q<p\text{ free}} = \{ x^p_v : v\in \mathbf{V}_p \},
\]
where we use again the notation introduced in \autoref{prg:construction}.
We are now ready to conclude the proof by showing that $f_p$ is injective whenever $p$ is free. Thus, assume that $p\in \Pfree (E)$.  Take $a\in \ker (f_p)$, and write 
\[
a= \sum_{v\in \mathbf{V}_p} n_v x^p_v,
\]
where $(n_v)_{v\in \mathbf{V}_p}\in \oplus^{\text{o}}_{v\in \mathbf{V}_p} (\NN\cup \{0\})$.
Then we have
\begin{align*}
    e_p &= f_p (a) = f_p \left(\sum _{v\in \mathbf{V}_p} n_v x^p_v \right)\\
    &=\sum _{v\in \Sgen_q, q<p\text{ reg.}} n_v \varphi _{p,q}(\gamma_q^{-1}(a_v)) + \sum_{q<p\text{ free}} n_{v^q} \varphi_{p,q}(q)\\
    &= \varphi_p \left(\sum _{z\in \mathbf{X}_p} n_{\gamma ^p (z)} \chi (z)\right).
\end{align*}
It follows that $\sum _{z\in \mathbf{X}_p} n_{\gamma ^p (z)} \chi (z) \in K_p $, so that $(n_{\gamma^p(z)})_{z\in \mathbf{X}_p} \in \mathbf{N}_p$; see \autoref{prg:construction}(II). Hence there is a unique $i\in \NN$ such that
$\mathbf{n}_i = (n_{\gamma^p(z)})_{z\in \mathbf{X}_p} $. Now by the construction in \autoref{prg:construction}(II), the set $X_i\in C_{v^p}$ emits $n_{\gamma^p(z)}$ edges from $v^p$ to $\gamma^p(z)$ plus a loop at $v^p$. Using ($e_{\rm free}$) we see that $a= \sum_{v\in \mathbf{V}_p} n_v x^p_v = 0_p$, showing that $f_p$ is injective. 

This concludes the proof of the lemma.
\end{proof}


We can now show one of our main theorems.

\begin{thm}\label{thm:MondToGraph}
Let $(M,\mathfrak m)$ be a countable super adaptable marked monoid. Then, with
$(E,C) = \mathcal G (\mathcal I (M,\mathfrak m))$, we have an isomorphism of marked monoids 
\[
     (M, \mathfrak m) \cong (\mathcal{M}(E,C), \mcan).
\]
In particular, for any countable super adaptable monoid $M$ there exists a countable super adaptable graph $(E,C)$ such that $M\cong\mathcal{M}(E,C)$.
\end{thm}
\begin{proof}
Let $(M,\mathfrak m)$ be a countable super adaptable marked monoid. Then 
$\mathcal I (M, \mathfrak m)$ is a super adaptable system. Let 
\[
(E,C) = \mathcal G (\mathcal I (M,\mathfrak m))
\]
be the associated separated graph. 
By \autoref{lma:SuperAdpMonoid} there is a plain isomorphism of systems 
$\mathcal I (M, \mathfrak m) \cong \mathcal I (E,C)$. By \autoref{lma:morphisms-of-systems-induce-homos}, and since the isomorphism is plain, there is an induced isomorphism of marked monoids
\begin{equation}
	\label{eq:iso1}
	(\mathcal M (\mathcal I (M, \mathfrak m)), \mathfrak m_{\mathcal I (M, \mathfrak m)}) \cong (\mathcal M (\mathcal I (E,C)), \mathfrak m_{\mathcal I (E,C)}). 
	\end{equation}  
Now by \autoref{prp:isoIsysMon} there is a plain isomorphism $\mathcal I (E,C)\cong \mathcal I (\mathcal M (E,C), \mcan)$. Just as before we obtain an induced isomorphism of marked monoids 
	\begin{equation}
	\label{eq:iso2}
 (\mathcal M (\mathcal I (E,C)), \mathfrak m_{\mathcal I (E,C)}) \cong (\mathcal M (\mathcal I (\mathcal M (E,C), \mcan)), \mathfrak m_{\mathcal I (\mathcal M (E,C), \mcan)}). 
\end{equation}  
We finally obtain
\begin{align*}
	(M,\mathfrak m) & \overset{2.10(1)}{\cong} (\mathcal M (\mathcal I (M,\mathfrak m)),\mathfrak m_{\mathcal I (M,\mathfrak m)}) \\
	& \overset{\eqref{eq:iso1}}{\cong } ( \mathcal M ( \mathcal I (E,C)), \mathfrak m_{\mathcal I (E,C)})\\
	 & \overset{\eqref{eq:iso2}}{\cong }  (\mathcal M (\mathcal I (\mathcal M (E,C), \mcan)), \mathfrak m_{\mathcal I (\mathcal M (E,C), \mcan)}) \\
	  &  \overset{2.10(1)}{\cong} (\mathcal M (E,C), \mcan).
\end{align*}
Hence $(M,\mathfrak m)\cong (\mathcal M (E,C), \mcan)$, as desired.
\end{proof}

\begin{exa}\label{exa:SupAdapt}
We end the chapter by showcasing an example of a super adaptable graph $(E,C)$ whose associated (primely generated, conical, refinement) monoid $\mathcal{M}(E,C)$ cannot be realized as the monoid of an adaptable separated graph.

Concretely, consider the following separated graph $(E,C)$:
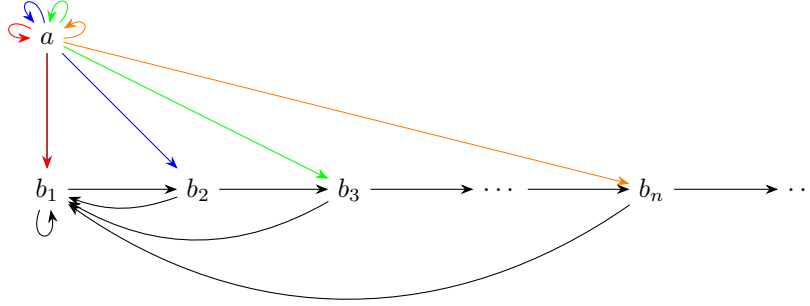
\begin{center}
\begin{tikzpicture}[>=Stealth, rounded corners, scale=0.8]
  \tikzstyle{vertex} = [.]

  \node[vertex] (a) at (0,2.5) {$a$};

  \node[vertex] (b1) at (0,0) {$b_1$};
  \node[vertex] (b2) at ([xshift=2.5cm]b1) {$b_2$};
  \node[vertex] (b3) at ([xshift=2.5cm]b2) {$b_3$};
  \node (dots) at ([xshift=2.5cm]b3) {$\dots$};       
  \node[vertex] (bn) at ([xshift=2.5cm]dots) {$b_n$};
  \node (ldots) at ([xshift=2.5cm]bn) {$\ldots$};     

  \draw[->] (b1) -- (b2);
  \draw[->] (b2) -- (b3);
  \draw[->] (b3) -- (dots);
  \draw[->] (dots) -- (bn);
  \draw[->] (bn) -- (ldots);

  \draw[->] (a) -- (b1);
  \draw[->, red] (a) -- (b1);
  \draw[->, blue] (a) -- (b2);
  \draw[->, green] (a) -- (b3);
  \draw[->, orange] (a) -- (bn);

  \draw[->, blue, in=130, out=100, looseness=8] (a) to (a);
  \draw[->, green, in=80, out=50, looseness=8] (a) to (a);
  \draw[->, red, in=180, out=150, looseness=8] (a) to (a);
  \draw[->, orange, in=30, out=0, looseness=8] (a) to (a);

  \draw[->, in=-80, out=-110, looseness=8] (b1) to (b1);
  \draw[->, bend left=20] (b2) to (b1);
  \draw[->, bend left=30] (b3) to (b1);
  \draw[->, bend left=35] (bn) to (b1);
\end{tikzpicture}
\end{center}
It follows from its construction that $(E,C)$ is super adaptable with associated poset $\mathbb{P}(E)=\{f,r\}$, where $f:=\{a\}$ and $r:=\{b_i \mid i\in\NN\}$. Construct the monoid $M:=\mathcal{M}(E,C)$ as done in \autoref{dfn:monoidsepgraph}. Then, $M$ is generated by $\{a,b_1,b_2,\dots\}$ and is subject to the following relations: for any $i\geq 1$, we have
\[
a=a+b_i \quad \text{ and }\quad
b_i=b_{i+1}+b_1.
\]

It is readily checked that $M$ does not have a finitely generated presentation. Therefore, any separated graph realizing $M$ needs to have a free connected component $f:=\{a\}$ with infinitely many colors. We conclude that no adaptable separated graph can realize $M$ while, by construction, the super adaptable graph $(E,C)$ does.
\end{exa}

\chapter{Realizing super adaptable monoids}\label{sec:RealMon}

In this chapter we finish the proof of \autoref{thmIntro:Main} by showing that, for any countable super adaptable graph and any field $K$, there exists a regular $K$-algebra $Q_K (E,C)$ such that $V(Q_K (E,C))$ is isomorphic to $\mathcal{M}(E,C)$. This corresponds to \hyperlink{Step4}{Step (4)} in the outline.

\section{The construction}
\label{subsect:the-construction}

Extending the work of \cite{AraBosPar20Sel} on adaptable graphs, we aim to construct the \emph{regular algebra of a super adaptable graph}. The original construction heavily relied  on the fact that each vertex on an adaptable graph has a finite number of colors. To address this limitation, we first introduce the notion of \emph{auxiliary data} for a super adaptable graph, as follows:

Let $(E,C)$ be a countable super adaptable graph. For each $p\in \Pfree (E)$, we fix a countable (finite or infinite) set $\Lambda_p$, with the property that the sets $\Lambda _p$, $p\in \Pfree (E)$, along with $\NN$ and $E^1$, are pairwise disjoint. For each non-minimal $p\in \Pfree (E)$ fix an injective map
\[
    \sigma^p = \sigma_E^p \colon E^1_p \sqcup \Lambda _p \sqcup \NN \to \NN.
\]

We write $\NN_p:= \Lambda_p \sqcup \NN$ whenever $p\in \Pfree (E)$, and $\NN_p:=\NN$ whenever $p\in \Preg (E)$.
In general, we will denote by $\alpha$ the elements of $E^1_p$, for $p\in \Pfree (E)$. Connectors, as defined in \autoref{dfn:SupAdSepGra}, will be generically denoted by $\beta$. 
We will also write $\sigma_E$, $\Lambda^E_p$ or $\NN^E_p$ whenever we want to specify the separated graph $(E,C)$. Although technically one should also specify the colouring $C$ and write, for example, $\sigma_{(E,C)}$, it will often be clear what the colouring $C$ is. Thus, we will omit this for ease of notation.

We refer to the pair $(\Lambda = \{\Lambda_p\}_{p\in \Pfree (E)}, \sigma= \{\sigma^p\}_{p\in \Pfree (E), \, p \text{ non-minimal}})$ as a set of \emph{auxiliary data}\index[terms]{super adaptable graph!auxiliary data} for $(E,C)$. Again for ease of notation, we will simply denote a super adaptable graph $(E,C)$ with a prescribed choice of auxiliary data $(\Lambda,\sigma)$ by  $(E,C,\sigma)$. We will denote by $\mathcal T$ the family of all such triples $(E,C,\sigma)$.  

\begin{pgr}[The algebra $\mathcal S_K(E,C,\sigma)$]\label{pgr:AlgebraSKEC}
For $(E,C,\sigma)\in \mathcal T$ and a field $K$, the algebra $\mathcal S_K(E,C,\sigma)$\index[symbols]{$\mathcal S_K(E,C,\sigma)$} is the $*$-algebra with generators $E^0\cup E^1$ and $t_i^v$, $i\in \NN_{[v]}$, $v\in E^0$, and relations
given by the separated graph relations
\begin{enumerate}
\item[(V)] $vv'=\delta_{v,v'}v$ for all $v,v'\in E^0$,
\item[(E1)] $s(e)e=e=er(e)$ for all $e\in E^1$,
\item[(E2)] $r(e)e^*=e^*=e^*s(e)$ for all $e\in E^1$,
\item[(SCK1)] $e^*e'=\delta_{e,e'}r(e)$ for all $e,e'\in X$, $X\in C$, and
\item[(SCK2)] $v=\sum_{e\in X}ee^*$ for every finite set $X\in C_v$, $v\in E^0$;
\end{enumerate}
together with the following additional relations: Let $v\in E^0$ and write $p:=[v]$ for the corresponding prime in $I: =\mathbb P (E)$. We have
        \begin{itemize}
            \item[(i)] $v\cdot t_i^v=t_i^v=t_i^v \cdot v$, $t_i^v \cdot (t_i^v)^{*}=(t_i^v)^{*}\cdot  t_i^v=v$, and $t_i^v \cdot t_j^v = t_j^v\cdot t_i^v$ for all $i,j\in\NN_p$.
            \item[(ii)] Whenever $p$ is regular, we have
            $t_i^{s(e)} \cdot e = e\cdot t_i^{r(e)}$ for any $e\in E^1$ with $s(e)\in E_p^0$ and any $i\in \NN_p = \NN$.
            \item[(iii)] Whenever $p$ is free, we have 
            \begin{itemize}
            \item[(a)] $t_i^{v^p} \cdot \alpha =\alpha \cdot  t_i^{v^p}$ and $t_i^{v^p} \cdot \alpha^*=\alpha^* \cdot t_i^{v^p}$ for any $i\in\mathbb{N}_p$ and any $\alpha \in E_p^1$.
            \item[(b)] $\alpha \alpha'=\alpha'\alpha$ and $\alpha^*\alpha'=\alpha'\alpha^*$ for every $\alpha,\alpha'\in E_p^1$ such that $\alpha \ne \alpha'$.
            \item[(c)] $\beta^* \beta'=0$ for all distinct connectors $\beta,\beta'$ such that $s(\beta) = s(\beta')= v^p$.
            \end{itemize}
            \item[(iv)]  Whenever $p$ is free and non-minimal, we have
            \begin{itemize}
            \item[(a)] $(t_i^{v^p})^{\pm}\cdot \beta =\beta  \cdot (t_{\sigma^p(i)}^{r(\beta)})^{\pm}$ for any $i\in \NN_p$ and any connector $\beta$ such that $s(\beta) = v^p$. 
            \item[(b)] $\alpha \cdot \beta =\beta \cdot t_{\sigma^p(\alpha)}^{r(\beta)}$ for any $\alpha\in E^1_p$ and any connector $\beta\in X_{\alpha'}$ with $\alpha'\in E^1_p$ and $\alpha'\ne \alpha$. 
            \item[(c)] $\alpha^* \cdot \beta =\beta \cdot (t_{\sigma^p(\alpha)}^{r(\beta)})^{-1}$ for any $\alpha\in E^1_p$ and any connector $\beta\in X_{\alpha'}$ with $\alpha'\in E_p^1$ and $\alpha'\ne \alpha$. 
            \end{itemize}
            \end{itemize}

Let us remark that, since this is a presentation of $\mathcal S_K(E,C,\sigma)$ within the category of $*$-algebras, the $*$-relations of the above relations are enforced in the $*$-algebra $\mathcal S_K(E,C,\sigma)$.
\end{pgr}


In what follows, we construct the \emph{regular algebra}\index[terms]{regular algebra!of a super adaptable triple}\index[symbols]{$Q_K(E,C,\sigma)\quad$(regular algebra of $(E,C,\sigma)$} $Q_K(E,C,\sigma)$ of a super adaptable triple $(E,C,\sigma)\in \mathcal T$  following the blueprint from \cite{AraBosPar20Sel} for adaptable graphs.
 Notice that the particular type of universal localisation we are using here is the one described at the beginning of Subsection 2.1 of \cite{AraBosPar20Sel}. We refer the reader to that paper and to
    the books \cite{Cohn,Schofield} for background material on universal localisation and related matters.
    
 \begin{dfn}
     \label{dfn:regular-algebra}
Let $(E,C,\sigma)$ be a triple in $\mathcal T$, with auxiliary data $(\Lambda,\sigma)$. We first define an auxiliary algebra $\mathcal S^1_K(E,C, \sigma)$, as follows:

    Let $\mathcal S_K(E,C,\sigma)$ be the $*$-algebra defined above. 
    For each vertex $v\in E^0$, we invert all nonzero polynomials in the variables $t_i^v$, for $i\in \NN_{[v]}$, obtaining a new algebra $\mathcal{S}^1_K(E,C,\sigma)$. In this way, we obtain a copy of the 
purely transcendental extension field $L= K(t_i)$ in countably many variables inside each corner $v\mathcal{S}^1_K(E,C,\sigma) v$, for $v\in E^0$.
    
    Then, in the algebra $\mathcal{S}^1_K(E,C,\sigma)$, we consider, for each free prime $p\in \Pfree (E)$, the set $\Sigma(p)$ of all the polynomials of the form $f(\alpha) \in L[\alpha : \alpha\in E_p^1 ]$ such that $v(f)=0$; cf. \cite[Definition 2.7]{AraBosPar20Sel}. 
    In short, these are the polynomials $f(\alpha)$ in the variables $\alpha\in E_p^1$ which cannot be written in the form $(\alpha')^i g(\alpha)$ for some $\alpha'\in E_p^1$ and some $i\in \NN$. 

    Finally, for each regular prime $p \in \Preg (E)$, we consider the set $\Sigma (p)$ of all square matrices, with coefficients in the path $L$-algebra over the component $E_p$, which become invertible under the augmentation map $\varepsilon$; see \cite[Definition 2.7]{AraBosPar20Sel} for details.  

Now set $\Sigma = \bigcup_{p\in \mathbb P (E)} \Sigma (p)$, and define the \emph{regular algebra} $Q_K(E,C,\sigma)$ as
\[
Q_K(E,C,\sigma) := \mathcal S^1_K (E,C,\sigma) \Sigma^{-1},
\]
which is the universal localisation of $\mathcal S^1_K (E,C,\sigma)$ with respect to the set $\Sigma$. 
\end{dfn}
    
 We remark that $Q_K(E,C,\sigma)$ is not in general a $*$-algebra, since the involution $*$ on $\mathcal S_K(E,C,\sigma)$ cannot be extended to $Q_K(E,C,\sigma)$.  


\section{Restriction and quotient}
\label{subsect:restrict-and-quotient}

Suppose that $(E,C)$ and $(F,D)$ are super adaptable graphs such that $(E,C)$ is a complete subobject of $(F,D)$ in the category ${\rm SGr}$ of separated graphs; see \autoref{pgr:morphisms-in-SGr}.
We assume that all non-minimal free primes $p\in \Pfree (E)$ are also free primes in $F$, and recall that all regular primes $p\in \Preg (E)$ are automatically regular in $F$ by \autoref{prp:SuperGr}.
Let $(\Lambda ^F, \sigma_F)$ be auxiliary data for $(F,D)$, as defined at the beginning of \autoref{subsect:the-construction}.

Then, we can \emph{restrict} $\sigma_F$ to $(E,C)$ by taking subsets $\Lambda^E_p\subseteq \Lambda ^F_p$ for each $p\in \Pfree (E)$ and taking, for each non-minimal free element $p\in \Pfree (E)$,
\[
\sigma_E^p\colon E^1_p \sqcup \Lambda_p^E \sqcup \NN \to \NN,
\]
to be the restriction of $\sigma_F^p$ to $E^1_p \sqcup \Lambda_p^E \sqcup \NN$. This can be done because $p$ is a non-minimal free prime of $F$, 
$E^1_p\subseteq F^1_p$, and $\Lambda_p^E\subseteq \Lambda_p^F$. In this situation we will write $(E,C,\sigma_E) \le (F,D,\sigma_F)$, and we will denote by  $Q_K(E,C,\sigma_E)$ and $Q_K(F,D,\sigma_F)$ the corresponding algebras. 

\begin{prp}
\label{prp:restriction}
Let $(E,C,\sigma_E) \le  (F,D,\sigma_F)$ be as above. Then, there exists an induced $K$-algebra homomorphism $\iota \colon Q_K(E,C,\sigma_E)\rightarrow Q_K(F,D,\sigma_F)$ sending 
the canonical generators of $Q_K(E,C,\sigma_E)$ to the corresponding elements of $Q_K(F,D,\sigma_F)$. 
\end{prp}

\begin{proof}
Since $(E,C)$ is a complete subobject of $(F,D)$, the relations involving edges and vertices of $E$ are preserved by $\iota$. Since $\sigma_E^p$ is the restriction of $\sigma_F^p$ for each non-minimal free prime $p$ of $E$, the relations (i)-(iv) are also preserved by $\iota$. We thus obtain a well-defined $K$-algebra homomorphism $\iota \colon \mathcal S_K (E,C,\sigma_E) \to \mathcal S_K(F,D,\sigma_F)$.
Since all nonzero polynomials in $K[t_i^v:i\in \NN_{[v]}]$ become invertible in the corner $v\mathcal S_K^1(F,D,\sigma_F) v$, for $v\in E^0$, we obtain that $\iota$ extends uniquely to a $K$-algebra homomorphism, also denoted by $\iota$, from $\mathcal S_K^1(E,C,\sigma_E)$ to $\mathcal S_K^1(F,D,\sigma_F)$. Since $\Sigma_E \subseteq \Sigma_F$, it is clear that this map can be further uniquely extended to a $K$-algebra homomorphism $\iota\colon  Q_K(E,C,\sigma_E)\rightarrow Q_K(F,D,\sigma_F)$. 
\end{proof}

\begin{rmk}
    Observe that functoriality at this level is clear, that is, if we have $(E_1,C_1,\sigma_{E_1}) \le (E_2,C_2,\sigma_{E_2}) \le (E_3,C_3,\sigma_{E_3})$, then $\iota_{E_3,E_1}=\iota_{E_3,E_2}\circ \iota_{E_2,E_1}$. 
\end{rmk}

We now study the behaviour of our construction with respect to quotients of separated graphs. We first recall the following definitions from \cite{AraGoo12Crelle}.

\begin{dfn}[{\cite[Definition 6.3]{AraGoo12Crelle}}]
Let $(E, C)$ be a finitely separated graph. A subset $H$ of $E^0$ is said to be
\begin{itemize}
    \item[(a)] \emph{hereditary}\index[terms]{graph!hereditary set} if for any $e \in E^1$, we have $r(e) \in H$ whenever $s(e)\in H$.
    \item[(b)] \emph{$C$-saturated} if for any $v\in E^0$ and $X\in C_v$, $v\in H$ whenever $r(X) \subseteq H$.\index[terms]{separated graph!saturated set}
\end{itemize}
We denote by $\mathcal H (E, C)$ the lattice of hereditary $C$-saturated subsets of $E^0$.
\end{dfn}

Note that if $(E,C)$ is a super adaptable separated graph, then every hereditary subset $H$ of $E^0$ is automatically $C$-saturated. 

Given $H\in \mathcal H (E,C)$, we can form a \emph{quotient separated graph}\index[terms]{separated graph!quotient}\index[symbols]{$(E/H,C/H)\quad$(separated graph quotient)} $(E/H,C/H)$ \cite{AraGoo12Crelle} by setting
\begin{itemize}
    \item[(a)] $(E/H)^0= E^0\setminus H$,
    \item[(b)] $(E/H)^1= \{ e\in E^1 \mid r(e)\notin H\}$, and
    \item[(c)] for $X\in C$, set $X/H= \{e\in X\mid r(e)\notin H\}$. For $v\in (E/H)^0$, set $(C/H)_v = \{ X/H \mid X\in C_v\}$, which is a partition of 
$s^{-1}_{(E/H)^0} (v)$, and $C/H= \sqcup_{v\in (E/H)^0} (C/H)_v$. 
\end{itemize}

Take now $(E,C,\sigma_E)\in  \mathcal T$, with auxiliary data $(\Lambda^E,\sigma_E)$, and $H\in \mathcal H (E,C)$.
The separated graph $(E/H,C/H)$ is not in general super adaptable. In order to get a super adaptable graph, we need to define a new separated graph $((E/H)^{\ad} , (C/H)^{\ad})$, as follows:

Set $((E/H)^{\ad})^0 =(E/H)^0 = E^0\setminus H$, and
\[
((E/H)^{\ad})^1 = (E/H)^1 \setminus \left\{\alpha \in E_p^1\; \middle| \;
\begin{aligned}
    &p\in \Pfree (E) \text{ and }\\
    &r(\beta )\in H \text{ for all connectors } \beta \in X_\alpha
\end{aligned}
 \right\}.
\]
Take $(C/H)^{\mathrm{ad}}_v= (C/H)_v$ for $[v]\in \Preg (E)$ such that $v\notin H$, and 
\[
(C/H)^{\mathrm{ad}}_{v^p} = \{ (X_\alpha)/H \mid \text{ there exists a connector } \beta \in X_\alpha \text{ such that } r(\beta)\notin H \}
\]
for $p\in \Pfree (E)$ such that $v^p\notin H$. 

Note that $((E/H)^{\ad}, (C/H)^{\ad})$ is a super adaptable graph, with
\[
\Preg ((E/H)^{\ad}) = \Preg (E)\setminus \{[v]\in \Preg (E) \mid v\in H \}
\]
and 
\[
\Pfree ((E/H)^{\ad}) = 
\Pfree (E)\setminus \{p\in \Pfree (E) \mid v_p\in H \}.
\]

Further, we define 
\[
\Lambda _p^{(E/H)^{\ad}} =  \{\alpha \in E^1_p\mid r(\beta)\in H \text{ for all connectors } \beta\in  X_\alpha  \}  \sqcup \Lambda_p^E 
\]
for each $p \in \Pfree ((E/H)^{\ad})$, and  
\[
\sigma^p_{(E/H)^\ad} \colon ((E/H)^\ad)^1 \sqcup \Lambda_p^{(E/H)^\ad} \sqcup \NN = E_p^1 \sqcup \Lambda_p^E \sqcup \NN \to \NN
\]
by $\sigma^p_{(E/H)^\ad} = \sigma_E^p$ for each non-minimal $p \in \Pfree ((E/H)^{\ad})$. 

For $H\in \mathcal H(E,C)$, denote by $I(H)$ the ideal of $Q_K(E,C,\sigma_E)$ generated by $H$. 

\begin{prp}
\label{prp:quotients}
Suppose that $(E,C,\sigma_E)\in \mathcal T$, and that $H\in \mathcal H (E,C)$. Then there is a canonical $K$-algebra isomorphism  
\[
Q_K(E,C,\sigma_E) /I(H) \longrightarrow  Q_K ((E/H)^\ad ,(C/H)^\ad ,\sigma_{(E/H)^\ad})
\]
sending $v+I(H)$ to $v$ for $v\in E^0\setminus H$, $e+I(H)$ to $e$ whenever $r(e)\notin H$ and $[s(e)]\in \Preg (E)$, 
$\alpha + I(H)$ to $\alpha$ whenever $\alpha \in ((E/H)^\ad)^1_p$ for $p\in \Pfree ((E/H)^\ad)$, and $\alpha + I(H)$ to $t_{\alpha}$ whenever $\alpha \in E^1_p$,
$s(\alpha)= r(\alpha)\notin H$, and $r(\beta )\in H$ for all connectors $\beta \in X_\alpha$.  
    \end{prp}

\begin{proof}
The proof uses the universal properties of the objects involved. First, one defines a homomorphism 
\[
\mathcal S _K(E,C,\sigma_E)  \longrightarrow  Q_K ((E/H)^\ad ,(C/H)^\ad ,\sigma_{(E/H)^\ad})
\]
by using the rules in the statement, and sending $H$ to $0$ and $r^{-1}(H)$ to $0$. Further, all the variables $(t_i^v)^\pm$ must be sent to the corresponding variables $(t_i^v)^\pm$ when $v\notin H$. One can check that all the defining relations of $\mathcal S _K(E,C,\sigma_E)$ are preserved by this homomorphism. By the universal property of universal localization, this homomorphism induces a $K$-algebra homomorphism 
\[
 Q_K(E,C,\sigma_E)  \longrightarrow  Q_K ((E/H)^\ad ,(C/H)^\ad ,\sigma_{(E/H)^\ad})
\]
which factors through the quotient by the ideal $I(H)$. Using the same method, one can construct an inverse map. 
\end{proof}


\section{The proof for finite super adaptable graphs}
\label{subsect:proof-finite-case}

The aim of this section is to prove \autoref{thmIntro:Main} for \emph{finite} super adaptable graphs. Concretely, we will show the following:

\begin{thm}
    \label{thm:realization-finite-case}
Let $(E,C)$ be a finite super adaptable graph, and let $(\Lambda,\sigma)$ be auxiliary data for $(E,C)$.
Then $Q_K(E,C,\sigma)$ is a von Neumann regular ring, and the natural map 
\[
    \mathcal M (E,C)\to V (Q_K(E,C,\sigma))
\]
is a monoid isomorphism. 
\end{thm}

For adaptable graphs, the proof of \autoref{thm:realization-finite-case} is \cite[Theorem~B]{AraBosPar20Sel}. Recall from \autoref{pgr:AdapGraph} that all finite adaptable graphs are finite super adaptable graphs, but that the converse is not true. However, the proof of \cite[Theorem~B]{AraBosPar20Sel} can be adapted to finite super adaptable graphs by doing a few modifications. Namely, these modifications involve keeping track of the functions $\sigma^p$ that we have used to define our algebras $Q_K(E,C,\sigma)$ through all the processes performed to the separated graphs appearing in \cite{AraBosPar20Sel}. We heavily encourage the reader to
follow the proof of \cite[Theorem~B]{AraBosPar20Sel} whilst, in parallel, apply the modifications listed below.

Just like \cite[Theorem~B]{AraBosPar20Sel}, the proof of \autoref{thm:realization-finite-case} proceeds by a sequence of reduction steps to certain building blocks, and then a reconstruction of the separated graph (with auxiliary data) $(E,C,\sigma)$, the monoid
$\mathcal M (E,C)$, and the algebra $Q_K(E,C,\sigma)$ from said building blocks, in such a way that we can infer \autoref{thm:realization-finite-case} from the corresponding result for the blocks.

The fact that the result holds for the building blocks will follow from the results in \cite{Ara19Poset}, which we explain in detail below. Note that, at the level of the constructions with separated graphs and associated monoids, we follow exactly the same steps as in \cite{AraBosPar20Sel}. The only variation concerns the consideration here of auxiliary data, and the changes that this introduces in the corresponding algebras. 

Let us have a look at the different steps, keeping track of the auxiliary data we are presently using:

\begin{ntn}
Let $(E,C)$ be a super adaptable graph, with corresponding poset $I=\mathbb P (E)=E^0/{\sim}$. If $J$ is a lower subset of $I$, the restriction graph $E_J$ has a natural structure of separated graph $(E_J,C_J)$, under which it is clearly super adaptable. When $J= I\downarrow [v]$ for a vertex $v\in E^0$, we will denote by $T(v)$ the separated graph $(E_J,C_J)$, and we will say that $T(v)$ is the \emph{tree} of $v$. Of course, the vertex set of $T(v)$, denoted by $T^0(v)$, is the set
of all $w\in E^0$ such that $v\ge w$. We will also need to consider the \emph{strict tree} of $v$, which is the separated graph $\tilde T (v)$ obtained by restricting $(E,C)$ to the lower subset $J' = \{ [w] \in E^0 \mid [w] < [v] \}$ of $I$. Note that lower subsets of $I$ correspond bijectively to hereditary subsets of $E^0$.

We will use the following notation: Let $(E,C)$ and $I$ be as above.
For $[v]\in \Ireg$, set
\[
    X_{[v]} = \{ e\in E^1 \mid s(e)\in [v] \}= s^{-1}([v]),
\]
that is, $X_{[v]}$ is the set of arrows departing
from the strongly connected component $[v]$. 
Now write
\[
    \ol{C} =  \Big( \bigsqcup_{v\in I_{{\rm free}}} C_v\Big) \sqcup \{ X_{[v]} \mid [v] \in \Ireg \},
\]
and set $\ol{C}_v = C_v$ if $v\in \Ifree$ and $\ol{C}_v= \{ X_{[v]}  \}$
if $v\in \Ireg$. 
\end{ntn}

We can now introduce condition (F) for super adaptable graphs; cf. \cite[Section 3]{AraBosPar20Sel}.

Recall the definitions of a tree and a forest for a poset, given in \autoref{dfn:poset-be-a-tree}.

\begin{dfn}
\label{dfn:consition-F}
Let $(E,C)$ be a super adaptable separated graph. We say that $(E,C)$ satisfies \emph{condition (F)}\index[terms]{super adaptable graph!condition (F)} if  
for each $v\in E^0$ there is at most one $X\in \ol{C}\setminus \ol{C}_v$ such that $r(X)\cap [v] \ne \emptyset$.\end{dfn}

By \cite[Lemma 3.1]{AraBosPar20Sel}, if $(E,C)$ is a finite super adaptable graph satisfying condition (F), then $I$ is a forest.

We now introduce a class of morphisms that will be considered in this chapter. Note that they do not belong in general to the category SGr.

\begin{dfn}
 \label{dfn:covermorphism}
 Let $(E,C)$ and $(F,D)$ be two super adaptable separated graphs. A \emph{cover}\index[terms]{super adaptable graph!cover} $\phi \colon (F,D) \to (E,C)$ is a graph homomorphism $\phi = (\phi^0,\phi^1)\colon F\to E$ such that the following conditions hold:
 \begin{itemize}
  \item[(a)] $\phi ^0$ and $\phi^1 $ are surjective.
  \item[(b)] For each $v\in F^0$, the map $\phi^1$ induces a bijection $\phi^1_v\colon s_F^{-1} (v) \to s^{-1}_E (\phi^0(v))$ such that $\phi ^1_v(X)\in C_{\phi^0(v)}$ for each $X\in D_v$. In particular, $\phi^1$ induces a bijection
  $X\mapsto \phi ^1(X)$ from $D_v$ onto $C_{\phi^0 (v)}$.
 \end{itemize}
\end{dfn}

The notion of crowned pushout from \cite[Definition 5.1]{AraBosPar20Sel} is crucial for our approach. 
We will use this notion here, with the extra information of the auxiliary data.

\begin{dfn}
 \label{dfn:crowned-pushout}
Let $((E_1,C^1,\sigma_1), (E_2,C^2,\sigma_2))$ be a pair of finite super adaptable graphs with auxiliary data $(\Lambda^i,\sigma_i)$, $i=1,2$.
We say that the pair $((E_1,C^1,\sigma_1), (E_2,C^2,\sigma_2))$ is a \emph{crowned pushout}\index[terms]{super adaptable graph!crowned pushout} if there is a cover map $\phi \colon (E_1,C^1)\to (E_2,C^2)$ satisfying conditions (i)-(v)
in \cite[Definition 5.1]{AraBosPar20Sel} and such that for each $p\in \Pfree (E_1)$ we have $\Lambda_p^1 = \Lambda_{\phi(p)}^2$, $\sigma_2^{\phi(p)} (i) =\sigma_1^p(i)$ for each 
$i\in \NN_p$, and $\sigma_2^{\phi(p)} (\phi^1(\alpha)) = \sigma_1^p (\alpha)$ for all $\alpha \in (E_1)_p^1$. 
\end{dfn}

In \autoref{dfn:crowned-pushout} above, we have implicitly used that $\phi$ sends free primes to free primes. This property and other important properties of crowned pushouts are shown in \cite[Lemma 5.2]{AraBosPar20Sel}.

\begin{thm}
\label{thm:reduction-theorem-for-conditionF}
Let $(E,C,\sigma)$ be a finite super adaptable graph with auxiliary data $(\Lambda, \sigma)$. Then there is a finite sequence of finite super adaptable graphs $(E_i,C^i,\sigma_i)$ with auxiliary data $(\Lambda_i, \sigma_i)$, $i=0,1,\dots ,n$, such that
\begin{itemize}
    \item[(a)] $(E_n,C^n,\sigma_n)=(E,C,\sigma)$;
    \item[(b)] $(E_0,C^0)$ satisfies condition (F);
    \item[(c)] each pair $((E_{i},C^{i},\sigma_{i}),(E_{i+1},C^{i+1},\sigma_{i+1}))$ is a crowned pair.
\end{itemize}
\end{thm}
\begin{proof}
    The proof follows the steps in \cite[Theorem 3.3]{AraBosPar20Sel}. Note that, in the inductive step of that proof, we can define unique auxiliary data $(\Lambda_{(F',D')},\sigma_{(F',D')})$ for $(F',D')$
   satisfying the required properties with respect to the assumed auxiliary data $(\Lambda_{(F,D)},\sigma_{(F,D)})$ for $(F,D)$. As observed in \cite[Section 5]{AraBosPar20Sel}, after all the steps involved in the proof of \cite[Theorem 3.3]{AraBosPar20Sel} are completed, we can define the above chain, starting with the separated graph $(E_0,C^0,\sigma_0)$, and ending with the original separated graph $(E,C,\sigma)= (E_n,C^n,\sigma_n)$.  
\end{proof}

The next step consists in establishing \autoref{thm:realization-finite-case} in the particular case where $(E,C)$ satisfies condition (F) (this is done in \cite[Section 4]{AraBosPar20Sel} for adaptable graphs). For this purpose, we further decompose a finite super adaptable separated graph $(E,C)$ satisfying condition (F) by using some subgraphs, called the \emph{building blocks} of $(E,C)$. 

\begin{pgr}[Building blocks]
     Say that $\varphi \colon E^0\setminus \text{Sink}(E) \to C$ is a \emph{choice function}\index[terms]{super adaptable graph!choice function}  if $\varphi (v)\in C_v$ for each $v\in E^0\setminus \text{Sink}(E)$. Given such a choice function $\varphi$, define a graph $E_{\varphi}$ by  $E_{\varphi}^1 = \bigsqcup _{v\in E^0\setminus \text{Sink}(E)} \varphi (v)$ and $E_{\varphi}^0 = s_E(E_{\varphi}^1) \cup r_E(E_{\varphi}^1)$. The source and range maps  in $E_{\varphi}$ are defined in  such a way that the inclusion map $E_{\varphi} \to E$ becomes a graph homomorphism.  A \emph{building block}\index[terms]{super adaptable graph!building blocks} of $(E,C)$ is a connected component of a graph of the form $E_{\varphi}$. We will denote by $\mathcal F$ the collection  of all the building blocks of $(E,C)$. Observe that, since $(E,C)$ satisfies condition {\rm (F)}, the associated poset of each building block is a tree.
\end{pgr}

\begin{rmk}\label{rmk:BldBlocks}
    Let $G$ be a building block. This is a non-separated graph, but we may also think of it as a special instance of a super adaptable separated graph, in which 
the non-minimal free primes emit just a single loop (and at least one connector). 
Here, for auxiliary data $(\Lambda^G, \sigma_G)$, if $p$ is a non-minimal free prime, then the value of $\sigma_G^p$ on the unique $\alpha \in G_p^1$ is irrelevant for the definition of the algebra $Q_K(G,\sigma_G)$, because relations (iv)(b) and (iv)(c) only apply when $|G_p^1|>1$. Hence all the structure is determined by the auxiliary sets $\Lambda _p^G$ and the injective maps $\sigma^p_G\colon \Lambda_p^G\sqcup \NN\to \NN$.
\end{rmk}

By \autoref{rmk:BldBlocks}, we can fix mutually disjoint countable sets $\Lambda_p^G$, for $p\in \Pfree (G)$, and injective maps $\sigma^p \colon \Lambda^G_p\sqcup \NN \to \NN$ for each non-minimal free prime, and consider the corresponding $K$-algebra $Q_K(G,\sigma_G)$. This algebra is isomorphic to a certain regular algebra $Q_{\mathbf K}(G)$ of $G$ over a poset of fields $\mathbf K$. We recall what these are below:

\begin{pgr}[The algebra $Q_{\mathbf K}(G)$]
    The theory presented in this paragraph was first developed in \cite{Ara19Poset}. Let $(I,\le )$ be a finite poset. 
Following \cite{Ara19Poset}, we define a \emph{poset of fields}\index[terms]{poset!of fields} as a family $\mathbf K = \{ K_i : i\in I \}$ of fields $K_i$ with the property that $K_i\subseteq K_j$ if $j\le i$.
Let $G$ be a finite directed graph. We assume that there is a pre-order $\le $ on $G^0$ such that $v\le w$ whenever there is a directed path $\gamma $ such that $s_G(\gamma) = w$ and $r_G(\gamma) = v$, and 
we further assume that the partially ordered set $I:= G^0/{\sim}$, associated to the pre-ordered set $(G^0,\le)$, is a tree with greatest element $i_0:=[v_0]$. Denote by $[v]$ the class of $v\in G^0$ in $I$.

Given a poset of fields $\mathbf K$ over $I$, we define the algebra $P_{\mathbf{K}}((G))$ as the algebra of formal power series of the form $a=
\sum_{\gamma \in \text{Path} (G)} a_{\gamma} \gamma $, where each
$a_{\gamma}\in K_{[r(\gamma )]}$. The usual multiplication of formal
power series gives a structure of algebra over $K_0:=K_{i_0}$ on
$P_{\mathbf{K}}((G))$. Indeed if $(a\gamma)(b\mu)$ is nonzero, where $a\in K_{[r(\gamma)]}$ and $b\in K_{[r(\mu )]}$, then $s(\mu) = r(\gamma )$
and it follows from the property of $\le$ that $[r(\gamma)] \ge [r(\mu )]$ in $I$. Then we have $K_{[r(\gamma)]}\subseteq K_{[r(\mu)]}$ 
and so $ab\in K_{[r(\mu)]} = K_{[r(\gamma \mu)]}$, which shows that the product in $P_{\mathbf{K}}((G))$ is well-defined. 

The path algebra $P_{\mathbf{K}}(G)$\index[terms]{path algebra}\index[symbols]{$P_{\mathbf{K}}(G)\quad$(path algebra)} is defined as the subalgebra of
$P_{\mathbf{K}}((G))$ consisting of all the series in
$P_{\mathbf{K}}((G))$ having finite support. We have a natural augmentation homomorphism 
\[
    \epsilon \colon P_{\mathbf K} ((G)) \longrightarrow \bigoplus _{v\in G^0} K_{[v]}v .
\]
We denote by $\Sigma $ the set of all square matrices over $P_{\mathbf K}(G)$ which are sent to invertible matrices by $\epsilon$. 

Write $R:=P_{\mathbf{K}}(G)$. For any $v\in G^0$ such that $s^{-1}(v)\neq\emptyset$ we put
 $s^{-1}(v)=\{e^v_1,\dotsc,e^v_{n_v}\}$, and we
consider the left $R$-module homomorphism
 \begin{align*}
  \mu_v\colon Rv&\longrightarrow \bigoplus_{i=1}^{n_v}Rr(e^v_i)\\
  r&\longmapsto\left(re^v_1,\dotsc,re^v_{n_v}\right)
 \end{align*}
 Write $\Sigma_1=\{\mu_v\mid v\in G^0,\,s^{-1}(v)\neq \emptyset\}$, and define $Q_{\mathbf K}(G) := P_{\mathbf K}(G) (\Sigma \cup \Sigma_1)^{-1}$
\end{pgr}

\begin{thm}[\cite{Ara19Poset}]
\label{thm:Poset-of-Fields}
With the previous notation, the following properties hold:
\begin{itemize}
 \item[(a)] $Q_{\mathbf{K}}(G)$ is a hereditary von Neumann regular ring.
 \item[(b)] The natural map $\mathcal M(G) \to V (Q_{\mathbf K}(G))$ is a monoid isomorphism. 
\end{itemize}
\end{thm}

The algebra $Q_{\mathbf K}(G)$ is called the \emph{regular algebra of $G$ over the poset of fields}\index[terms]{regular algebra!of building block}\index[symbols]{$Q_{\mathbf K}(G)\quad$(regular algebra of building block)} $\mathbf K$. Note that it is an algebra over $K_0$ (where $K_0=K_{i_0}$).    

\begin{prp}
\label{prp:realization-for Posets}
Let $(G,\sigma_G)$ be a building block. Then, there is a poset of fields $\mathbf K$ such that 
$Q_K(G,\sigma_G) \cong Q_{\mathbf{K}}(G)$. In particular, we obtain that $Q_K(G,\sigma_G)$ is a separative von Neumann regular ring and the natural map
$\mathcal M (G)\to V(Q_K(G,\sigma_G))$ is an isomorphism. 
\end{prp}

\begin{proof}
We follow the argument in \cite{AraBosPar20Sel}. The main difficulty is to choose a suitable poset of fields $\mathbf K =\{K_p \mid p\in I\}$, where $(I,\le)$ is the poset associated to the path-way relation on $G^0$. We set $\sigma:= \sigma_G$ and $\Lambda_p := \Lambda_p^G$.

$I$ is a tree, so let $p_0$ be its maximum element. For each element $p\in I$, there is a unique maximal chain
\[
    p=p_i < p_{i-1} < \cdots < p_1<p_0.
\]
Recall the notation $\NN_p= \Lambda_p\sqcup \NN$ for $p\in I$ (where $\Lambda_p=\emptyset$ whenever $p\in \Ireg$). We set $\sigma^p= \text{id}_\NN$ whenever $p\in \Ireg$, and write 
\[
    \NN = \sigma^{p} (\NN_p) \sqcup \Omega_p
\]
for each non-minimal $p\in I$ (note that $\Omega_p=\emptyset $ if $p\in \Ireg$). 

We will define inductively countable sets $\NN_p'$ such that $\NN_p' \subseteq \NN_q'$ whenever $q\le p$, and bijections $\psi_p \colon \NN_p\to \NN'_p$ for all $p\in I$ satisfying certain properties. 
The fields $K_p$ will then be defined by
\[
    K_p:= K(t_l: l\in \NN'_p ),
\]
so that $\mathbf K = \{K_p \mid p\in I\}$ will be a poset of fields. 

Now, to start, take $\NN_{p_0}' = \NN_{p_0}$ and $\psi _0=\text{id}_{\NN_{p_0}}$. For $p\in I$ with $p\ne p_0$, the sets $\NN_p'$ and the maps $\psi_p\colon \NN_p\to \NN_p'$ are recursively defined as follows:

Let $p=p_i< p_{i-1} < \cdots < p_0$ be the unique maximal chain starting with $p$.
Assume that we have defined $\NN_{p_j}'$ and $\psi_{p_j}\colon \NN_{p_j} \to \NN_{p_j}'$ for $0\le j\le i-1$. Then, set 
\begin{itemize}
\item[(1)] $\NN_{p_i}' = \Lambda_{p_i} \sqcup \NN_{p_{i-1}}' \sqcup \Omega_{p_{i-1}}'$, where $\Omega _{p_{i-1}}'$ is a set disjoint from $\Lambda_{p_i} \sqcup \NN_{p_{i-1}}'$
such that there is a bijection $\tau_{p_{i-1}} \colon \Omega_{p_{i-1}} \to \Omega_{p_{i-1}}'$. 
\item[(2)]
$\psi_i\colon \NN_{p_i}= \Lambda_{p_i} \sqcup \sigma^{p_{i-1}}(\NN_{p_{i-1}}) \sqcup \Omega_{p_{i-1}} 
\to \Lambda_{p_i} \sqcup \NN_{p_{i-1}}' \sqcup \Omega_{p_{i-1}}'=\NN_{p_i}'$
is given by 
\[
    \psi_i = \text{id}_{\Lambda_{p_i}} \sqcup \psi_{p_{i-1}} \circ (\sigma^{p_{i-1}})^{-1} \sqcup \tau_{p_{i-1}}.
\]
\end{itemize}

We have now defined a poset of fields $\mathbf K = \{ K_p \mid p\in I\}$, and can thus consider the map
\[
    \varphi \colon Q_K(G,\sigma_G) \longrightarrow Q_{\mathbf K}(G)
\]
by sending the generators $G^0\cup G^1\cup (G^1)^*$ to the corresponding generators in $Q_{\mathbf K}(G)$, and 
\[
    \varphi (t_l^v) = t_{\psi_{[v]}(l)} v\in Q_{\mathbf K}(G)
\]
for $v\in G^0$ and $l\in \NN_{[v]}$.
 
We need to show that the defining relations of $Q_K(G,\sigma_G)$ are preserved by $\varphi$. This is obvious for all relations except for relations (ii) (in the case where $e$ is a connector) and (iv)(a). These relations can be formulated in a uniform way as 
\[
    t_l^v \beta = \beta t^{r(\beta)}_{\sigma^p(l)}
\]
for all $l\in \NN_p$, where $v\in G^0$, $p=[v]$, and $\beta $ is a connector with $s(\beta) = v$. Now set $p_i=[r(\beta)]$, and let $p_i<p_{i-1}< \dots < p_0$ be the unique maximal chain starting with $p_i$. We then have $p= [v]= p_{i-1}$, and we compute
\begin{align*}
 \varphi(\beta) \varphi( t^{r(\beta)}_{\sigma^{p_{i-1}}(l)}) &=  \beta \varphi (t^{r(\beta)}_{\sigma^{p_{i-1}}(l)})\\
&  =\beta (t_{\psi_{p_i}(\sigma^{p_{i-1}}(l))}r(\beta))=t_{\psi_{p_i}(\sigma^{p_{i-1}}(l))} (\beta r(\beta)) \\
& = t_{\psi_{p_i}(\sigma^{p_{i-1}}(l))} \beta =  t_{\psi_{p_{i-1}}\circ (\sigma^{p_{i-1}})^{-1}(\sigma^{p_{i-1}}(l))} \beta \\
& =  t_{\psi_{p_{i-1}}(l)} \beta=  (t_{\psi_{p_{i-1}}(l)}v) \beta =  \varphi (t^v_l) \beta  =  \varphi (t^v_l ) \varphi ( \beta).
\end{align*}
This shows the desired property and thus $\varphi$ gives a well-defined $K$-algebra homomorphism
 $\varphi \colon \mathcal  S_K(F,\sigma )\to Q_{\mathbf K}(G)$. It is easily checked that this induces 
 a homomorphism, also denoted by $\varphi$, from $Q_K(G,\sigma_G)$ to $Q_{\mathbf K}(G)$. By using \cite[Proposition 2.6]{Ara19Poset}, we obtain a well-defined inverse map $\varphi^{-1}$ from $Q_{\mathbf K}(G)$ onto $Q_K(G,\sigma_G)$.
 Hence $Q_{\mathbf K}(G)$ and $Q_K(G,\sigma_G)$ are isomorphic, and the result now follows from 
 \cite{Ara19Poset} and \cite[Theorem 3.6.21]{AAS}.
 \end{proof}


The next step is to reconstruct the triple $(E,C,\sigma)$, where $(E,C)$ has property (F), in terms of its building blocks. We need to show that
this process works well at the level of graphs, monoids, and algebras. For separated graphs, this is done exactly as in  \cite[Section 4]{AraBosPar20Sel}, and the part concerning the monoids goes in the same way. Our result is as follows:

\begin{thm}
    \label{thm:realization-finite-case-with-condition-F}
Let $(E,C)$ be a finite super adaptable separated graph satisfying condition (F), and let $(\Lambda,\sigma)$ be arbitrary auxiliary data for $(E,C)$.
Then, $Q_K(E,C,\sigma)$ is a separative von Neumann regular ring and the natural map 
\[
    \mathcal M (E,C)\to V (Q_K(E,C,\sigma))
\]
is a monoid isomorphism. 
\end{thm}

Let us sketch the main steps of the reconstruction process for algebras. It is noteworthy that the constructions in our setting are even more natural than those in \cite{AraBosPar20Sel}, because the way we define the sets $\Lambda_p$ and the maps $\sigma_p$ in our auxiliary data sets is more general and far less rigid than that of \cite{AraBosPar20Sel}; in particular, it allows us to extend/restrict the construction, thus being able to use it through inclusions, quotients, and direct limits, which is critical for obtaining our main result.

The inductive argument in this direction is given in \cite[Notation 4.6]{AraBosPar20Sel} when one has no auxiliary data. In our situation, which incorporates auxiliary data, one does the following:

\begin{pgr}{(The inductive step of the reconstruction)} 
\label{not:NotationPullbacks}
Let $(E,C)$ be a finite super adaptable graph satisfying condition {\rm (F)}. Denote by $(I,\leq)$ the natural associated poset and let $J$ be a lower subset of $\Ifree$ containing all the sinks of $E$. 
	
Let $v\in \Ifree$ such that $v$ is minimal in $\Ifree \setminus J$. We further assume that $|C_v|>1$. Let $\varphi \colon \Ifree \setminus (J\cup \{ v \}) \to C$ be a choice function, so that $\varphi (w) \in C_w$ for each $w\in \Ifree \setminus (J\cup \{ v \})$. We recall from \cite[Notation~4.6]{AraBosPar20Sel} that (by definition) 
\[
    E^1_\varphi := \Big(\bigsqcup _{w\in  \Ifree \setminus (J\cup \{v\})} \varphi (w)\Big) \sqcup \Big(\bigsqcup_{w\in E^0\setminus (\Ifree \setminus (J\cup \{v\}))} s_E^{-1} (w)
\Big)
\]
and $E_{\varphi}^0 : =s_E(E_{\varphi}^1) \cup r_E(E_{\varphi}^1)$. The source and range maps, and the structure of
$C^{\varphi}$ are the natural ones, making the inclusion map $(E_{\varphi}, C^{\varphi}) \to (E,C)$ a morphism in the category
SGr defined in \autoref{pgr:morphisms-in-SGr}. Let $(F,D)$ be the unique connected component of $(E_{\varphi},C^{\varphi})$ such that $v\in F^0$. We consider the triple $(F,D,\sigma_F)$, where $(\Lambda^F, \sigma_F)$ are arbitrary auxiliary data for $(F,D)$. 
    
Write $C_v = \{ X_1,\dots , X_r \}$, and let $\varphi_i \colon E^0\setminus J\to C$ be the unique choice function which extends $\varphi$ and such that $\varphi_i (v)= X_i$. Let $(F_i,D^i)$ be the unique connected component of $(E_{\varphi_i}, C^{\varphi_i})$ with $v\in F_i^0$. Here $(E_{\varphi_i}, C^{\varphi_i})$ has the same meaning as above, using now the set $J$ instead of the set $J\cup \{v\}$. The separated graph $(F_i,D^i)$ is a finite super adaptable graph. We will need to apply our induction hypothesis not to the restriction of $(F,D,\sigma_F)$ to $(F_i,D^i)$, but to a triple $(F_i,D^i,\sigma_i)$, where $\Lambda^i_{[v]}$ is a somewhat larger set than $\Lambda^F_{[v]}$. Hence the correct form of the induction hypothesis is that \autoref{thm:realization-finite-case-with-condition-F} holds for arbitrary triples $(F',D',\sigma')$, where $(F',D')$ is a complete subobject of $(E,C)$ which is determined by a choice function $\varphi'$ whose domain is strictly larger than the domain of $\varphi$, and $(\Lambda',\sigma')$ are arbitrary auxiliary data for $(F',D')$. \autoref{prp:realization-for Posets} shows that the base case holds under these assumptions, and there the domain of $\varphi'$ is the larger possible domain. Under this induction hypothesis, we proceed to show that \autoref{thm:realization-finite-case-with-condition-F} holds for $Q_K(F,D,\sigma_F)$.

Denote $H = \tilde T_F^0(v)$ and $H_i= \tilde T_{F_i}^0(v)$. Then $H$ is a hereditary $D$-saturated subset of $F^0$ and each $H_i$ is a hereditary $D^i$-saturated subset of $F_i^0$ and a hereditary and $D$-saturated subset of $F^0$. Note that, by \cite[Lemma 4.3]{AraBosPar20Sel}, we have $H=\bigsqcup_{i=1}^r H_i$. Now observe that
\[
    (F_i,D^i) = ((F/(\oplus_{j\ne i} H_j))^{\ad}, (D/(\oplus_{j\ne i} H_j))^{\ad} ) .
\]

Taking $\Lambda^i = \Lambda ^{(F/(\oplus_{j\ne i} H_j))^\ad}$ and  $\sigma_i = \sigma_{(F/(\oplus_{j\ne i} H_j))^{\ad}}$, we obtain auxiliary data for $(F_i,D^i)$ so
that
\[
    \Lambda^i_p = \Lambda^F_p \text{ if } p\ne [v],\text{ and } \Lambda^i_{[v]} = \Lambda^F_{[v]} \sqcup \{\alpha_j : j\ne i\},
\]
where $\alpha_k$ is the only loop in $X_k$, for $k=1,\dots , r$. 

Now take 
\[
    (\ol{F},\ol{D},\ol{\sigma}) : = ((F/H)^{\ad}, (D/H)^{\ad},\sigma_{(F/H)^{\ad}}) .
\]
Note that 
\[
    \ol{\Lambda}_p = \Lambda^F_p \text{ if } p\ne [v],\text{ and } \ol{\Lambda}_{[v]} = \Lambda^F_{[v]} \sqcup \{\alpha_i \mid 1\le i\le r \}.
\]
Let $\mathcal H$ be the ideal of $Q:=Q_K(F,D,\sigma_F)$ generated by $H$ and let $\mathcal H _i$ be the ideal of $Q_i:= Q_K(F_i, D^i,\sigma_i)$ generated by $H_i$. It follows from \cite[Proposition 2.22]{AraBosPar20Sel} that $\mathcal H = \oplus_{i=1}^r \mathcal H_i$. By \autoref{prp:quotients}, we have
\[
    Q_i \cong Q/(\oplus _{j\ne i} \mathcal H _j) \quad \text{ and }\quad  \ol{Q} \cong Q/\mathcal H .
\]
Let $\theta _i \colon Q \to Q_i$, $\rho_i\colon Q_i\to \ol{Q}$, and $\theta \colon Q\to \ol{Q}$ be the canonical maps provided by \autoref{prp:quotients}. We have that $\theta = \rho_i\circ \theta_i$ for all $i=1,\dots , r$, and $\theta_i\colon Q\to Q_i$ is the pullback (limit) of the $K$-algebra homomorphisms $\rho_i\colon Q_i\to \ol{Q}$, $i=1\dots, r$; see \cite[Proposition 4.12]{AraBosPar20Sel}. 

We can now apply our induction hypothesis to each $(F_i,D^i,\sigma_i)$ to obtain that $Q_i= Q_K(F_i,D^i,\sigma_i)$ is a separative von Neumann regular ring, and the natural map $\mathcal M (F_i,D^i)\to V(Q_i)$ is an isomorphism, for $i=1,\dots , r$. Now \cite[Proposition 4.15]{AraBosPar20Sel} holds in our setting. The key point is that, for each $i=1\dots , r$,
we have that $\partial _i ([t^v_{\alpha_i}]) = \sum _{\beta \in X_i'} [\beta \beta^*]$ (cf. \cite[page 51]{AraBosPar20Sel}) and that $[t^v_{\alpha_i}] \in (\rho_j)_* (K_1(e_jQ_je_j))$ for all $j\ne i$, since $\rho_j (t^v_{\alpha_i})= t^v _{\alpha_i}$ for all $j\ne i$, see \autoref{prp:quotients}.

It follows from \cite[Theorem 4.13]{AraBosPar20Sel} that $Q=Q_K(F,D,\sigma_F)$ is a separative von Neumann regular ring and that the natural map
$\mathcal M (F,D)\to V(Q_K(F,D,\sigma_F))$ is a monoid isomorphism. 

This concludes the proof of the induction step and, hence, that of \autoref{thm:realization-finite-case-with-condition-F}.

Finally, we can follow \emph{verbatim} all the steps in \cite[Section 5]{AraBosPar20Sel}, using \autoref{thm:reduction-theorem-for-conditionF} and \autoref{thm:realization-finite-case-with-condition-F}, to get a proof of \autoref{thm:realization-finite-case}.
\end{pgr}

\section{Proof of the general case}
\label{subsect:proof-general-case}

We can now put all the pieces together. We show:

\begin{thm}
    \label{thm:realization-general-case}
Let $(E,C)$ be a countable super adaptable graph, and let $(\Lambda^E,\sigma_E)$ be auxiliary data for $(E,C)$.
Then, $Q_K(E,C,\sigma_E)$ is a von Neumann regular ring and the natural map 
\[
    \mathcal M (E,C)\to V (Q_K(E,C,\sigma))
\]
is a monoid isomorphism. 
\end{thm}

To show this result, we first show that the algebra $Q_K(E,C,\sigma_E)$ can be written as a direct limit (colimit) of algebras 
associated to finite super adaptable graphs.

\begin{thm}
\label{thm:QK-is-an-injective-direct-limit}
Let $(E,C)$ be a countable super adaptable graph, and let $(\Lambda^E , \sigma_E)$ be auxiliary data for $(E,C)$.
Then, there exists an inductive system 
\[
    \{(E_n,C^n,\sigma_n), \iota_{n}\}_{n\in \NN}
\]
of finite super adaptable complete subgraphs $(E_n,C^n)$ of $(E,C)$, with auxiliary data $(\Lambda^n, \sigma_n)$, and injective $K$-algebra homomorphisms 
\[
    \varphi_n\colon Q_K(E_n,C^n,\sigma_n) \to Q_K(E_{n+1},C^{n+1},\sigma_{n+1})
\]
such that $(E,C)\cong \varinjlim(E_n,C^n)$ and $Q_K(E,C,\sigma_E) \cong \varinjlim Q_K(E_n,C^n,\sigma_n)$.
\end{thm}

\begin{proof}
    By \autoref{prp:SupAdapLim}, we can write $(E,C)$ as a direct limit in SGr of a sequence of finite complete subobjects $(E_n,C^n)$ such that the induced maps
    $\rho_n\colon \mathbb P (E_n,C^n)\to \mathbb P (E,C)$ send non-minimal free primes to free primes. We may now consider the restriction of the auxiliary data $(\Lambda^E,\sigma_E)$ to each $(E_n,C^n)$, which makes sense since non-minimal free primes of $(E_n,C^n)$ are sent to free primes in $(E,C)$, and regular primes in $(E_n,C^n)$ are sent to regular primes in $(E,C)$; see \autoref{prp:restriction}.
Concretely, we set $\Lambda^n_p = \Lambda ^E_p$ if $p$ is a non-minimal free element of $\mathbb P (E_n)$ or a minimal free element of $\mathbb P (E_n)$ such that $\rho_n(p)\in \Pfree (E)$, and we set $\Lambda^n_p = \emptyset$ if $p$ is a minimal free element in $\mathbb P (E_n)$ such that $\rho_n(p)\in \Preg (E)$. By \autoref{prp:restriction}, we have induced algebra homomorphisms 
\[
\varphi_n\colon Q_K(E_n,C^n,\sigma_n) \to Q_K(E_{n+1},C^{n+1}, \sigma_{n+1}),
\]
and
\[
\Phi_n \colon Q_K(E_n,C^n,\sigma_n) \to Q_K(E,C, \sigma_E)
\]
sending the generators of $Q_K(E_n,C^n,\sigma_n)$ to the corresponding elements in these algebras. It is clear that $\Phi_n= \Phi_{n+1}\circ \varphi_n$ for all $n\in \NN$, and thus we have an induced
map $\Phi\colon \varinjlim Q_K(E_n,C^n, \sigma_n) \to Q_K(E,C,\sigma_E)$. Additionally, the map $\Phi$ is clearly surjective. In order to show it is also injective, we may define an inverse map
$\Psi \colon Q_K(E,C,\sigma_E) \to \varinjlim Q_K(E_n,C^n,\sigma_n)$ by first defining a homomorphism $\mathcal S_K(E,C,\sigma_E) \to \varinjlim Q_K(E_n,C^n,\sigma_n)$ and then
extending it to $Q_K(E,C,\sigma_E)$ by using the universal property of universal localization. 

Finally, we show that the maps $\varphi_n \colon Q_K(E_n,C^n,\sigma_n) \to Q_K(E_{n+1},C^{n+1},\sigma_{n+1})$ are injective. For this, consider the following commutative diagram
\[
\begin{tikzcd}[column sep=huge, row sep=large]
\mathcal{M}(E_n, C^n)
  \arrow[r, "\mathcal{M}(\iota_n)"]
  \arrow[d, "\cong"']
&
\mathcal{M}(E_{n+1}, C^{n+1})
  \arrow[d, "\cong"]
\\
V\bigl(Q_K(E_n, C^n, \sigma_n)\bigr)
  \arrow[r, "V(\varphi_n)"]
&
V\bigl(Q_K(E_{n+1}, C^{n+1}, \sigma_{n+1})\bigr)
\end{tikzcd}
\]

Since $Q_K(E_n,C^n,\sigma_E)$ is a von Neumann regular algebra, its lattice of two-sided ideals is naturally isomorphic to the lattice of order-ideals
of its $V$-monoid  $V(Q_K(E_n,C^n,\sigma_E))\cong \mathcal{M}(E_n,C^n)$ (by \autoref{thm:realization-finite-case}). In turn, the order-ideals of 
$\mathcal M (E_n,C^n)$ are in bijective correspondence with the hereditary and $C^n$-saturated subsets of $E_n^0$ (\cite[Corollary 6.10]{AraGoo12Crelle}). 
Since the map $\mathcal M (\iota_n)$ sends the generators $a_v$, $v\in E_n^0$, of $\mathcal{M}(E_n,C^n)$ to non-zero elements in $\mathcal{M}(E_{n+1},C^{n+1})$,
it follows that the kernel of $\varphi_n$ must be the trivial ideal, and thus $\varphi_n$ is injective. 
\end{proof}

We can now readily finish the proof of \autoref{thm:realization-general-case}. 
Indeed, we only need to apply \autoref{thm:QK-is-an-injective-direct-limit}, \autoref{thm:realization-finite-case},
and the continuity of both the $\mathcal{M}$-functor (\cite[Section 4]{AraGoo12Crelle}) and the $V$-functor, as follows.  


\medskip

\noindent\emph{Proof of \autoref{thm:realization-general-case}.}
 Let $(E,C)$ be a countable super adaptable graph and let $(\Lambda^E, \sigma_E)$ be auxiliary data for $(E,C)$. Then, by \autoref{thm:QK-is-an-injective-direct-limit}, there exists an inductive system $((E_n, C^n, \sigma_n), \iota_n)$ of finite super adaptable complete subgraphs with auxiliary data $(\Lambda^n,\sigma_n)$, and injective $K$-algebra homomorphisms
 \[
    \varphi_n\colon Q_K(E_n,C^n,\sigma_n) \to Q_K(E_{n+1},C^{n+1},\sigma_{n+1})
 \]
such that $Q_K(E,C,\sigma_E) \cong \varinjlim Q_K(E_n,C^n,\sigma_n)$. Since each $Q_K(E_n,C^n, \sigma_n)$ is von Neumann regular by \autoref{thm:realization-finite-case}, it follows that $Q_K(E,C,\sigma_E) $ is a von Neumann regular ring, Moreover, again by \autoref{thm:realization-finite-case}, the natural map
\[
    \mathcal{M}(E_n,C^n)\rightarrow V(Q_K(E_n,C^n,\sigma_n))
\]
is a monoid isomorphism for each $n\ge 1$.

By continuity of the functors $\mathcal{M}(\cdot)$ and $V(\cdot)$, we have that
\[
    \mathcal{M}(E,C)\cong \varinjlim\mathcal{M}(E_n,C^n)\cong \varinjlim V(Q_K(E_n,C^n,\sigma_n))\cong V(Q_K(E,C,\sigma_E)).
\]
This completes the proof. \qed
\vspace{.2truecm}


We can now state and prove the main result of the paper.

\begin{thm}\label{thm:Main}
    Let $M$ be a countable super adaptable monoid, and let $K$ be any field. Then, there exists a von Neumann regular $K$-algebra $R$ such that $V(R)\cong M$.
\end{thm}
\begin{proof}
 Let $M$ be a countable super adaptable monoid. By \autoref{thm:MondToGraph} there exists a countable super adaptable graph $(E,C)$ such that $M\cong \mathcal{M}(E,C)$. Take arbitrary auxiliary data $(\Lambda, \sigma)$ for $(E,C)$. Then, by \autoref{thm:realization-general-case}, $Q_K (E,C,\sigma)$ is a von Neumann regular ring and the natural map
 $\mathcal{M} (E,C) \to V(Q_K(E,C,\sigma))$ is an isomorphism. This gives the result.
 \end{proof}

\begin{cor}\label{cor:RegRealizable}
    Every countable primely generated regular conical refinement monoid can be realized by a regular ring.
\end{cor}
\begin{proof}
    It is clear from \autoref{dfn:superadaptableIsystem} that any countable  primely generated regular conical refinement monoid is super adaptable. Thus, the result follows from \autoref{thm:Main}.
\end{proof}


\chapter{A characterization of super adaptability}

In this final chapter we characterize countable super adaptable monoids (\autoref{sec:SupAdaIndLim}) and also countable primely generated regular  refinement monoids (\autoref{sec:RegSAGraph}). Interestingly, both characterizations involve natural classes of (separated) graph monoids.

\section{Super adaptable monoids as direct limits}\label{sec:SupAdaIndLim}

\begin{ntn}
    Let $M$ be a primely generated conical refinement monoid. We denote by $\Pmin (M)$\index[symbols]{$\Pmin (M)\quad$(set of minimal primes)} the set of minimal primes of $M$, and we denote by $\Pnmfree (M)$ (respectively $\Pmfree(M)$) the set of non-minimal (respectively, minimal) free primes of $M$. \index[symbols]{$\Pnmfree (M)\quad$(set of free non-minimal primes)}\index[symbols]{$\Pmfree (M)\quad$(set of free minimal primes)}
\end{ntn}

Let $f\colon M\to N$ be a monoid homomorphism between two primely generated conical refinement monoids $M$ and $N$. 
If $f$ sends primes of $M$ to primes of $N$, then we have a well-defined induced map $\mathbb P (\ol{f}) \colon \mathbb P (\ol{M}) \to \mathbb P (\ol{N})$ defined by $\mathbb P (\ol{f}) (\ol{p}) = \ol{f(p)}$ for $\ol{p}\in \mathbb P (\ol{M})$, where recall that $\ol{M}$ is the antisymmetrization of $M$, which is a primitive monoid; see \autoref{pgr:FreeRegPrimes}.  

\begin{dfn}
    \label{dfn:free-injective}
Let $f\colon M \to N$ be a monoid homomorphism between primely generated conical refinement monoids $M$ and $N$. We say that 
$f$ is \emph{free-injective}\index[terms]{free-injective morphism} if the following conditions are satisfied:
\begin{enumerate}
    \item[(a)] $f$ sends the primes of $M$ to primes of $N$;
\item[(b)]  $\mathbb P (\ol{f}) ( \Pnmfree (\ol{M})) \subseteq \Pnmfree (\ol{N})$;
\item[(c)] for each $p\in \Pnmfree (\ol{N})$, we have $|\mathbb P (\ol{f})^{-1} (\{p\})| \le 1$.
\end{enumerate}
    \end{dfn}

\begin{thm}
\label{thm:charac-super-adaptable}
    Let $M$ be a countable conical refinement monoid. The following conditions are equivalent:
\begin{itemize}
    \item[(i)] $M$ is a super adaptable monoid.
    \item[(ii)] $M\cong \mathcal M (E,C)$, where $(E,C)$ is a countable super adaptable graph.
    \item[(iii)] $M\cong \varinjlim_{n\in \NN} (M_n,\gamma_n)$ for a directed system $(M_n,\gamma_n)$, where each $M_n$ is a finitely generated conical refinement monoid, and each $\gamma_n\colon M_n\to M_{n+1}$ is a free-injective monoid homomorphism.   
\end{itemize}
    \end{thm}

    \begin{proof}
(i)$\implies $(ii) follows from \autoref{thm:MondToGraph}.

For (ii)$\implies $(iii), suppose that $M\cong \mathcal M (E,C)$ for a super adaptable graph $(E,C)$. 
By \autoref{prp:SupAdapLim} and the continuity of the $\mathcal M$-functor, we have 
\[
    M\cong \mathcal M (E,C) \cong \varinjlim \mathcal M (E_n,C^n),
\]
where $(E_n,C^n)$ are finite complete super adaptable subgraphs of $(E,C)$ with the property that the induced maps
$\rho_n \colon \mathbb P (E_n) \to \mathbb P (E_{n+1})$ send non-minimal free primes to free primes.
Set $M_n:= \mathcal M (E_n,C^n)$, for $n\in \NN$, and let $\gamma_n \colon M_n\to M_{n+1}$ be the monoid homomorphisms induced by the inclusions $(E_n,C^n) \to (E_{n+1},C^{n+1})$ in SGr. Observe that $M_n$ are finitely generated conical refinement monoids.

Recalling from \autoref{lma:identifing-the-primes} that 
$\mathbb P (F)$ is canonically isomorphic to $\mathbb P (\ol{\mathcal M (F,D)})$ for any super adaptable graph $(F,D)$, we immediately deduce that the maps $\gamma_n$ send primes of $M_n$ to primes of $M_{n+1}$, so that condition (a) in \autoref{dfn:free-injective} is satisfied. Hence, there exists a well-defined monoid homomorphism $\mathbb P (\ol{\gamma}_n)\colon \mathbb P (\ol{M}_n)\to \mathbb P (\ol{M}_{n+1})$. Let $p\in \Pnmfree (\ol{M}_n)$, for $n\ge 1$. Then $p$ corresponds to a non-minimal free prime in $\Pfree (E_n)$, with associated vertex $v^p\in E_n^0$. By construction, $p$ is also a free prime in $E_{n+1}$, and $v^p$ is not a sink in $E_{n+1}$ because it is not a sink in $E_n$, which shows $\mathbb P (\ol{\gamma}_n) (p)\in \Pnmfree (E_{n+1})$. This proves (b) in \autoref{dfn:free-injective}.

Now assume that $p\in \Pnmfree (E_{n+1}) = \Pnmfree (\ol{M}_{n+1})$. Then, the strongly connected component of $p$ in $E_{n+1}$ contains a unique vertex $v^p$. Suppose that $p= \mathbb P (\ol{\gamma}_n) (q)$ for some $q\in \mathbb P (\ol{M}_n)= \mathbb P (E_n)$. Then, we must have that $q\in \Pfree (E_n)$, hence the strongly connected component $(E_n)_q$
contains a unique vertex $v^q$. Since $\ol{\gamma}_n (q)=p$, it follows that $v^p=v^q\in E_n^0$ and thus $q$ is the only
element of $\mathbb P (E_n)= \mathbb P (\ol{M}_n)$ such that $\mathbb P (\ol{\gamma}_n) (q)= p$. This shows condition (c) in \autoref{dfn:free-injective} and, consequently, that $\gamma_n$ is a free-injective homomorphism. 

For (iii)$\implies $(i), observe first that the hypothesis that $\gamma_n$ sends primes of $M_n$ to primes of $M_{n+1}$, along with the hypothesis that all the monoids $M_n$ are conical and primely generated, imply that all the maps $\gamma_n$ are conical, that is, $a=0$ whenever $\gamma_n(a) =0$. If we denote by $\psi_n\colon M_n\to M$ the canonical maps into the direct limit, then we conclude that all the maps $\psi_n$ are also conical.

We first show that $M$ is primely generated. For this, write $\gamma_{n,m} = \gamma _{n-1}\circ \cdots \circ \gamma_{m}\colon M_m\to M_n$ for $n>m$, with $\gamma_{n,n}= \text{Id}_{M_n}$. Note that, since $\gamma_n$ sends primes to primes and $\psi_n= \psi_{n+1}\circ \gamma_n$ for all $n\ge 1$, we have 
inclusions
\[
    \psi_1(\mathbb P (M_1)) \subseteq \psi_2(\mathbb P (M_2)) \subseteq \cdots \subseteq M.
\]
\medskip 

\noindent\hypertarget{Claim1.7.3}{\textbf{Claim 1.}} \emph{We have
\begin{equation}
    \label{eq:main-identity-of-prime-sets}
    \mathbb P (M) = \bigcup_{n=1}^{\infty} \psi_n (\mathbb P (M_n)). 
\end{equation}}
\medskip 

\noindent\emph{Proof of Claim 1.} Take $p\in \mathbb P (M_n)$ for some $n\ge 1$, and suppose that $\psi _n(p) \le a+b$ in $M$.
Take $d\in M$ such that $\psi _n(p)+d = a+b$. Then, there is $m\ge n$ and $a',b',d'\in M_m$ such that $\gamma_{m,n}(p) + d'= a'+b'$, $\psi_m(a')=a$, $\psi_m(b')=b$ and $\psi_m(d')=d$. 
Since $\gamma_{m,n}(p)$ is prime in $M_m$, we deduce that either $\gamma_{m,n}(p) \le a'$ or $\gamma_{m,n}(p) \le b'$. Therefore we have that either $\psi_n(p)\le a$ or $\psi _n(p)\le b$. 
This shows that $\bigcup_{n=1}^{\infty} \psi_n (\mathbb P (M_n))\subseteq \mathbb P (M)$. 

For the other inclusion, take $p\in \mathbb P (M)$. There is $a\in M_n$, for some $n\ge 1$, such that $\psi_n (a) = p$. Write $a= \sum_{i=1}^k p_i$ for some $p_i\in \mathbb P (M_n)$. Since $p$ is prime in $M$, there exists $i\in \{1,\dots , k\}$ such that $p\le \psi_n(p_i)$.
It follows that 
\[
    \psi_n (p_i) \le \psi_n(a) = p \le \psi_n (p_i).
\]
Hence there exists some $m\ge n$ such that 
\[
    \gamma_{m,n} (p_i)\le \gamma_{m,n} (a) \le \gamma_{m,n} (p_i)
\]
in the monoid $M_m$. Since $\gamma_{m,n}(p_i)$ is prime in $M_m$, it follows that $\gamma_{m,n}(a)$ is prime in $M_m$. We conclude that 
$p= \psi_m (\gamma_{m,n}(a)) \in \psi_m (\mathbb P (M_m))$, as desired. 
\medskip

Now, observe that \hyperlink{Claim1.7.3}{Claim 1} immediately implies that $M$ is primely generated, and it remains to show that $M$ is super adaptable. For this we need to define a suitable mark $\mathfrak m$ on $M$ and show that the $I$-system $\mathcal I (M,\mathfrak m)$ is super adaptable. We recall from \autoref{exa:ExasSupAdaptMon}~(a) the important fact that $\mathcal I (N,\mathfrak m)$ is a super adaptable system for each finitely generated conical refinement monoid $N$ and \emph{any} mark $\mathfrak m$ on $N$. 

Since $\psi_n$ sends primes to primes, there is a well-defined map $\mathbb P (\ol{\psi}_n) \colon \mathbb P (\ol{M}_n) \to \mathbb P (\ol{M})$. It is straightforward to show that  
$\ol{M} = \varinjlim (\ol{M}_n,\ol{\gamma}_n)$, with limiting maps $\ol{\psi}_n$. It follows from \eqref{eq:main-identity-of-prime-sets} that 
$\mathbb P (\ol{M}) = \bigcup_{n=1}^{\infty} \ol{\psi}_n (\mathbb P (\ol{M}_n))$, and note that the image by $\ol{\psi}_n$ of a regular prime of $\ol{M}_n$ is a regular prime
of $\ol{M}$. We have $\Pfree (\ol{M}) = \Pmfree (\ol{M}) \sqcup \Pnmfree (\ol{M})$.\medskip 

\noindent\hypertarget{Claim2.7.3}{\textbf{Claim 2.}} \emph{We have
\begin{enumerate}
\item[(a)] $\Pnmfree (\ol{M}) = \bigcup_{n=1}^{\infty} \ol{\psi}_n (\Pnmfree (\ol{M}_n))$.
\item[(b)] $\Pmfree (\ol{M}) = \bigcup _{n=1}^{\infty} \{ \ol{\psi}_n(q) : q\in \ol{M}_n \text{ and } \ol{\gamma}_{m,n}(q) \in \Pmfree (\ol{M}_m)\, \forall m\ge n \}. $
\end{enumerate}
}
\medskip 

\noindent\emph{Proof of Claim 2.} For (a), suppose that $q\in \Pnmfree (\ol{M}_n)$ for some $n\ge 1$. If $2\ol{\psi}_n(q) = \ol{\psi}_n (q) $, then there is $m\ge n$ such that 
$2\ol{\gamma}_{m,n} (q) = \ol{\gamma}_{m,n} (q)$, hence $\ol{\gamma}_{m,n}(q)$ is a regular prime, which contradicts condition (b) in \autoref{dfn:free-injective}. 
Hence $\ol{\psi}_n(q)$ is free. We show that it is non-minimal. Since $q$ is non-minimal in $\ol{M}_n$, it follows that there is another prime $p\in \mathbb P (\ol{M}_n)$ such that 
$p<q$, which implies that $q=p+q$. But now $\ol{\psi}_n(q)= \ol{\psi}_n (p)+ \ol{\psi}_n (q)$. If $\ol{\psi}_n (q)$ is minimal, then $\ol{\psi}_n(p)= \ol{\psi}_n(q)$, and thus $\ol{\psi}_n(q)$ is a regular prime, which
contradicts what we have just shown. Hence $\ol{\psi}_n (q) \in \Pnmfree (\ol{M})$, as required.

Conversely, suppose that $p\in \Pnmfree (\ol{M})$. There exists a prime $q\in \mathbb P (\ol{M}_n)$, for some $n\ge 1$, such that $\ol{\psi}_n(q) =p$. Note that $q\in \Pfree (\ol{M}_n)$. Since $p$ is non-minimal, there is $p'\in \mathbb P (\ol{M})$ such that $p'<p$. There is then some $m\ge n$ and some
$q'\in \mathbb P (\ol{M}_m)$ such that $\ol{\psi}_m(q')=p'$ and $ q' < \ol{\gamma}_{m,n} (q) $. This shows that $q'':=\ol{\gamma} _{m,n} (q)\in \Pnmfree (\ol{M}_m)$ and $p= \ol{\psi}_m(q'')\in \ol{\psi}_m (\Pnmfree (\ol{M}_m))$. This concludes the proof of (a).

We now show (b). If $p\in \Pmfree (\ol{M})$, then there is a prime $q\in \mathbb P (\ol{M}_n)$, for some $n\ge 1$, such that $\ol{\psi}_n(q) =p$. Then $q$ is necessarily a free prime in $\ol{M}_n$. 
By (a), $q$ is necessarily minimal in $\ol{M}_n$. The same argument can be applied to $\ol{\gamma}_{m,n}(q)$ for all $m\ge n$. A similar argument gives the other inclusion. This ends the proof of \hyperlink{Claim2.7.3}{Claim 2}.\medskip 

With these preliminaries at hand, we can now define the mark $\mathfrak m$ on $M$. Note that necessarily any mark $\mathfrak m$ on a primely generated conical refinement monoid $N$ must satisfy that $\mathfrak m_p= p$ for each $p\in \Pmfree (N)$, since in that case $N_p = p\NN$. Hence, the mark $\mathfrak m$ is completely determined by its value on the set $\Pnmfree (N)$.  

We first construct inductively marks $\mathfrak m_n$ on $M_n$ such that 
$\gamma_n (\mathfrak m_n) \subseteq \mathfrak m_{n+1}$ for all $n\ge 1$. We start with an arbitrary mark $\mathfrak m_1$ on $M_1$. Suppose that $n\ge 1$ and that we have defined marks $\mathfrak m_i$ on 
$M_i$ for all $1\le i \le n$, such that, for $1\le i <n$,
\[
    (\mathfrak m_{i+1})_{\ol{\gamma}_i(p)} =\gamma_i ((\mathfrak m_i)_p)
\]
for all $p\in \Pnmfree (\ol{M}_i)$ and for all $p\in \Pmfree (\ol{M}_i)$ such that $\ol{\gamma}_i (p) \in \Pnmfree (\ol{M}_{i+1})$. We now define $\mathfrak m_{n+1}$. 
Observe that, for each $p\in \Pnmfree (\ol{M}_n)$, the element $\gamma_n ((\mathfrak m_n)_p)$ belongs to the component 
$(M_{n+1})_{\ol{\gamma}_n (p)}$ of $M_{n+1}$, and, by condition (b) in \autoref{dfn:free-injective}, $\ol{\gamma}_n (p) \in \Pnmfree (\ol{M}_{n+1})$. 
Moreover, by condition (c) in \autoref{dfn:free-injective}, all elements in the family of free primes
\[
\begin{split}
    K :=&\{ \ol{\gamma} _n(p) : p\in \Pnmfree (\ol{M}_n)\}\\
    &\cup \{ \ol{\gamma}_n(p): p\in \Pmfree (\ol{M}_n) \text{ and } \ol{\gamma}_n (p)\in \Pnmfree (\ol{M}_{n+1}) \}
\end{split}
\]
are pairwise distinct, hence we can select a representative $(\mathfrak m_{n+1})_{\ol{\gamma}_n(p)}$ in each $\equiv$-class of these elements by the formula
\begin{equation}
\label{eq:compatibility-of-marks}
    (\mathfrak m_{n+1})_{\ol{\gamma}_n(p)} =\gamma_n ((\mathfrak m_n)_p) .
\end{equation}
Finally we define $(\mathfrak m_{n+1})_p$ as an arbitrary element in the $\equiv$-class of $p$, for all the elements $p\in \Pnmfree (\ol{M}_{n+1})\setminus K$.

Now define a mark $\mathfrak m$ on $M$ by setting, for each $p\in \Pnmfree (\ol{M})$, 
\begin{equation}
\label{eq:compatibility-of-marks-for-psi}
\mathfrak m_p = \psi_n ((\mathfrak m_n)_q),     
\end{equation}
where $q\in \Pfree (\ol{M}_n)$ and $\ol{\psi}_n(q) =p$. By \hyperlink{Claim2.7.3}{Claim 2}~(a), all elements $p\in \Pnmfree (\ol{M})$ can be written in this form, and by \eqref{eq:compatibility-of-marks}, the definition of 
$\mathfrak m_p$ does not depend on the chosen $n$ and $q$. Hence the above formula gives a well-defined mark $\mathfrak m$ on $M$. 

The final step consists in showing that the induced $I$-system $\mathcal I (M,\mathfrak m)$ is a s.a.~system. 
We will use the following facts: For each $n\in \NN$ and each $p\in \mathbb P (\ol{M}_n)$, the map $\psi_n$ induces a monoid homomorphism
$\psi_n^p \colon (M_n)_p \to M_{\ol{\psi}_n (p)}$ which in turn defines a group homomorphism, also denoted by $\psi_n^p$, from $G((M_n)_p)$ to $G(M_{\ol{\psi}_n (p)})$, which restricts to a group homomorphism $G_p \to G_{\ol{\psi}_n(p)}$. If $p\in \Pnmfree (\ol{M}_n)$, then $\ol{\psi}_n(p) \in \Pnmfree (\ol{M})$ (by \hyperlink{Claim2.7.3}{Claim 2}~(a)), and we denote also by $\psi^p_n$ the semigroup homomorphism 
\[
    \psi^p_n\colon (M_n)_p^{\mathcal I} = \NN\times G_p \to M_{\ol{\psi}_n(p)}^{\mathcal I} = \NN\times G_{\ol{\psi}_n(p)}
\]
defined by $\psi^p_n (t,g)= (t,\psi_n^p(g))$ for $g\in G_p$. If $p\in \Pmfree (\ol{M}_n)$, then note that $(M_n)_p^{\mathcal I} = \NN$ and that $p$ can be identified with a prime of $M_n$. If $\ol{\psi}_n(p) \in \Pfree (\ol{M})$, then define
$\psi^p_n\colon \NN \to  M_{\ol{\psi}_n(p)}^{\mathcal I}$ by $\psi_n^p (t)=(t,0_{G_{\ol{\psi}_n(p)}})$. Finally, if $p\in \Pmfree (\ol{M}_n)$ and $\ol{\psi}_n(p)\in \Preg (\ol{M})$, then define
$\psi^p_n\colon \NN \to  G_{\ol{\psi}_n(p)}$ by $\psi^p_n (t)= t\psi_n (p)$. In any case, we have a well-defined homomorphism $\psi_n^p \colon (M_n)_p^{\mathcal I} \to M_{\ol{\psi}_n (p)}^{\mathcal I}$.

If $p\in \Pnmfree (\ol{M}_n)$, $q\in \mathbb P (\ol{M}_n)$, and $q<p$, then necessarily $\ol{\psi}_n (q) <\ol{\psi}_n (p)$. This is clear whenever $q\in \Preg (\ol{M}_n)$, because
then $\ol{\psi}_n(q) \in \Preg (\ol{M})$ and $\ol{\psi}_n(p)\in \Pnmfree (\ol{M})$ (by \hyperlink{Claim2.7.3}{Claim 2}~(a) above), and it follows from \autoref{dfn:free-injective}~(c) whenever 
$q\in \Pfree (\ol{M}_n)$ (use that $\ol{\gamma}_{m,n} (q) < \ol{\gamma}_{m,n}(p)$ for all $m\ge n$). We denote by $\varphi_{p,q}$, $q<p$, the homomorphisms associated to the $\mathbb P (\ol{M}_n)$-system
$\mathcal I (M_n,\mathfrak m_n)$, and we use the same notation for the ones associated to the $\mathbb P (\ol{M})$-system $\mathcal I (M,\mathfrak m)$; see \autoref{pgr:IM}.
Suppose that $p\in \Pnmfree (\ol{M}_n)$ for some $n\ge 1$ and that $q\in \mathbb P (\ol{M}_n)$ satisfies $q<p$. Then since $\ol{\psi}_n( q) <\ol{\psi}_n (p)$, we have a well-defined 
map $\varphi_{\ol{\psi}_n(p),\ol{\psi}_n(q)}\colon M_{\ol{\psi}_n(q)}^{\mathcal I} \to G_{\ol{\psi}_n(p)}$, and we have
\begin{equation}
    \label{eq:key-for-limit}
    \psi_n^p (\varphi _{p,q} (z)) = \varphi_{\ol{\psi}_n(p),\ol{\psi}_n(q)} (\psi_n^q(z))
\end{equation}
for all $z\in (M_n)_q^{\mathcal I}$. This is clear when $q\in \Preg (\ol{M}_n)$ and when $q\in \Pmfree (\ol{M}_n)$ and $\ol{\psi}_n (q) \in \Preg (\ol{M})$, and it follows from \eqref{eq:formula-for-varphi-m} and \eqref{eq:compatibility-of-marks-for-psi} when $q \in \Pfree (\ol{M}_n)$ and $\ol{\psi}_n (q) \in \Pfree (\ol{M})$. 

 Take $p\in \Pnmfree (\ol{M})$. We have to check
that $G_p$ is generated as a semigroup by 
\[
    (\cup_{q<p\text{ reg.}}\varphi_{p,q} (G_q))\cup \{ \varphi _{p,q}(q)\}_{q<p\text{ free}}.
\]
Take an arbitrary element $x\in G_p \subseteq G(M_p)$. Then $x=(p+a)-p$, where $a\in M$ and $p+a\le p$. By using \hyperlink{Claim2.7.3}{Claim 2}~(a), we may find $p'\in \Pnmfree (\ol{M}_n)$ and $a'\in M_n$ such that
$\psi_n(a')=a$, $\ol{\psi}_n (p')=p$, and $p'+a'\le p'$. Hence $(p'+a')-p'\in G_{p'} \subseteq  G((M_n)_{p'})$ and since $(M_n,\mathfrak m_n)$ is super adaptable there exist regular primes
$q_1,\dots ,q_r$ and free primes $q_{r+1},\dots ,q_{r+s}$ in $\ol{M}_n$, such that $q_i < p'$ for $1\le i\le r+s$, elements $z_i\in G_{q_i}$, $i=1,\dots , r$, and $l_j\in \NN$, $j=1,\dots ,s$, such that
\[
    (p'+a')-p'= \sum_{i=1}^r \varphi_{p',q_i} (z_i) + \sum_{j=1}^s l_j \varphi_{p',q_{r+j}} (1, 0_{G_{q_{r+j}}}).
\]
We now have, using \eqref{eq:key-for-limit},
 \begin{align*}
   x =  (p+a)-p & = \psi_n^{p'} ((p'+a')-p')\\
   &= \sum _{i=1}^r \psi_n^{p'} (\varphi_{p',q_i} (z_i)) + \sum_{j=1}^s l_j \psi_n^{p'} (\varphi _{p',q_{r+j}}(1,0_{G_{q_{r+j}}})) \\
     & = \sum _{i=1}^r \varphi_{p,\ol{\psi}_n(q_i)} (\psi_n^{\ol{\psi}_n(q_i)}(z_i)) + \sum_{j=1}^s l_j \varphi _{p,\ol{\psi}_n (q_{r+j})}(\psi_n^{\ol{\psi}_n(q_{r+j})}(1, 0_{G_{q_{r+j}}})) \\
& = \sum _{i=1}^r \varphi_{p,\ol{\psi}_n(q_i)} (\psi_n (z_i)) + \sum_{j=1}^s l_j \varphi _{p,\ol{\psi}_n (q_{r+j})}(\psi_n^{\ol{\psi}_n(q_{r+j})}(1, 0_{G_{q_{r+j}}})).
 \end{align*}
 
Using that $\psi_n^{\ol{\psi}_n(q_{r+j})}(1, 0_{G_{q_{r+j}}}) = (1, 0_{G_{\ol{\psi}_n(q_{r+j})}})$ whenever $\ol{\psi}_n(q_{r+j}) \in \Pfree (\ol{M})$, we conclude that $x$ belongs to the semigroup generated by 
$(\cup_{q<p\text{ reg.}}\varphi_{p,q} (G_q))\cup \{ \varphi _{p,q}(q)\}_{q<p\text{ free}}$.

This shows that $M$ is super adaptable, which finishes the proof. 
\end{proof}    

\section{Regular super adaptable monoids as graph monoids}\label{sec:RegSAGraph}

We now proceed to obtain a characterization of the class of countable primely generated regular refinement monoids. Recall from \cite[Theorem 3.1]{ParWeh14Unpubl} that all countable regular refinement monoids are tame. However, not all of them are primely generated. 
For example, the antisymmetric regular refinement monoids are precisely the distributive join-semilattices, and one can readily exhibit a countable distributive join-semilattice which is not primely generated. Indeed, one may consider the 
Boolean algebra $B$ of all open compact subsets of a Cantor set. In this case there are no primes in $B$, so that in particular $B$ is not primely generated.

We obtain the following characterization.

\begin{thm}
    \label{thm:charac-regular-primely-generated}
Let $M$ be a conical regular refinement monoid. Then, the following conditions are equivalent:
\begin{itemize}
    \item[(i)] $M$ is a countable primely generated monoid. 
\item[(ii)] $M\cong \mathcal M (E)$ where $E$ is a countable super adaptable graph such that $\Pfree (E)= \emptyset$ (and hence $E$ is endowed with the trivial separation).
\item[(iii)] $M\cong \mathcal M (E)$ for a row-finite countable graph $E$.
\end{itemize}
Moreover, in this case $M$ is realizable by a purely infinite von Neumann regular ring and by a purely infinite real rank zero graph $\mathrm{C}^*$-algebra. 
  \end{thm}

In particular, \autoref{thm:charac-regular-primely-generated} describes precisely which regular conical refinement monoids appear as graph monoids, thus generalizing \cite[Theorem~3.6]{AraPar17JAlg}. Note that, by the theorem above, regular graph monoids are primely generated. This is not true in general, as the following example shows:

\begin{exa}[{\cite[Comments before Theorem~6.3]{AraMorPar07NonStab}}]\label{exa:GraphNotPG}
    Consider the graph
    \begin{center}
\begin{tikzpicture}[>=Stealth, rounded corners, scale=0.7, transform shape]
  \tikzstyle{vertex} = [.]

    \node[vertex] (b1) at (0,2.5) {$p_0$};
  \node[vertex] (b2) at ([xshift=2.5cm]b1) {$p_1$};
  \node[vertex] (b3) at ([xshift=2.5cm]b2) {$p_2$};
  \node (dots) at ([xshift=2.5cm]b3) {$\dots$};       
  \node[vertex] (bn) at ([xshift=2.5cm]dots) {$p_n$};
  \node (ldots) at ([xshift=2.5cm]bn) {$\ldots$};     

  \node[vertex] (a) at (0,0) {$a$};

  \draw[->] (b1) -- (b2);
  \draw[->] (b2) -- (b3);
  \draw[->] (b3) -- (dots);
  \draw[->] (dots) -- (bn);
  \draw[->] (bn) -- (ldots);

  \draw[->] (b1) -- (a);
  \draw[->] (b2) -- (a);
  \draw[->] (b3) -- (a);
  \draw[->] (bn) -- (a);
\end{tikzpicture}
\end{center}
  The only prime element in its monoid is $a$, so it cannot be primely generated.
\end{exa}

We need to show that a situation as in \autoref{exa:GraphNotPG} cannot happen in case $\mathcal M (E)$ is regular. The key notion we need is that of a light vertex, as follows. 

\begin{dfn}
    \label{dfn:light-vertex}
Let $E$ be a row-finite graph. A vertex $v$ in $E$ is said to be a \emph{light vertex}\index[terms]{vertex!light} in case $v$ is not a sink and the strongly connected component $E_{[v]}$ of $v$ consists of the single vertex $v$ without any edges. 
\end{dfn}

Note that in \autoref{exa:GraphNotPG} all vertices $p_i$ are light vertices of $E$.

In the following, we use the description of the monoid $\mathcal M (E)$ of a row-finite graph $E$ in terms of the free commutative monoid
$F$ generated by $E^0$, introduced in \cite[Section 4]{AraMorPar07NonStab}.
The nonzero
elements of $F$ can be written in a unique form up to permutation
as $\sum _{i=1}^n v_i$, where $v_i\in E^0$ (repetition of elements is allowed). Now we give a
description of the congruence on $F$ generated by the relations \eqref{eq:definition-graph-monoid} of the graph monoid.  

 It will be convenient to introduce the
following notation, in analogy to that of \autoref{dfn:monoidsepgraph}. For a non-sink $v\in E^0$, write
\[
    {\bf r}(v):=\sum _{e\in s^{-1}(v)} r(e)\in F .
\]
With this new notation, the relations in \eqref{eq:definition-graph-monoid} become $v={\bf r}(v)$ 
for every $v\in E^0\setminus \text{Sink}(E)$.

\begin{dfn}\label{def:binary}
Define a binary relation $\rightarrow_1$ on $F\setminus \{0\}$ as follows. Let $\sum _{i=1}^n v_i\in F\setminus \{ 0 \}$, and suppose that $v_j$ is not a sink of $E$. Then $\sum _{i=1}^n v_i\rightarrow_1 \sum _{i\ne j}v_i+{\bf r}(v_j)$. Let $\rightarrow $ be the transitive and reflexive closure of $\rightarrow _1$ on $F\setminus \{0\}$, that is, $\alpha\rightarrow \beta$ if and only if there is a finite string $\alpha =\alpha _0\rightarrow _1 \alpha _1\rightarrow _1 \cdots \rightarrow _1 \alpha _t=\beta.$

Let $\sim$ be the congruence on $F$ generated by the relation $\rightarrow_1 $ (or, equivalently, by the relation $\rightarrow$). Namely $\alpha\sim \alpha$ for all $\alpha \in F$ and, for $\alpha,\beta \ne 0$, we have $\alpha\sim \beta$ if and only if there is a finite string $\alpha =\alpha _0,\alpha_1, \dots,\alpha _n=\beta$, such that, for each $i=0,\dots ,n-1$, either $\alpha _i\rightarrow _1 \alpha _{i+1}$ or $\alpha_{i+1}\rightarrow_1 \alpha _i$. The number $n$ above will be called the \emph{length} of the string.
\end{dfn}

It is clear that $\sim $ is the congruence on $F$ generated by relations \eqref{eq:definition-graph-monoid}, and so $\mathcal M(E) = F/{\sim}$.

The {\em support} of an element $\gamma$ in $F$, denoted $\text{supp}(\gamma)\subseteq E^0$, is the set of basis elements
appearing in the canonical expression of $\gamma$.

A basic and fundamental result is the \emph{confluence property} of the relation $\sim$ on $F$ \cite[Lemma 4.3]{AraMorPar07NonStab}.  This asserts that, for nonzero elements $\alpha,\beta $ in $F$,  $\alpha\sim \beta$ if and only if there is $\gamma \in F\setminus \{0\}$ such that $\alpha\to \gamma $ and $\beta \to \gamma$. 

\begin{lma}
\label{lma:non-light-are-prime}
Let $E$ be a row-finite graph. Then any non-light vertex of $E$ is prime in $\mathcal M (E)$. 
\end{lma}
\begin{proof}
We will work within the free commutative monoid $F$ generated by $E^0$, using the relations introduced before.
Suppose first that $v$ is a sink in $E$. If $v+\lambda = \alpha + \beta$ in $F$, then, by the confluence property, there is some $\gamma\in F$
such that $v+\lambda \to \gamma$ and $\alpha+\beta \to \gamma$. It is now clear that $v\in \text{supp} (\gamma)$, because $v$ is a sink.
On the other hand, by \cite[Lemma 4.2]{AraMorPar07NonStab}, we can write $\gamma = \gamma_1+\gamma_2$ with $\alpha\to \gamma_1$ and $\beta\to \gamma_2$. 
It follows that either $[v]\le [\gamma_1]=[\alpha]$ or $[v]\le [\gamma_2] = [\beta]$, and $[v]=a_v$ is prime in $\mathcal M (E)$. 
Suppose now that $v$ is not a sink in $E$. Since $v$ is not a light vertex, it follows from \autoref{lma:charcSPI}
that there is a cycle $c$ based at $v$. Now suppose that $v+\lambda = \alpha + \beta$ in $F$. As before, there is $\gamma\in F$ such that 
$v+\lambda \to \gamma$ and $\alpha+\beta \to \gamma$. Since $\gamma$ contains some vertex in the cycle $c$, there is $\gamma'\in F$ with $\gamma \to \gamma'$
such that $v$ belongs to the support of $\gamma'$. Replacing $\gamma$
 with $\gamma'$, we can therefore assume that $v\in \text{supp} (\gamma)$, and now the same argument we did before shows that either $[v]\le [\alpha]$
 or $[v]\le [\beta]$. Hence $a_v=[v]$ is a prime element of $\mathcal M (E)$. 
\end{proof}

Recall from \autoref{subsect:proof-finite-case} that for a vertex $v\in E^0$, we denote by $T(v)$ the tree of $v$ seen as a subgraph of $E$. The set $T^0(v)$ of vertices of $T(v)$ is precisely the set of vertices $w\in E^0$ such that there is a path from $v$ to $w$. 

For $X\subseteq E^0$, we denote by $\ol{X}$ the hereditary saturated closure of $X$.

We say that an element $a$ of a commutative monoid $M$ is \emph{primely generated} in case $a$
is a finite sum of prime elements of $M$. 




 We will use the following standard notation (see e.g. \cite{Bro01Canc}): given elements $a,b$ in a commutative monoid $M$, we write 
 $a\propto b$ if $a\le nb$ for some $n\ge 1$, and we write $a\asymp b$ if $a\propto b$ and $b\propto a$.

We can now provide the proof of \autoref{thm:charac-regular-primely-generated}. 

\medskip

\noindent\emph{Proof of \autoref{thm:charac-regular-primely-generated}.} (i)$\implies $(ii) follows from \autoref{thm:MondToGraph}, and (ii)$\implies $(iii) is trivial. For (iii)$\implies $ (i), we only need \autoref{lma:non-light-are-prime} and the results in \cite[Section 3.8]{AAS}.

Since $\mathcal M(E) \cong M$ is regular and $\mathcal M (E) \cong V(L_K(E)) $ for any fixed field $K$, it follows that each vertex $v\in E^0$ is properly infinite in $L_K(E)$; see \cite[Definition 3.8.3]{AAS}. 

Let $v$ be a vertex in $E$. It suffices to check that $a_v$ is primely generated. Since $v$ is properly infinite in $L_K(E)$, it follows from 
\cite[Proposition 3.8.12]{AAS} that there exists SPI vertices $w_1,\dots ,w_r$ in $T^0(v)$ such that $v\in \ol{\{w_1,\dots ,w_r\}}$.
Let $x= a_{w_1}+\cdots +a_{w_r}\in \mathcal M (E)$. Then, by \autoref{lma:non-light-are-prime}, $x$ is a primely generated element of $\mathcal M (E)$. 
By the argument given in the proof of \cite[Proposition 3.8.12]{AAS}, we have $x\asymp a_v$ in $\mathcal M(E)$. Hence, it follows from \cite[Theorem 5.18(3)]{Bro01Canc} that 
$a_v$ is also primely generated. This completes the proof of this implication.

For the last part, note that here we do not need our setting with auxiliary data in order to realize $M$ by a von Neumann regular ring, since the monoid $M$ is regular.
Using a graph $E$ satisfying either (ii) or (iii), we have from \cite{AraBru02Quiver} and \cite[Theorem 7.1]{AraMorPar07NonStab} that
\[
    M\cong \mathcal M (E) \cong V(L_K(E)) \cong V (C^*(E))\cong V(Q_K(E)),
\]
where $Q_K(E)$ is the regular algebra of the quiver $E$. 

As noted in the proof of (iii)$\implies $(i), every vertex of $E$ is properly infinite in $L_K(E)$. Thus, it follows from \cite[Theorem 3.8.16 and Corollary 3.8.19]{AAS} or \cite[Theorem 7.4]{ArandaGoodPerSil2010}, that $L_K(E)$ is a (properly) purely infinite exchange ring. Since $Q_K(E)$ is von Neumann regular and $V(Q_K(E))$ is a regular refinement monoid, $Q_K(E)$ is also a (properly) purely infinite ring. Finally, it follows from \cite[Theorems 2.3 and 2.5]{HS2003} that $C^*(E)$ is a purely infinite $\mathrm{C}^*$-algebra of real rank zero.  
\qed

\chapter{Final remarks and open questions}\label{chap:FROQ}

In this paper we have shown that a very large class of countable primely generated refinement monoids are realizable by von Neumann regular $K$-algebras, for an arbitrary field $K$. However, many questions in several directions remain open. After summarizing which classes are covered by our result in \autoref{sec:Chart}, we bundle these questions in three different sections: Realizing larger families, obtaining monoid models, and the $\mathrm{C}^*$-realization problem. 

\section{Diagram of Countable Conical Refinement Monoids}\label{sec:Chart}

The following diagram summarizes all types of conical refinement monoids discussed throughout the paper. To simplify the presentation, we focus on the countable setting. An explanation of each part can be found right below it.

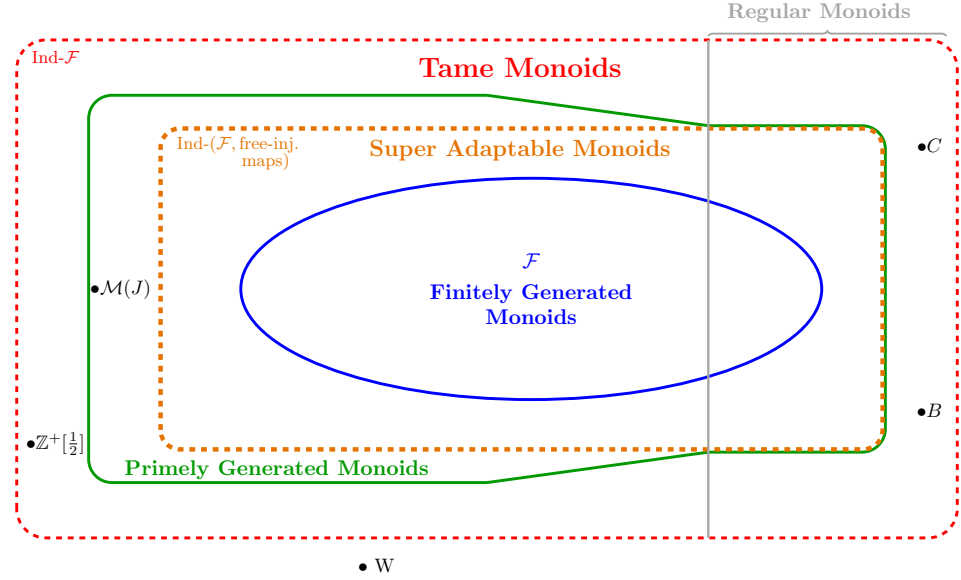
\begin{figure}[H]
    \resizebox{1\textwidth}{!}{
        \begin{tikzpicture}[>=Stealth, scale=1.0, transform shape]
    \draw[red, dashed, line width=1.5pt, rounded corners=15pt] (-8.5, -4.5) rectangle (8.5, 4.5);
    \node[text=red, font=\bfseries\LARGE] at (0.6, 4.0) {Tame Monoids};

    \draw[green!60!black, line width=1.5pt, rounded corners=12pt]
        (4.0, 2.95) -- (7.2, 2.95) 
        -- (7.2, -2.95) 
        -- (4.0, -2.95); 
    \draw[green!60!black, line width=1.5pt, rounded corners=12pt]
        (0, 3.5) -- (-7.2, 3.5) 
        -- (-7.2, -3.5) 
        -- (0, -3.5); 
    \draw[green!60!black, line width=1.5pt] (0, 3.5) -- (4.0, 2.95);  
    \draw[green!60!black, line width=1.5pt] (0, -3.5) -- (4.0, -2.95); 
        
    \node[text=green!60!black, font=\bfseries\large] at (-3.8, -3.27) {Primely Generated Monoids};

    \draw[orange!90!black, dashed, line width=2.3pt, rounded corners=10pt] (-5.9, -2.9) rectangle (7.15, 2.9);
    \node[text=orange!90!black, font=\bfseries\Large] at (0.6, 2.5) {Super Adaptable Monoids};

    \node[draw=blue, line width=1.5pt, ellipse, minimum width=10.5cm, minimum height=4cm, fill=white, inner sep=0pt] (fc) at (0.8, 0.0) {};
    \node[text=blue, align=center, font=\bfseries\large] at (0.8, 0.0) {$\mathcal{F}$ \\[0.2em] Finitely Generated\\Monoids};

    \draw[gray!70, line width=1.2pt] (4.0, -4.5) -- (4.0, 4.5);
    \draw[gray!70, decorate, decoration={brace, raise=2pt}, line width=1pt] (4, 4.5) -- (8.3, 4.5);
    \node[text=gray!70, font=\bfseries\large] at (6.0, 5.0) {Regular Monoids};

    \node[text=orange!90!black, font=\small] at (-4.5, 2.6) {Ind-($\mathcal{F},\text{free-inj.}$};
    \node[text=orange!90!black, font=\small] at (-4, 2.3) {maps)};
    
    \node[text=red, font=\small] at (-7.8, 4.2) {Ind-$\mathcal{F}$};
    \node[text=black, font=\normalsize] at (-7.8, -2.8) {$\bullet \mathbb{Z}^+[\frac{1}{2}]$ };
    
    \node[text=black, align=left, font=\normalsize] at (-2.0, -5.0) {$\bullet$ W};

\node[text=black, align=left, font=\normalsize] at (-6.62, 0.0) {$\bullet \mathcal{M} (J)$};

    \node[text=black, align=center, font=\normalsize] at (8.0, -2.2) {$\bullet B$};

    \node[text=black, align=center, font=\normalsize] at (8.0, 2.6) {$\bullet C$};
\end{tikzpicture}
    }
    \begin{minipage}{\textwidth}
        \small \,\\\\
        $B:=\{\text{compact open sets of the Cantor space}\}$. \\
        $C:=V(A)$ for some well-chosen extended Cuntz limit $A$. See 
        \cite{ParWeh06Semi} and \autoref{prbl:Cuntz-limit}. \\
        $\mathcal M (J)$ is the monoid associated to the $\ZZ$-system $J$ built in \autoref{exam:not-super-adaptable}, with $H=\ZZ_2$ or $H=\ZZ$.\\
        $W$ is a countable wild monoid constructed in \cite{AraGoo17RealProbAtiy}. $W$ is realizable by a regular 
        $K$-algebra if and only if the field $K$ is countable. 
    \end{minipage}
    \caption{Countable Conical Refinement Monoids Chart}
    \label{fig:Dibuix}
\end{figure}

Recall that \emph{tame monoids} are ---by definition--- the \emph{Ind-completion} of the category $\mathcal{F}$ of finitely generated conical refinement monoids and their morphisms. By the characterization obtained in \autoref{sec:SupAdaIndLim}, we see that \emph{super adaptable monoids} are the \emph{Ind-completion} of the subcategory of finitely generated conical refinement monoids with free-injective maps. Further, we see that this completion lies within the class of \emph{primely generated} conical refinement monoids. The main theorem of this paper (\autoref{thm:realization-general-case}) shows that any super adaptable monoid is realizable by a regular ring. 

By \cite[Theorem 3.1]{ParWeh14Unpubl}, any regular conical refinement monoid is tame and, by \autoref{cor:RegRealizable} any countable primely generated regular conical refinement monoid is super adaptable.

Lastly, we have illustrated ---when known--- inclusions that are strict. For instance, the \emph{dimension monoid} $\mathbb{Z}^+[\frac{1}{2}]$ of dyadic numbers is tame (since it can be written as the direct limit of the sequence $(\mathbb{Z}^+,\overset{\times 2}\rightarrow)$), and yet fails to be primely generated. Also, we have several examples of countable regular conical refinement monoids failing to be primely generated, or equivalently, super adaptable. See $B$ and $C$ in the diagram, and \autoref{sect:realizing-larger-families} for a complete discussion.
Finally, we build in \autoref{exam:not-super-adaptable} an example of a countable primely generated conical refinement monoid $M= \mathcal M (J)$
which is not super adaptable, where $J$ is a suitably defined $\ZZ$-system. 

In addition, there are many examples of wild monoids in the literature, see for instance \cite{AraGoo15SgpFor,AraGoo17RealProbAtiy, Ruz2018}. We have highlighted in the diagram the countable wild monoid $W$ studied in \cite{AraGoo17RealProbAtiy}. It is shown in that paper that $W$ is realizable by a regular $K$-algebra if and only if $K$ is a countable field.
(Note that the realization results in the present paper do not depend on the field.)

\section{Realizing larger families}
\label{sect:realizing-larger-families}

The next step in the programme is the realization of all countable tame refinement monoids. Note that all primely generated refinement monoids are tame (see \cite[Theorem 4.6]{AraPar16Isr}), and that we have obtained in \autoref{thm:charac-super-adaptable} precise information on the class of direct limits of finitely generated refinement monoids that give rise to super adaptable refinement monoids.

Looking at the complement of the class of all primely generated refinement monoids within the class of tame refinement monoids, we find a very classical class of monoids: dimension monoids. By definition, a dimension monoid is a cancellative, unperforated conical refinement monoid. As stated in \autoref{thm:Effros-Handel-Shen-Grillet}, these monoids were characterized in \cite{EHS1980,Grillet1976} as the directed limits of simplical monoids $(\ZZ^+)^n$. The following lemma, characterizing the dimension monoids which are primely generated, is elementary.

\begin{lma}
\label{lma:dimension-primely-generated}
The primely generated dimension monoids are exactly the free commutative monoids.  
\end{lma}

\begin{proof}
Let $M$ be a cancellative primely generated conical refinement monoid. Note that $M$ is necessarily antisymmetric, hence $M$ is a primitive monoid. By the characterization of primitive monoids (see \autoref{pgr:PrimMonod}), $M\cong M(D,\lhd)$, where $\lhd$ is a transitive antisymmetric relation on the set $D$, and $M(D,\lhd)$ is the quotient of the free commutative monoid on $D$ by the congruence generated by the relations $p=p+q$ for each $p,q\in D$ such that $q\lhd p$.
Since the elements of $D$ represent different primes in $M(D,\lhd)$, which is cancellative by hypothesis, it follows that $\lhd$ is the empty relation. Hence $M\cong M(D,\lhd)$ is the free commutative monoid on $D$.
    \end{proof}

We see from \autoref{lma:dimension-primely-generated} that the class of countable primely generated dimension monoids is very thin. The Bratteli diagram associated to such monoids cannot have neither bifurcations nor multiplicities. By the seminal work of Elliott \cite{Ell76Real}, all countable dimension monoids are realizable by von Neumann regular rings and by real rank zero $\mathrm{C}^*$-algebras. Hence this is a class of realizable tame refinement monoids which is not covered by the methods of the present paper. 

Note that countable tame refinement monoids are obtained by replacing simplicial monoids, which are the cancellative finitely generated conical refinement monoids, with general finitely generated conical refinement monoids as the monoids showing up in a direct limit representation. We have a good understanding of the constituents, thanks to the combinatorial structure of a finitely generated $I$-system from \cite{AraPar16Isr}. What remains to be explored is how to make a suitable analogue with the case of dimension monoids, yielding a combinatorial understanding of the morphisms between two such monoids, as follows.

\begin{prbl}
    \label{prbl:morphisms}
    Given two finitely generated conical refinement monoids $M$ and $N$, describe the monoid homomorphisms $f\colon M\to N$ combinatorially in terms of the associated finitely generated $I$-systems. 
\end{prbl}

For instance if $M=(\ZZ^+)^n$ and $N=(\ZZ^+)^m$ are two simplical monoids, then the 
corresponding $I$-systems can be identified with the finite sets $[1,n]$ and $[1,m]$
consisting of $n$ and $m$ free minimal primes respectively. A homomorphism
$f\colon M\to N$ can then be described in terms of a bipartite directed graph joining the
points in $[1,n]$ to the points in $[1,m]$ by edges corresponding to the non-negative integers 
appearing in the matrix associated to $f$ in the canonical basis. To answer \autoref{prbl:morphisms}, one would need to generalize this combinatorial picture to 
general homomorphisms between finitely generated conical refinement monoids. This would be a first step in the
construction of appropriate morphisms between von Neumann regular algebras realizing the building blocks, which we know exist thanks to the results in \cite{AraBosPar20Sel}.


Going to the opposite extreme in the family of tame refinement monoids, let us now consider some known cases of regular refinement monoids. Recall from \cite[Lemma 2.1]{GooParWeh05Semi} that a commutative monoid $M$ is regular if and only if it is a (join-)semilattice of abelian groups. Moreover, by \cite[Theorem 3.2]{GooParWeh05Semi}, the regular monoid $M$ is a refinement monoid if and only if its semilattice of idempotents is distributive and  $M$ satisfies the Mayer-Vietoris property (MVP) of \cite[p. 7]{GooParWeh05Semi}. When all associated abelian groups are trivial, then the regular refinement monoids are just the distributive semilattices.
Goodearl and Wehrung show in \cite[Theorem 6.6]{GooWeh01RepDist} that every distributive semilattice is the direct limit of finite Boolean semilattices and semilattice homomorphisms. Here a finite Boolean semilattice is a semilattice isomorphic to the Boolean 
algebra ${\bf 2}^n$ of subsets of an $n$-element set. Note that, in terms of the associated $I$-system, the Boolean semilattice 
${\bf 2}^n$ has a very simple structure, since all primes are minimal regular primes, and the associated groups are all trivial.
Nevertheless, we can express any distributive semilattice as a direct limit of such very simple building blocks, using arbitrary semilattice homomorphisms between them. 
On the other hand, Pudl\'ak's lemma \cite[Fact 4, p. 100]{Pudlak1985} states that every distributive semilattice is the {\it directed union} of its finite distributive semilattices. Here the building blocks are general antisymmetric finitely generated regular refinement monoids (so that all associated groups are trivial), but the connecting homomorphisms are injective. This also shows that injectivity of the connecting homomorphisms is not enough to guarantee that the limit object is primely generated, since there are countable distributive semilattices which are not primely generated.

In \cite{GooParWeh05Semi}, the authors generalize \cite[Theorem 6.6]{GooWeh01RepDist} by considering slightly larger groups in the associated $I$-systems of the building blocks. They consider building blocks $M$ which are finite direct sums of monoids of the form 
$\ZZ_n\sqcup \{0\}$ for $n\ge 1$, and they obtain in \cite[Theorem 6.4]{GooParWeh05Semi} an internal characterization of the monoids which are direct limits of this family of building blocks. Observe that all primes of these building blocks are minimal and regular, and the associated groups are just finite cyclic groups. Moreover, they show in \cite[Theorem 7.2]{GooParWeh05Semi} that all these monoids can be realized as the $V$-monoids of the family of Cuntz limits. Here a {\it Cuntz limit} is a $\mathrm{C}^*$-algebra which can be written as a sequential direct limit of finite direct products of algebras of matrices $M_n (\mathcal O _m)$, $1\le n<\infty, 2\le m<\infty$, where $\mathcal O _m$ is the Cuntz $\mathrm{C}^*$-algebra. This result was extended in \cite{ParWeh06Semi} to cover the case where the building blocks can also contain a summand of the form $\ZZ\sqcup \{0\}$, so that every building block is a finite direct sum of monoids $\ZZ_n\sqcup \{0\}$, where $n\ge 0$. The involved $\mathrm{C}^*$-algebras are the {\it extended Cuntz limits}; see \cite{ParWeh06Semi} for details. 
In particular these authors obtain a realization result by $\mathrm{C}^*$-algebras of real rank zero of a large class of regular refinement monoids, 
which lies beyond the class of primely generated regular conical refinement monoids. Observe however that this class does not exhaust all primely generated regular conical refinement monoids: If $M$ is a primely generated regular conical refinement monoid such that some of the homomorphisms between the associated groups in the $I$-system of $M$ is not injective, then, by \cite[Theorem 6.6]{ParWeh06Semi}, $M$ is not a direct limit of finite direct sums of monoids of the form 
$\ZZ_n\sqcup \{0\}$, for $n\ge 0$.

In view of the above results, we propose here the following open question, which is a special case of the realization problem for tame refinement monoids by
von Neumann regular rings.

\begin{prbl}
    \label{prbl:Cuntz-limit}
Let $M$ be a sequential direct limit of monoids which are finite direct sums of monoids of the form 
$\ZZ_n\sqcup \{0\}$, for $n\ge 0$. Let $K$ be a field. Is then $M$ realizable by a von Neumann regular $K$-algebra?
  \end{prbl}

Note that in \autoref{prbl:Cuntz-limit} all building blocks are graph monoids, but the direct limit is not necessarily
 a graph monoid. Indeed, by \autoref{thm:charac-regular-primely-generated}, the direct limit will be a graph monoid exactly when the monoid $M$ is primely generated.  


\section{Monoid models}

Concerning the realization project, one can wonder whether we can obtain a larger class of realizable monoids by considering graph monoids of arbitrary countable graphs, which do not need to be row-finite. However there is a process which associates to each countable graph $E$ another graph $F$ which is a countable row-finite graph, called a {\it desingularization} of $E$. It is shown in \cite{DriTom2005} that the graph $\mathrm{C}^*$-algebras $C^*(E)$ and $C^*(F)$ are strongly Morita-equivalent.
Similarly, it is shown in \cite{AbrAran2008} that the Leavitt path algebras $L_K(E)$ and $L_K(F)$ are Morita-equivalent. Hence the associated $V$-monoids are identical. Thus, we can restrict ourselves to the consideration of row-finite graphs.

A complete monoid-theoretic characterization of graph monoids, that is, monoids of the form $\mathcal M (E)$
for a countable row-finite graph $E$, is still missing:

\begin{prbl}
\label{prbl:charactrize-garph-monoids}
Find a monoid-theoretic characterization of graph monoids within the class of tame refinement monoids. 
Are all primely generated graph monoids super adaptable?
\end{prbl}

We know from \autoref{exa:GraphNotPG} that there are graph monoids which are not primely generated. Indeed there is a plethora of such examples. By \cite{Drinen} all countable dimension monoids are graph monoids, however by \autoref{lma:dimension-primely-generated} only the  trivial examples of dimension monoids are primely generated. On the other hand, all regular graph monoids are primely generated by \autoref{thm:charac-regular-primely-generated}. 

In a slightly different direction, it was shown in \cite[Corollary~7.6]{AraBosParSims21IMRN} that any finitely generated conical refinement monoid admits an inverse semigroup model in the following sense: Given any such monoid $M$, there exists an inverse semigroup $S$ such that its tight groupoid $\mathcal{G}:=\mathcal{G}_{\text{tight}}(S)$ is such that $M\cong {\rm Typ}(\mathcal{G})$, where ${\rm Typ}(\mathcal{G})$ is the type monoid of $\mathcal{G}$. More generally, the same is true for any monoid of the form $\mathcal{M}(E,C)$ for an adaptable graph $(E,C)$ \cite[Theorem~7.5]{AraBosParSims21IMRN}. Indeed, an explicit inverse semigroup $S(E,C)$ is associated to each adaptable separated graph $(E,C)$ in \cite[Section 2]{AraBosParSims21IMRN} and then it is shown in \cite[Theorem 7.5]{AraBosParSims21IMRN} that 
$$\mathcal M (E,C)\cong {\rm Typ} (\mathcal G_{\text{tight}}(S(E,C))).$$
We expect the same to also hold for super adaptable monoids by following the techniques from \cite{AraBosParSims21IMRN}. The first part of the following problem appeared in \cite[Remark~7.7]{AraBosParSims21IMRN}, and is currently being studied by some of the authors:

\begin{prbl}\label{prbl:saGroupoid}
    Let $M$ be a finitely generated conical refinement monoid, let $K$ be a field, and let $A_K(\mathcal{G})$ be the Steinberg algebra of the groupoid above (for an adaptable graph $(E,C)$ such that $M\cong \mathcal M (E,C)$). Is it true that $M\cong V(A_K (\mathcal{G}))$?
    In case this is true, can we extend the result to all countable super adaptable monoids, using a suitable definition of the inverse semigroup $S(E,C)$?
\end{prbl}

More generally, one could ask:
\begin{prbl}\label{prbl:RealGroupoid}
    For which (countable) conical refinement monoids $M$ does there exist an inverse semigroup $S$ such that $M\cong {\rm Typ}(\mathcal{G}_{\text{tight}}(S))$? For such monoids, does one have $M\cong V(A_K(\mathcal{G}_{\text{tight}}(S)))$ for a suitable $S$?
\end{prbl}

Indeed, using known results in the literature, we can provide a positive answer to the first part of \autoref{prbl:RealGroupoid}, although the inverse semigroups appearing in this approach are not explicitly described. The second part of \autoref{prbl:RealGroupoid} is open even for finitely generated conical refinement monoids.

Let $M$ be a countable conical refinement monoid.
By \cite[Theorem 4.8.9]{weh2017} (see also \cite[Theorem 4.26]{Ara26:Intro}) there exists a locally compact Hausdorff second countable zero-dimensional space $X$ and an action of a countable discrete group $G$ by homeomorphisms on $X$ such that $M\cong \text{Typ}(X,G,\mathbb K)$. Let $\mathcal G$ be the transformation groupoid associated to the action of $G$ on $X$. Observe that $\mathcal G$ is a second countable ample Hausdorff groupoid. Let $S$ be the ample semigroup of $\mathcal G$, consisting of all the compact open bisections of $\mathcal G$. Then $S$ is an inverse semigroup and by \cite[Theorem 4.8]{Exel2010}, 
$$\mathcal G_{\text{tight}}(S)\cong \mathcal G.$$
Hence 
$$\text{Typ}(\mathcal G_{\text{tight}}(S)) \cong \text{Typ} (\mathcal G)\cong \text{Typ} (X,G,\mathbb K) \cong M$$
as desired.


\section{\texorpdfstring{The $\mathrm{C}^*$-realization problem}{The C*-realization problem}}

As stated in \autoref{subsec:RR0}, the realization problem for real rank zero \ca{s} (\autoref{qstIntro:RealRR0}) has deep connections to many other open problems in the field. Currently, the techniques developed for regular rings do not seem to have a direct analytic analogue. At its core, the approach to \autoref{qstIntro:RealRR0} ---perhaps restricted to certain classes of monoids--- should mimic the strategy taken in this paper, that is,
\begin{enumerate}
    \item find combinatorial models for the monoids;
    \item prove that for such models there exists a (possibly not real rank zero) \ca{} associated to the model whose Murray-von Neumann semigroup agrees with the monoid;
    \item carefully perform operations on your \ca{} to make it real rank zero whilst keeping its Murray-von Neumann semigroup.
\end{enumerate}

For (1), \autoref{thm:MondToGraph} does this for super adaptable monoids, giving a separated graph model $(E,C)$. \autoref{prbl:saGroupoid} would do the same with the tight groupoid $\mathcal{G}_{\text{tight}}(S)$ associated to an inverse semigroup $S$ instead. To each of these objects one can naturally associate \ca{s} $C^*(E,C)$ \cite[Definition~1.5]{AraGoo11SepGra} and $C^*(\mathcal{G}_{\text{tight}}(S))$ respectively. Thus, one can ask:

\begin{prbl}[{\cite[Problem~7.6]{AraGoo11SepGra}}]
\label{prbl:Cstarmonoid}    Let $(E,C)$ be a separated graph. Is it true that $V(C^*(E,C))\cong \mathcal{M}(E,C)$?
\end{prbl}

In the case of non-separated graphs, the answer is affirmative, because the natural algebra $\ast$-monomorphism $\iota:L_\mathbb{C}(E)\hookrightarrow C^*(E)$ induces a monoid isomorphism from $V(L_\mathbb{C}(E))$ onto $V(C^*(E))$; see \cite[Theorem 7.1]{AraMorPar07NonStab}. Unfortunately, the arguments used there do not apply for general separated graphs.

\begin{prbl}[{\cite[Remark~7.7]{AraBosParSims21IMRN}}]
 \label{prbl:Cstar-groupoid-monoid}   Let $M$ be a finitely generated conical refinement monoid, and let $C^*(\mathcal{G})$ be the \ca{} of the groupoid from \cite[Corollary~7.6]{AraBosParSims21IMRN} (which is amenable). Is it true that $M\cong V(C^* (\mathcal{G}))$? The same can be asked for general super adaptable monoids if the answer to \autoref{prbl:saGroupoid} is positive. 
\end{prbl}

In fact, the answer could be related to Problems \ref{prbl:saGroupoid} and \ref{prbl:RealGroupoid} if we would be able to characterize for which ample groupoids $\mathcal{G}$ the natural algebra $\ast$-homomorphism $\iota: A_\mathbb{C}(\mathcal{G}) \rightarrow C^*(\mathcal{G})$ induces a monoid isomorphism from $V(A_\mathbb{C}(\mathcal{G}))$ onto $V(C^*(\mathcal{G}))$.

We now observe that we cannot expect a general affirmative answer to the question of whether the natural $*$-homomorphism from a complex $*$-algebra $A$ to its enveloping $C^*$-algebra $B$ induces a monoid isomorphism $V(A)\cong V(B)$, even when that $*$-homomorphism is injective (which is the case for the $*$-algebras from the above paragraph). 

An example showing the kind of obstruction that could appear is the class of irrational rotation algebras \cite{DavidsonBook1996}. Given any irrational number $\theta\in [0,1)$, the irrational rotation algebra $A_\theta$ is the universal $C^*$-algebra generated by two unitaries $u,v$ subject to the relation $vu=e^{2\pi\theta} uv$. This algebra is known to be a simple, nuclear, real rank zero, stable rank one algebra with $K_0(A_\theta)=\mathbb{Z}+\theta^{-1}\mathbb{Z}$ --a rank two dense subgroup of $\mathbb{R}$--, and monoid $V(A_\theta)=(\mathbb{Z}+\theta^{-1}\mathbb{Z})\cap \mathbb{R}^+$. But $A_\theta$ contains a dense complex $\ast$-subalgebra $R_\theta$ --known as the McConnell-Pettit algebra \cite{McConnellPettit1988}-- which is a simple, hereditary noetherian domain of stable rank two with $K_0(R_\theta)=\mathbb{Z}$ whose monoid is non-cancellative. None of them is a separated graph algebra, and while there are groupoid models of the irrational rotation algebra, the McConell-Pettit algebra cannot be the Steinberg algebra of any groupoid, as it is a domain which is not a group algebra, since it does not admit an augmentation homomorphism.

An example that might be relevant for \autoref{prbl:Cstar-groupoid-monoid}
is as follows. Let $A_n=\mathbb C [t_1^\pm,\dots ,t_n^\pm]$ be the algebra of Laurent polynomials in commuting variables $t_1,\dots ,t_n$. Then $V (A_n)=\ZZ^+$ and all finitely generated projective 
$A_n$-modules are free, by \cite[Corollary V.4.10]{Lam06Serre}. However, the algebra $B_n=C(\mathbb T^n)$ of continuous functions on the $n$-torus $\mathbb T^n$ is a completion of $A_n$, and $K_0(B_n)= \ZZ^{2^{n-1}}$. Hence, for $n>1$, the natural map
$V(A_n)\to V(B_n)$ cannot be an isomorphism. Note that $A_n=\mathbb C [\ZZ^n]$ is a Steinberg algebra and $B_n=C^*(\ZZ^n)$ is a group(oid) $C^*$-algebra.

Whilst we expect the answer to \autoref{prbl:Cstarmonoid} to be positive, the answer to \autoref{prbl:Cstar-groupoid-monoid}, in the ligth of the above example, is likely to be negative.

To do step (3), one would need to develop a theory resembling some form of universal localization\footnote{Several steps in such a theory would necessarily need to be different, as one cannot simply invert elements. More concretely, given a unital \ca{} $B$ and a unital subalgebra $A\subseteq B$, an element $a\in A$ is invertible in $A$ whenever it is in $B$.}. Recall that, in our case, we consider a universal localization of $\mathcal S_K (E,C)$ to obtain the regular algebra $Q_K (E,C)$ (\autoref{dfn:regular-algebra}). In particular, this process provides an embedding $\mathcal S_K (E,C)\to Q_K (E,C)$ which, conjecturally, is an isomorphism at the level of the Murray-von Neumann semigroup. In light of this, one can ask:

\begin{prbl}
\label{prbl:Cstar-RR0-embedding}
    Let $A$ be $C^*(E,C)$ for some separated graph $(E,C)$. When does there exist a real rank zero \ca{} $B$ and a $^*$-homomorphism $\iota\colon A\to B$ such that $V(\iota)\colon V(A)\to V(B)$ is an isomorphism? The same can be asked about the groupoid \ca{} in the previous questions.
\end{prbl}

\providecommand{\bysame}{\leavevmode\hbox to3em{\hrulefill}\thinspace}
\providecommand{\MR}{\relax\ifhmode\unskip\space\fi MR }
\providecommand{\MRhref}[2]{%
  \href{http://www.ams.org/mathscinet-getitem?mr=#1}{#2}
}
\providecommand{\href}[2]{#2}

\printindex[terms]
\printindex[symbols]

\end{document}